\documentclass[a4paper, final, reqno]{amsart}

\usepackage[x11names]{xcolor}

\usepackage[fontsize=10pt]{scrextend}

\usepackage{etoolbox}

\patchcmd{\section}{\scshape}{\scshape\bfseries}{}{}
\makeatletter
\renewcommand{\@secnumfont}{\scshape\bfseries}
\makeatother

\patchcmd{\part}{\bfseries}{\scshape\bfseries}{}{}

\DeclareRobustCommand{\SkipTocEntry}[5]{}

\usepackage[margin=2.5cm]{geometry} 

\usepackage{amsmath, bm, mathtools}
\usepackage{amsthm}
\usepackage{amssymb}

\numberwithin{equation}{section}

\usepackage[only,llbracket,rrbracket]{stmaryrd}

\usepackage[makeroom]{cancel}

\usepackage{enumitem}

\usepackage{tikz-cd}

\usepackage[final]{hyperref}
\hypersetup{
	unicode,
	colorlinks=true,
	linkcolor=blue,
	citecolor=Green4, 
	filecolor=magenta,      
	urlcolor=blue,
	pdfauthor={Author One, Author Two, Author Three}, 
	pdfkeywords={article, template, simple},
	pdfproducer={LaTeX}
}

\theoremstyle{plain}
\newtheorem{theorem}{Theorem}[section]
\newtheorem{corollary}[theorem]{Corollary}
\newtheorem{lemma}[theorem]{Lemma}

\newtheorem{proposition}[theorem]{Proposition}

\theoremstyle{definition}

\newtheorem{remark}[theorem]{Remark}

\usepackage{graphicx, color}
\graphicspath{{fig/}}

\usepackage{mathrsfs} 
\usepackage{xfrac}

\usepackage{lipsum}

\usepackage{scalerel}
\newlength\bshft
\def\fakebold#1{\ThisStyle{\ooalign{$\SavedStyle#1$\cr%
			\kern-\bshft$\SavedStyle#1$\cr%
			\kern\bshft$\SavedStyle#1$}}}

\title{On a visco-elastic Mullins-Sekerka System}

\author{Helmut Abels} 
\address{Fakultät für Mathematik, Universität Regensburg, 93040 Regensburg, Germany}
\email{\href{mailto:helmut.abels@ur.de}{helmut.abels@ur.de}}

\author{Harald Garcke} 
\address{Fakultät für Mathematik, Universität Regensburg, 93040 Regensburg, Germany}
\email{\href{mailto:harald.garcke@ur.de}{harald.garcke@ur.de}}

\author{Jonas Haselböck} 
\address{Fakultät für Mathematik, Universität Regensburg, 93040 Regensburg, Germany}
\email{\href{mailto:jonas.haselboeck@ur.de}{jonas.haselboeck@ur.de}}

\date{}

\newcommand\numberthis{\addtocounter{equation}{1}\tag{\theequation}}
\newcommand{\norm}[1]{ \left\|   #1 \right\| }

\newcommand{\abs}[1]{|#1|}

\newcommand{\N}{\mathbb{N}}

\newcommand{\vphi}{\varphi}
\newcommand{\eps}{\varepsilon}

\newcommand{\Tau}{\mathcal{T}}
\newcommand{\bnu}{\bm{\nu}}
\newcommand{\pt}{\partial_t}

\newcommand{\Wp}{W_{,\vphi}}
\newcommand{\E}{\mathcal{E}}
\newcommand{\A}{\mathcal{A}}

\newcommand{\WE}{W_{,\E}}
\newcommand{\F}{\mathcal{F}}

\newcommand{\R}{\mathbb{R}}
\newcommand{\bu}{\bm{u}}

\newcommand{\dx}{\, d\mathbf{x}}
\newcommand{\dy}{\, d\mathbf{y}}

\newcommand{\dt}{\, dt}

\newcommand{\ddt}{\frac{d}{dt}}

\newcommand{\dH}{\, d\mathcal{H}^{d-1}}

\newcommand{\C}{\fakebold{\mathbb{C}}}

\DeclareMathOperator\supp{supp}

\newcommand{\states}{\mathcal{M}}
\newcommand{\indic}{\mathcal{M}_{m_0}}
\newcommand{\disp}{\bm{X}}

\newcommand{\pO}{\partial \Omega}
\newcommand{\po}{\partial \Omega}
\renewcommand{\S}{ \mathbb{S}^{d-1}}

\newcommand{\jump}[1]{ \llbracket #1\rrbracket}
\newcommand{\bjump}[1]{ \big\llbracket #1\big\rrbracket}
\newcommand{\Bjump}[1]{ \Big\llbracket #1\Big\rrbracket}

\renewcommand{\div}{\textnormal{div}}

\def\Xint#1{\mathchoice
	{\XXint\displaystyle\textstyle{#1}}%
	{\XXint\textstyle\scriptstyle{#1}}%
	{\XXint\scriptstyle\scriptscriptstyle{#1}}%
	{\XXint\scriptscriptstyle\scriptscriptstyle{#1}}%
	\!\int}
\def\XXint#1#2#3{{\setbox0=\hbox{$#1{#2#3}{\int}$ }
		\vcenter{\hbox{$#2#3$ }}\kern-.6\wd0}}

\def\dashint{\Xint-}

\begin{document}
	
	\begin{abstract}
		We introduce a novel visco-elastic Mullins--Sekerka system with a prescribed constant contact angle at the boundary. The system is derived as an $H^{-1}$-$\bm{H}^1$-type gradient flow of an energy consisting of the perimeter together with capillary, elastic, and second-gradient contributions. Building on the framework of Hensel and Stinson (Arch.\ Ration.\ Mech.\ Anal.\ 248, 2024), we introduce a measure-valued solution concept featuring a sharp De Giorgi-type energy-dissipation inequality. Moreover, we establish existence of solutions via an implicit time discretization scheme, and prove existence of $BV$ solutions under an energy-conservation hypothesis.
	\end{abstract}
	
\maketitle	
\setcounter{tocdepth}{1}

\textbf{AMS-Classification:} 49Q20, 35D30, 53E10, 35R35, 74B10


\section{Introduction} \label{sec:derivation}

The evolution of interfaces in multiphase systems has long been a central topic in materials science, fluid mechanics and the analysis of nonlinear partial differential equations, with many real world applications ranging from solidification in binary alloys, and two-phase flows, to crystal growth. In recent years, mathematical biology has discovered these models as effective tools to describe tumor growth, prompting a rapidly growing number of publications. Among these models, the Cahn--Hilliard--Biot system, introduced by Strovik et al.\ \cite{STORVIK2022107799}, was the first to include both mechanical stress and fluid flow in a heterogeneous, saturated porous
medium. The addition of source terms to the underlying gradient flow structure further allows incorporating tumor growth (or even growth inhibiting effects) to an otherwise mass-conserving system. In subsequent analytical studies \cite{riethmüller2023wellposedness, MR4819610, haselboeck2024existence, haselboeck2026local}, the model was further augmented to also include visco-elasticity of Kelvin--Voigt type; for numerical results, see \cite{storvik2024sequential, MR4819610, brunk2024structurepreservingapproximationcahnhilliardbiot, riethmüller2026convergenceefficiencysplittingschemes}. The fact that the Cahn--Hilliard--Biot system is a diffuse interface model naturally raises the question of its sharp interface limit. While there are some works that have studied the sharp interface limit of the Cahn--Larché model \cite{garcke_00, garcke_2003, kwak} (which combines Cahn--Hilliard and linear elasticity \cite{larcht1982effect}), there are, to the best of our knowledge, no results on a visco-elastic extension.\\ 
As for sharp-interface models, Mullins and Sekerka introduced the system now bearing their names in the early 60s \cite{mullins1963morphological} in the context of material science and the stability properties of coarsening behavior in solidification processes. The fact that the corresponding equations can be derived as a gradient flow of the perimeter, highlighting its mass preserving, curvature driven evolution, became increasingly important for analytic treatments. In the 90s, Luckhaus and Sturzenhecker \cite{sturzenhecker} introduced a new concept for weak solutions utilizing $BV$-functions, where the interface corresponds to the common boundary of two sets, which are associated with a binary characteristic function $\chi$. Leveraging the $H^{-1}$-type gradient flow structure and a minimizing movement scheme, they were able to construct functions that satisfy both the evolution equation and the Gibbs-Thomson law in a distributional form. Their proof, however, crucially relies on the so-called \textit{energy conservation hypothesis}, meaning that they need to assume continuity of the perimeter functional with respect to weak-{*} convergence in $BV$. Relying on results from geometric measure theory on the properties of varifolds due to Schätzle \cite{schaetzle_01}, and using a generalized notion of mean curvature intrinsic to the interface that allows for a pointwise interpretation of the Gibbs--Thomson law \cite{roeger_interfaces_free}, Röger was later able to remove the energy conservation requirement \cite{roeger, abels_roeger_09}. \\ 
Further results on weak solutions include \cite {Bronsard_Garcke, laux_chambolle}, whereas analytic studies pertaining to the existence of classical solutions can be found in \cite{bogdan24, PrussSimonett2016, MR1421209, MR1607952, MR4322615}. Moreover, an elastic extension of this system was studied in \cite{MR3434747}. \\ 
Recently, Hensel and Stinson \cite{hensel_stinson} proposed a novel solution concept for the Mullins--Sekerka system, once again relying on the space gradient flow structure, that allows for measure valued solutions (varifolds) under certain compatibility and regularity requirements. Additionally, they are able to include intersections of the interface with the boundary and impose arbitrary (but constant) contact angles $\alpha \in (0, \tfrac{\pi}{2}]$ in the weak formulation. Another integral part of their concept is the characterization of the evolution through a single sharp energy dissipation inequality à la De Giorgi, which is only natural in light of the gradient flow structure, but also an important feature considering its crucial significance in a  recent proof \cite{fischer2024weakstronguniquenessprinciplemullinssekerka} of weak-strong uniqueness. \\ 
As for the connection to diffuse-interface models, Alikakos, Bates and Chen \cite{alikakos1994convergence} first used rigorous matched asymptotic expansions on the solutions to the Cahn--Hilliard system to find the singular limit. 
Chen \cite{chen1996global} was able to identify the sharp-interface limit of the Cahn--Hilliard equation as the Mullins--Sekerka system using the fairly weak concept of general varifolds. His approach was later replicated in similar situations, see e.g.\ \cite{abels_roeger_09}, but most importantly for our purposes, it was applied in \cite{kwak} to the Cahn--Larché system. Moreover, we refer to \cite{garcke_00} for results on the $\Gamma$-limit.\\ 
A recent study of the sharp interface limit of the purely elastic Cahn--Hilliard--Biot system using formal matched asymptotic expansions can be found in \cite{MR4873132}.
\par 
\medskip 
Since a thorough understanding of the sharp interface limit of the visco-elastic Cahn--Larché system is a prerequisite for the analysis of more involved models such as the Cahn--Hilliard--Biot equations, we restrict ourselves to this case in the following.\\
Before examining the strong formulation of our proposed system in detail, let us introduce the problem in precise mathematical terms and fix some notation. \\
Suppose that $\Omega \subset \R^d$, $d = 2, 3$, is a domain with smooth boundary $\partial \Omega$, enclosing a mixture of two visco-elastic solids (with distinct properties) that evolve over time. 
We can associate these two phases with the subdomains $\Omega^+(t)$ and $\Omega^-(t)$, respectively, denoting by $\Gamma(t)$ the interface in between them. If $\Gamma(t)$ intersects the boundary $\partial \Omega$, we require the angle between these hypersurfaces to equal $\alpha > 0$.\\ 
Observe that to fully capture the evolution of $\Gamma(t)$, it is sufficient to only consider a single phase, say $\Omega^+(t)$, since $\Omega^-(t)$ is already uniquely determined through $\Omega \setminus \overline{\Omega^+(t)}$. Hence, let us consider a finite time horizon $T_* > 0$ and a family $\mathcal{A} = (\mathcal{A}(t))_{t \in [0, T_*]}$ of evolving open sets in $\Omega$. Then, the interface between the two solids corresponds to the part of $\partial \mathcal{A}$ that is in $\Omega$, i.e, $\Gamma(t) = \partial \mathcal{A}(t) \cap \Omega$. Moreover, we denote by  $\mathcal{V} = \mathcal{V}(t, \bm{x})$, $\varkappa = \varkappa(t, \bm{x})$, $(t, \bm{x}) \in [0, T_*] \times \Gamma(t)$, the normal velocity and mean curvature of $\Gamma(t)$. Lastly, let $\bnu = \bnu(t, \bm{x})$ be the interior normal vector on $\Gamma(t)$, pointing towards $\A(t)$, and $\bm{\tau}_{\Gamma}$, $\bm{\tau}_{\po}$ the conormal vectors of $\partial \mathcal{A}$ with respect to $\Gamma(t)$ and $\po$, respectively (for more details see Section \ref{sec:gradient_flow}). \\ 
For a function $\bm{v} : \Omega \to \R^n$, we define the jump $\jump{\bm{v}}$ across $\Gamma(t)$ as the difference between the traces from both sides, i.e, 
\begin{equation*}
	\jump{\bm{v}} \coloneqq \textnormal{Tr}_{\A(t)} (\bm{v}) - \textnormal{Tr}_{\Omega \setminus \A(t)} (\bm{v}). 
\end{equation*}
\\ 
As usual, we describe the elastic deformations in the material by a vector field $\bm{u} : [0, T_*] \times \Omega \to \R^d$ and typically assume that the stress is given by $\C^\pm \big( \E(\bm{u}) -  \Tau^\pm \big)$, where $\C^\pm, \Tau^\pm$ are the constant elasticity tensor and eigenstress, respectively, associated with the material occupying $\Omega^\pm$.\\ 
Finally, let us associate $(\A(t))_{t\in[0, T_*]}$ with a family of characteristic functions $\chi \colon [0, T_*] \times \Omega \to \{0, 1\}$ such that $\chi(t) \equiv 1$ in $\A(t)$ and zero otherwise. Then, we may assume $\C, \C_{\nu}, \Tau$ to be functions of $\chi$ such that $\C^+ = \C(1)$, $\C^- = \C(0)$, with analogous conditions for the remaining maps, where $\C_\nu$ is the tensor of visco-elasticity. As usual, we assume that $\C$ (respectively $\C_\nu$) defines an endomorphism of $\R^{d \times d}_{sym}$, and is uniformly positive on symmetric matrices, i.e., for all $y \in \R$, $\E' \in \R^{d \times d}_{sym}$, it holds that 
\begin{align*}
	\C(y) \E' \in \R^{d \times d}_{sym}, \quad \textnormal{and} \quad \C(y) \E' \colon \E' \geq c_{\C} \norm{\E'}^2
\end{align*}
for some $c_{\C} > 0$. 
The elastic energy density is then given by $W(\chi, \E(\bm{u})) + \tfrac{1}{2} \abs{\nabla^2 \bm{u}}^2$, cf.\ \eqref{eq:energy_ms_regularized}, where, as it is typical for linear elasticity, we suppose $W$ to be a function of $\chi$ and $\E(\bm{u})$ with 
\begin{equation*}
	W(\chi, \E') \coloneqq \frac{1}{2} \C(\chi) \big( \E' - \Tau(\chi) \big) \colon \big( \E' - \Tau(\chi)\big). 
\end{equation*}\\ 
In the process of deriving a visco-elastic Mullins--Sekerka system, both as a gradient flow and also using formally matched asymptotics, we encountered difficulties concerning the jump condition of $\pt \bm{u}$, and related to that also of $\jump{\nabla \bm{u}}\bnu$. It turns out that this is due to a structural problem in the dynamics. 
Indeed, in order to prevent cracks or self intersections of the material, we have to assume that the displacement is continuous across the interface, i.e., 
\begin{equation*}
	\bm{0} = \jump{\bm{u}(t, \bm{x})} \quad \textnormal{for all} \quad \bm{x} \in \Gamma(t).
\end{equation*}
Suppose that $\Gamma(t)$ is an evolving interface, and consider a curve $ t \mapsto \bm{x}(t) \in \Gamma(t)$ such that $\tfrac{d}{dt} \bm{x}(t)$ is normal to $\Gamma(t)$. Then, we can examine the time derivative of $\jump{\bm{u}(t, \bm{x}(t))}$ along this trajectory at a fixed time $t_0$: 
\begin{equation}\label{eq:time_deriative_interface}
	\bm{0} = \ddt_{| t = t_0}  \jump{\bm{u}(t, \bm{x}(t))} 
	= \ddt_{| t = t_0} \textnormal{Tr}_{\Omega^+(t)} \bm{u}(t, \bm{x}(t)) - \ddt_{| t = t_0} \textnormal{Tr}_{\Omega^-(t)} \bm{u}(t, \bm{x}(t)). 
\end{equation}
If the traces $\bm{u}_\pm (t, \bm{x}(t)) \coloneq \textnormal{Tr}_{\Omega^\pm(t)} \bm{u}(t, \bm{x}(t))$ are sufficiently smooth to allow differentiation, the chain rule gives 
\begin{equation*}
	\ddt_{| t = t_0} \bm{u}_\pm(t, \bm{x}(t)) = \big( \pt \bm{u}_\pm   + \nabla \bm{u}_\pm \bnu \mathcal{V} \big) (t_0, \bm{x}(t_0)), 
\end{equation*} 
where we used that $\ddt_{| t = t_0} \bm{x}(t) = \mathcal{V} \bnu(t_0, \bm{x}(t_0))$. Substituting this back into \eqref{eq:time_deriative_interface} yields the compatibility condition
\begin{equation}\label{eq:compatibility}
	\jump{\pt \bm{u}} = - \jump{\nabla \bm{u}} \bnu \mathcal{V}  \quad \textnormal{on} \quad \Gamma(t). 
\end{equation}
In particular, since for $\bm{u} \in \bm{H}^1(\Omega)$ the jump $\jump{\nabla \bm{u}}$ does not vanish in general, the best regularity we can expect for the time derivative is $\pt \bm{u} \in \bm{H}^1(\Omega \setminus\Gamma(t))$, which is  detrimental for compactness arguments in an approximation scheme. Moreover, it is not entirely clear whether this is a condition on $\pt \bm{u}$ or rather for $\bm{u}$ itself, which might lead to an overdetermined system in the strong formulation. However, we observe that if $\bm{u} \in \bm{H}^2(\Omega)$ and $\pt \bm{u} \in \bm{H}^1(\Omega)$ the compatibility condition is automatically satisfied.\par 
Furthermore, let us examine what happens in the sharp interface limit of the visco-elastic Cahn--Larché system. In this case, and for a fixed $\eps > 0$, the chemical potential satisfies
\begin{equation}\label{eq:mu_ve_CL}
	\mu_\eps =-\eps \Delta \vphi_\eps + \frac{1}{\eps} \psi'(\vphi_\eps) + W_{,\vphi_\eps} (\vphi_\eps, \E(\bu_\eps)). 
\end{equation}
Recall that in \cite{luckhaus_modica} Luckhaus and Modica studied the stationary Cahn--Hilliard equation and were able to show that for $\eps \to 0$ the chemical potential $\mu_\eps$ converges to a quantity proportional to the mean curvature. In their proof they test the equation with $\nabla \vphi_\eps \cdot \bm{B}$, where $\bm{B} \in \bm{C}^\infty_c (\Omega)$, and use the chain rule to obtain
\begin{equation*}
	\int_\Omega \psi'(\vphi_\eps) \nabla \vphi \cdot  \bm{B} \dx  =  \int_\Omega \nabla \psi(\vphi_\eps) \cdot \bm{B} \dx.
\end{equation*}
After integration by parts, the derivative no longer falls on the order parameter $\vphi_\eps$, which becomes discontinuous as $\eps \to 0$, and one can pass to the limit in this term. \\ 
Using the same idea on \eqref{eq:mu_ve_CL}, we obtain the term 
\begin{equation*}
	\Wp(\vphi_\eps, \E(\bm{u}_\eps)) \nabla \vphi_\eps \cdot \bm{B},
\end{equation*}
which we wish to rewrite similarly. By chain rule it holds that 
\begin{equation*}
	\Wp(\vphi_\eps, \E(\bm{u}_\eps)) \nabla \vphi_\eps  = \nabla W(\vphi_\eps, \E(\bm{u}_\eps)) - \WE(\vphi_\eps, \E(\bm{u}_\eps)) \nabla \E(\bm{u}_\eps).
\end{equation*}
Along with the identity $- \nabla \cdot \big( \C_\nu(\vphi_\eps) \E(\pt \bm{u}_\eps) +\WE(\vphi_\eps, \E(\bm{u}_\eps))\big) = 0$, integration by parts yields 
\begin{align*}
	\int_\Omega& \Wp(\vphi_\eps, \E(\bm{u}_\eps)) \nabla \vphi_\eps \cdot \bm{B} \dx \\
	= &- \int_\Omega 
	\Big( W(\vphi_\eps, \E(\bm{u}_\eps)) \bm{I} - (\nabla \bm{u}_\eps)^T \big(\C_{\nu}(\vphi_\eps)  \E( \partial_t \bm{u}_\eps) +  \WE(\vphi_\eps, \E(\bm{u}_\eps)) \big)  \Big)  \colon \nabla \bm{B}  \dx \\ 
	&+  \int_\Omega \C_{\nu}(\vphi_\eps)  \E( \partial_t \bm{u}_\eps) \colon \nabla^2 \bm{u}_\eps \bm{B}  \dx , 	
\end{align*}
which no longer has any derivative on the order parameter $\vphi_\eps$. This result, however, suggests that the weak formulation of the sharp interface limit of the visco-elastic Cahn--Larché features the term $\nabla^2 \bm{u}$, which we do not control in any norm.\\ 
For these reasons, we decided to include the regularizing, second-order term $\abs{\nabla^2 \bm{u}}^2$, often referred to as hyperstress  in the context of non-simple materials, in the energy functional. Note that this is an established approach in nonlinear elasticity, cf.\ \cite{mielke_roub}. \\

\addtocontents{toc}{\SkipTocEntry}
\subsection*{Strong PDE formulation} We say that the family $\A(t)$ of open sets evolves by the regularized visco-elastic Mullins--Sekerka flow, if for each $t \in (0, T_*)$ there exists a chemical potential $w(t, \cdot)$ and a displacement $\bm{u}(t, \cdot)$, such that: 

\textbf{\\  In the bulk:}
\begin{subequations}\label{eq:sg_sharp_bulk_ve}
	\begin{alignat}{2}
		\Delta w &= 0   \quad & \textnormal{in } \Omega \setminus (\partial \mathcal{A}(t) \cap \Omega), \label{id:sg_mu}\\ 
		- \nabla \cdot \Big(\C_{\nu}(\chi) \E(\pt \bm{u})   + \WE (\chi,  \E(\bm{u})) - \div\, \nabla^2 \bm{u}\Big) & = \bm{0} \quad & \textnormal{in } \Omega \setminus (\partial \mathcal{A}(t) \cap \Omega). \label{id:sg_visco_elastic}
	\end{alignat}
\end{subequations}

\noindent
\textbf{On the interface:}
\begin{subequations}\label{eq:sg_sharp_jnterface}
	On $\Gamma(t)$, $t \in (0, T)$, the chemical potential satisfies the relation
	\begin{alignat*}{2}
		w =
		c_0 \varkappa  
		+ \bjump{\tfrac{1}{2} \abs{\nabla^2 \bm{u}}^2 + W(\chi, \E(\bm{u}))  }. \numberthis\label{id:sg_mu_interface}
	\end{alignat*}
	Moreover, we have the following identities on $\Gamma(t), \ t \in (0, T)$,
	\begin{alignat}{2}
		- \jump{\nabla w} \cdot \bnu &=   \mathcal{V} , \label{id:sg_mu_jump_ve}\\ 
		\jump{w} & = 0 , \\ 
		\jump{\bm{u}} & = \bm{0} , \label{id:sg_u} \\ 
		\jump{\nabla \bm{u} }\bnu  & = \bm{0} ,   \label{id:sg_nabla_u}\\ 
		\bjump{ \big( \C_{\nu} (\chi) \E(\pt \bm{u})   + \WE (\chi, \E(\bm{u})) - \div_{\Gamma}\, \nabla^2 \bm{u} \big)} \bnu - \bjump{ \div_{\Gamma}(\nabla^2 \bm{u} \bnu) } &= \bm{0} , \label{id:sg_normal_stress_ve} \\ 
		\bjump{  \nabla^2 \bm{u}}   (\bnu \otimes \bnu) &= \bm{0} .  \label{id:sq_second_normal_}
	\end{alignat}
	By the argument preceding \eqref{eq:compatibility}, it immediately follows from \eqref{id:sg_u} and \eqref{id:sg_nabla_u} that
	\begin{equation}
		\jump{\pt \bm{u}}  = \bm{0} \label{id:sg_jump_u_t_ve}. 
	\end{equation}
\end{subequations}

\begin{subequations}\label{eq:sharp_boundary_sg} 
	
	\noindent
	\textbf{On the boundary:} For all times $t \in (0, T)$, we impose the boundary conditions
	\begin{alignat}{2}
		\nabla w  \cdot \bm{n}_{\po}&= 0  \quad &&\textnormal{on }	\po \setminus \overline{\partial \mathcal{A}(t) \cap \Omega}, \\ 
		\big[ \C_{\nu} (\chi) \E(\pt \bm{u})   + \WE (\chi, \E(\bm{u})) - \div_{\po}\, \nabla^2 \bm{u} \big] \bm{n}_{\po} - \div_{\po}(\nabla^2 \bm{u} \bm{n}_{\po}) &= \bm{0}  \quad &&\textnormal{on } \po \setminus \overline{\partial \mathcal{A}(t) \cap \Omega} ,\\
		\nabla^2 \bm{u}  (\bm{n}_{\po} \otimes \bm{n}_{\po}) &=\bm{0}  \quad &&\textnormal{on } \partial \Omega. 
	\end{alignat}	
	To account for the contact angles, we additionally need 
	\begin{alignat}{2}
		\bnu \cdot  \bm{n}_{\partial\Omega} &= \cos \alpha \quad &&\textnormal{on } \partial \Omega \cap \overline{\partial \A(t) \cap \Omega }, \\ 
		\bjump{ \nabla^2 \bm{u}  \bm{n}_{\po}  }\bm{\tau}_{\po} 
		+ \bjump{ \nabla^2 \bm{u} \bnu }\bm{\tau}_\Gamma
		& = \bm{0}\quad &&\textnormal{on } \partial \Omega \cap \overline{\partial \A(t) \cap \Omega }. 
	\end{alignat}
\end{subequations}
\textbf{\\  Initial conditions:} To complete the system, it only remains to impose initial conditions.
\begin{subequations}\label{eq:sg_initial} 
	\begin{alignat}{2}
		\chi(0, \cdot) &= \chi_0 = \chi_{\A(0)} \quad &&\textnormal{in } \Omega, \\ 
		\bm{u}(0, \cdot) &= \bm{u}_0 \quad &&\textnormal{in } \Omega.
	\end{alignat}
\end{subequations}

\begin{remark}
	While the subsequent derivation is based on a gradient flow approach that replicates the dissipative structure of the diffuse interface setting, we expect this model to also arise as the sharp interface limit of a suitably modified visco-elastic Cahn--Larché system augmented with an analogous second-order regularization.  
\end{remark}

\addtocontents{toc}{\SkipTocEntry}
\subsection*{Properties of the system} First of all, we emphasize that the system is mass conserving. Indeed it holds that 
\begin{equation*}
	\ddt \int_\Omega \chi(t) \dx 
	=  - \int_{\Gamma(t)} \mathcal{V}(t) \dH 
	= \int_{\Gamma(t)} \jump{\nabla w} \cdot \bnu \dH 
	= - \int_\Omega  \Delta w \dx
	= 0. 
\end{equation*}
Moreover, as we will show in detail in Section \ref{sec:gradient_flow}, this model arises as an $H^{-1}$-$\bm{H}^1$ gradient flow of the energy 
\begin{align*}
	\F_{per}&+ \F_{cap}+ \F_{el} + \F_{\mathfrak{h}}\\
	&\coloneqq
	\int_{\Gamma(t)} c_0 \dH 
	+ \int_{ \partial \Omega} c_0 (\cos \alpha) \chi \dH
	+ \int_\Omega W(\chi, \E(\bm{u})) \dx + \int_\Omega \tfrac{1}{2} \abs{\nabla^2 \bm{u}}^2 \dx. 
	\numberthis \label{eq:energy_ms_regularized}
\end{align*}
The elastic energy density $W$, and the hyperstress energy $\F_{\mathfrak{h}}$ have already been discussed above. Moreover, the term $\F_{per}$ measures the perimeter of the open sets $\A(t)$ in $\Omega$, where $c_0 >0$ is a fixed surface tension constant. Lastly, we also fix two additional surface tension constants $\gamma_+, \gamma_-$ satisfying Young's relation $\abs{ \gamma_+ - \gamma_-} < c_0$, and deduce that there must exist an angle $\alpha \in (0, \pi)$ such that 
\begin{equation*}
	(\cos \alpha) c_0 = \gamma_+ - \gamma_-. 
\end{equation*}
Without loss of generality, we may assume that $\gamma_- \leq \gamma_+$; otherwise we could simply switch the roles of $\Omega^\pm$, leading to the restriction that $\alpha \in (0, \tfrac{\pi}{2}]$. With these assumptions at hand, we can consider a capillary contribution of the form,
\begin{equation*}
	\int_{\po} \gamma_+ \chi \dH + \int_{\po} \gamma_- (1 - \chi) \dH
	= \int_{\po} c_0 (\cos \alpha) \chi \dH + \int_{\po} \gamma_- \dH, 
\end{equation*}
where, in a slight abuse of notation, $\chi$ is actually the trace $\textnormal{Tr}_{\po} \chi$. Since the second integral is constant, neglecting it leads to an energetically equivalent formulation. \\
Moreover, a straightforward computation yields the following formal energy identity  
\begin{align*}
	\ddt \Big(& \F_{per}+ \F_{cap}+ \F_{el} + \F_{\mathfrak{h}} \Big) \\ 
	&=\begin{aligned}[t]
		&- \int_{\Gamma(t)} c_0 \varkappa \mathcal{V} \dH \\
		&- \int_{\Gamma(t)}\bjump{\tfrac{1}{2} \abs{\nabla^2 \bm{u}}^2 + W(\chi, \E(\bm{u}))  } \mathcal{V} \dH 
		- \int_\Omega \C_\nu (\chi) \E(\pt \bm{u}) \colon \E(\pt \bm{u}) \dx 
	\end{aligned}\\
	&=-  \int_\Omega \abs{\nabla w}^2\dx - \int_\Omega  \C_\nu (\chi) \E(\pt \bm{u}) \colon \E(\pt \bm{u}) \dx  
	\leq 0, 
\end{align*}
where we made use of the identities \eqref{eq:sg_sharp_bulk_ve}-\eqref{eq:sharp_boundary_sg}. For more detailed computations we refer to the derivation of the system as a gradient flow in Section~\ref{sec:gradient_flow}. We note that under the assumption that $\C_{\nu}$ is a uniformly positive definite fourth order tensor on the symmetric matrices, we can define an equivalent norm on the space $\bm{H}^1_\perp(\Omega)$ -- see Section \ref{sec:prelimiries} for a precise definition -- by 
\begin{equation*}
	\norm{\bm{z}}_{\bm{H}^1_\chi}^2  \coloneqq \int_\Omega \C(\chi) \E(\bm{z}) \colon \E(\bm{z}) \dx. 
\end{equation*}
In particular, this implies that the proposed system is dissipative, with the dissipation depending on the $\bm{H}^1$-norm of $\pt \bm{u}$.

\addtocontents{toc}{\SkipTocEntry}
\subsection*{Main results} 

The first novelty of our work is the introduction of a (regularized) visco-elastic Mullins--Sekerka system \eqref{eq:sg_sharp_bulk_ve}-\eqref{eq:sg_initial}, which is formally derived in Section~\ref{sec:gradient_flow} as a gradient flow of the energy \eqref{eq:energy_ms_regularized}. This functional was obtained by passing to the $\Gamma$-limit in the free energy of the Cahn--Larché system \cite{garcke_00}, with further augmentation by the regularizing hyperstress $\F_{\mathfrak{h}}$. The kinetics of the flow are prescribed by a $H^{-1}_{(0)}$-type metric with respect to the characteristic functions corresponding to a family of smoothly evolving open sets $\A$, along with a $\bm{H}^1$-type metric, depending on the state $\chi$, with respect to the displacement function $\bm{u}$. \\ 
Secondly, inspired by the work of Hensel and Stinson \cite{hensel_stinson}, which was also adapted in the context of two-phase flows \cite{abels2025weaksolutionssharpinterface}, we propose a weak solution concept that includes a sharp energy dissipation inequality in the spirit of De Giorgi, which is tied to the notion of curves of maximal slope in metric spaces, see e.g.\ \cite{ambrosio}. This is also in line with the concepts used for evolutionary $\Gamma$-convergence developed by Sandier and Serfaty \cite{sandier_serfaty}, as well as Le \cite{le08}. \\ 
Moreover, we show existence of solutions corresponding to our weak notion, and further establish the conditional existence of $BV$-solutions in the sense of Luckhaus and Sturzenhecker \cite{sturzenhecker}.

\begin{theorem}[Varifold solutions]\label{thm:varifold} 
	Let $\Omega \subset \R^d$, $d = 2, 3$, be a bounded, smooth domain. 
	For any given terminal time $T_* > 0$, initial data $\chi_0  \in BV (\Omega; \{0, 1\})$ and $\bm{u}_0 \in \bm{H}^1_{\perp}(\Omega) \cap \bm{H}^2(\Omega)$ there exists a measurable map $\chi \colon (0, T_*) \times \Omega \to \{0,1\}$, together with a family of oriented varifolds $(\mu_t)_{t \in (0, T_*)}$, $t \in (0, T_*)$, and a vector field $\bm{u} \colon (0, T_*) \times \Omega \to \R^d$ that satisfy: 
	\begin{enumerate}[label*=(\arabic*), nosep]
		\item (Regularity of the maps) The characteristic function $\chi$ and the displacement $\bm{u}$ are of the following classes: 
		\begin{align*}
			\chi &\in L^\infty_{w^*}(0, T_*; BV(\Omega; \{0,1\})) \cap H^1(0, T_*; H^{-1}_{(0)}(\Omega)), \\ 
			\bm{u} &\in H^1(0, T_*; \bm{H}^1_\perp(\Omega)) \cap L^\infty(0, T_*; \bm{H}^2(\Omega)). 
		\end{align*} 
		
		\item  (Consistency with initial data) It holds that $\chi(0) = \chi_0$ in $H^{-1}_{(0)}(\Omega)$ and $\bm{u}(0) = \bm{u}_0$ in $\bm{H}^1_\perp(\Omega)$. 
		\item (Structure of the oriented varifolds) For almost all $t \in (0, T_*)$, the oriented varifolds $(\mu_t)_{t\in (0, T_*)}$ can  be decomposed as $\mu_t = c_0 \mu_t^\Omega + c_0(\cos \alpha) \mu_t^{\po}$, where $\mu^\Omega_t$ and $\mu^{\po}_t$ are two separate varifolds that are given in their disintegrated form by 
		\begin{alignat*}{2}
			\mu_t^\Omega &= \abs{\mu_t^\Omega} \otimes (\mu_{t, \bm{x}}^\Omega)_{\bm{x}\in \overline{\Omega}}  &&\in M(\overline{\Omega} \times \S), \\  
			\mu_t^{\po}&= \abs{\mu_t^{\po}} \otimes (\delta_{\bm{n}_{\po} (\bm{x})})_{\bm{x}\in \overline{\po}} &&\in M({\partial \Omega} \times \S), 
		\end{alignat*}
		where $\mu^\Omega_{t, \bm{x}}$ is a probability measure on $\S$ for almost all $t \in (0, T_*)$ and $\abs{\mu^\Omega_t}$-almost all $\bm{x} \in \overline{\Omega}$. 
		Moreover, the measure $\abs{\mu^{\po}_t}$ coincides for almost all $t \in (0, T_*)$  with $g_t \mathcal{H}^{d-1}\llcorner\po$, where the non-negative function $g_t = g(t, \cdot)$ satisfies $g \in L^\infty((0, T_*) \times \po; [0, 1])$. 
		\item (Compatibility of the phase indicator) The varifolds $(\mu_t)_{t \in (0, T_*)}$ and the characteristic functions $\chi(t)$ are compatible in the sense that
		\begin{align}
			\int_{\overline{\Omega} \times \mathbb{S}^{d-1}} \bm{\eta} (\bm{x}) \cdot \bm{s}  \, d\mu^\Omega_t(\bm{x}, \bm{s})
			=
			\int_\Omega \bm{\eta} (\bm{x}) \cdot \, d \nabla \chi(t)
			\label{eq:compatibility_thm_I}
		\end{align}
		for all $ \bm{\eta} \in \bm{C}^1(\overline{\Omega})$ with $\bm{\eta}_{|\po} \cdot \bm{n}_{\po} = 0$, and 
		\begin{align*}		
			-\int_{\overline{\Omega} \times \S}& \bm{s} \cdot \bm{\xi} \, d \mu_t^{\Omega}(\bm{x}, \bm{s}) \\ 
			&= - \int_\Omega  \bm{\xi} \cdot d \nabla \chi(t, \cdot)
			+ (\cos \alpha) \Big(
			\abs{\mu_t^{\po}}(\po) - \int_{\po} \chi(t, \cdot ) \dH 
			\Big) 
			\numberthis \label{eq:compatibility_thm}
		\end{align*}
		for all $\bm{\xi} \in \bm{C}^1(\overline{\Omega})$ with $\bm{\xi}_{|\po} \cdot \bm{n}_{\po} = \cos \alpha$, and almost all $t \in (0, T_*)$. Moreover the trace of the interface on the boundary is contained in the sense that 
		\begin{equation}\label{eq:boundary_measure_inequaltiy}
			(\cos \alpha) \chi(t, \cdot ) \mathcal{H}^{d-1} \llcorner \po \leq \cos \alpha \Big( \abs{\mu_t^{\Omega}} \llcorner \po + \abs{ \mu_t^{\po} }  \Big) \quad \textnormal{for a.e.\ } t \in (0, T_*). 
		\end{equation}
		
		\item (Generalized Gibbs-Thomson law) There exists an associated potential $w \in L^2(0, T_*; H^1(\Omega))$ that satisfies the Gibbs--Thomson law in the weak form 
		\begin{align*}
			&\int_{\overline{\Omega} \times \S} (\bm{I} - \bm{s} \otimes \bm{s}) \colon \nabla \bm{B} \, d \mu_t (\bm{x}, \bm{s})\numberthis \label{eq:potential_thm}\\ 
			& \quad = \int_\Omega \chi \  \div \big( w \bm{B} \big) \dx  
			+ \int_\Omega \C_{\nu}(\chi)  \E(\pt \bm{u})
			\colon \nabla^2 \bm{u} \bm{B}  \dx 
			-\int_\Omega \tfrac{1}{2} \abs{\nabla^2 \bm{u}}^2\  \div \bm{B} \dx \\ 
			&\qquad -  \int_\Omega \Big[ W(\chi, \E(\bm{u})) \bm{I} -  ( \nabla \bm{u} )^T \big( \C_{\nu}(\chi)  \E (\pt \bm{u}) +  W_{, \E} (\chi, \E(\bm{u})  \big)\Big] \colon \nabla \bm{B} \dx 
			\\ 
			&\qquad + \int_\Omega  (\nabla^2 \bm{u})_{ijk} (\nabla \bm{u})_{ip}\,(\nabla^2 \bm{B})_{pjk}  \dx 
			+ \int_\Omega(\nabla^2 \bm{u})_{ijk}  \big( (\nabla^2 \bm{u})_{ipk}\,(\nabla\bm{B})_{jp} 
			+ (\nabla^2 \bm{u})_{ijp}\,(\nabla\bm{B})_{kp} \big) \dx
		\end{align*}
		for  all $\bm{B} \in \bm{C}^2(\Omega)$ with $\bm{B}_{|\partial \Omega} \cdot \bm{n}_{\partial_\Omega} = 0$ and almost all $t \in (0, T_*)$.
		
		\item (Time evolution of phase indicator) There is a potential 
		\begin{equation*}
			v = -(-\Delta_N)^{-1} \pt \chi \in L^2(0, T; H^1_{(0)} (\Omega))
		\end{equation*}
		such that 
		\begin{equation}\label{eq:dt_chi}
			\int_0^{T_*} \int_\Omega {\chi}\,  \pt \zeta \dx \dt
			- \int_0^{T_* }\int_\Omega  \nabla v \cdot \nabla \zeta \dx \dt = - \int_\Omega \chi_0 \zeta(0, \bm{x}) \dx, 
		\end{equation}
		for all  $\zeta \in C^1_c([0, T_*) \times \Omega) \cap H^1(0, T_*; H^1_{(0)})$. Moreover, it holds that $\int_\Omega \chi(t) \dx  = \int_\Omega \chi_0 \dx$ for almost all $t \in (0, T_*)$. 
		
		\item (Elasticity equation) Moreover, the displacement satisfies 
		\begin{align}\label{eq:elastic_thm}
			\int_0^{T_*}\int_\Omega 
			\big[\C_{\nu}(\chi) \E ( \pt \bm{u}) 
			+ \WE ({\chi}, \E({\bm{u}}))
			\big] 
			\colon \E(\bm{v}) 
			+ \nabla^2 {\bm{u}}\colon  \nabla^2 \bm{v} \dx \dt  = 0
		\end{align}
		for all $\bm{v} \in L^2(0, T_*; \bm{H}^2(\Omega))$.  
		\item(Sharp energy dissipation inequality) These functions further satisfy for almost all $0 \leq s \leq T \leq T_*$ the following energy dissipation inequality 
		\begin{align*}\label{eq:energy_dissipation_thm}
			\F(\mu, \chi, \bm{u}) (T)
			+ &\frac{1}{2} \int_s^T 
			\norm{\pt \chi}_{H^{-1}_{(0)}}^2 + \norm{ \pt \bm{u}}^2_{\bm{H}^1_{\chi(t)}} \dt \\ 
			&+  \frac{1}{2} \int_s^T  \norm{\nabla w}_{L^2}^2 
			+  \norm{ \pt \bm{u}}^2_{\bm{H}^1_{\chi(t)}} \dt 
			\leq  \F(\mu, \chi, \bm{u}) (s) \leq \F(\chi_0, \bm{u}_0). 
			\numberthis
		\end{align*}
		Here, we write (in a slight abuse of notation)
		\begin{equation*}
			\F(\mu, \chi, \bm{u}) (t) 
			\coloneqq 
			\int_{\overline{\Omega}} d|\mu_t| 
			+ \int_\Omega W(\chi, \E(\bm{u}))(t) \dx 
			+ \int_\Omega \tfrac{1}{2} | \nabla^2 \bm{u} |^2(t) \dx. 
		\end{equation*}
		
		\item (Measurability  of the total mass measure) The total mass measure associated with the oriented varifolds $(\mu_t)_{t \in (0, T_*)}$, i.e., 
		\begin{equation*}
			\abs{\mu_t}(\overline{\Omega}) = c_0 \abs{\mu_t^\Omega}(\overline{\Omega}) + c_0 (\cos \alpha) \abs{\mu^{\po}_t} (\partial \Omega)
		\end{equation*}
		is a measurable map from $t \in (0, T_*)$ to $[0, \infty)$.  
	\end{enumerate}
\end{theorem}

\medskip 

\begin{corollary}\label{cor:properties_varifold}
	There exists a function $\theta \colon (0, T) \times \Omega \to [0, \infty)$ with $\abs{\theta} \leq 1$ almost everywhere, such that it hold for almost all $t \in (0, T_*)$ 
	\begin{equation}\label{eq:normal_proportional}
		\int_{\mathbb{S}^{d-1}} \bm{s} \, d\mu_{t, \bm{x}} (\bm{s})
		= \left\{
		\begin{alignedat}{2}
			&\theta(t, \bm{x}) \bm{\nu}( t, \bm{x}) \quad && \textnormal{for $\abs{\mu^\Omega_t}$-a.e. }	\bm{x} \in \partial^* \mathcal{A}(t), \\ 
			&0 \quad && \textnormal{for $\abs{\mu^\Omega_t}$-a.e. }	\bm{x} \not\in \partial^* \mathcal{A}(t), \\ 
		\end{alignedat}
		\right.
	\end{equation}
	where $ \bm{\nu}( t, \bm{x}) = \tfrac{\nabla \chi(t, \bm{x})}{\abs{\nabla \chi(t, \bm{x})}}$ denotes the interior normal of the reduced boundary $\partial^* \mathcal{A}(t)$ with $\chi(t) = \chi_{\mathcal{A}(t)}$. Moreover, if $\abs{\mu(t)}(\Omega) = \abs{\nabla \chi(t)}$, then $\mu_{t, \bm{x}} = \delta_{\bnu(t, \bm{x})}$ for $\abs{\mu_t}$-a.e.\ $x \in \Omega$. 
\end{corollary}

\begin{proof}
	Recall that \eqref{eq:compatibility_thm_I} requires  
	\begin{equation*}
		\int_{\overline{\Omega} \times \mathbb{S}^{d-1}}  \bm{\eta} (\bm{x})  \cdot  \bm{s}\, d\mu_t^\Omega(\bm{x}, \bm{s})
		=
		\int_\Omega \eta (\bm{x}) \cdot \, d \nabla \chi(t)
	\end{equation*}
	for all $ \bm{\eta} \in \bm{C}^1(\overline{\Omega})$ with $ \bm{\eta}_{|\po}(t) \cdot \bm{n}_{\po} = 0$ and a.e.\ $t \in (0, T_*)$.
	In particular, this implies for all open sets $A \subset \Omega$ and all $\bm{f} \in \bm{C}^0_c(\Omega)$ with $\supp \bm{f} \subset A$ and $\norm{\bm{f}}_{\infty} \leq 1$ that 
	\begin{align*}
		\int_\Omega \xi (\bm{x}) \cdot \, d \nabla \chi(t) 
		= \int_{\overline{\Omega} \times \mathbb{S}^{d-1}}  \bm{\xi} (\bm{x}) \cdot \bm{s} \, d\mu^\Omega_t(\bm{x}, \bm{s})
		= \int_A \Big( \int_{\mathbb{S}^{d-1}} \bm{\xi} (\bm{x}) \cdot \bm{s} \, d\mu_{t, \bm{x}}(\bm{s}) \Big)   d \abs{\mu_t}(\bm{x})
		\leq 
		\int_A 1    d \abs{\mu_t}(\bm{x}), 
	\end{align*}
	which entails
	\begin{equation*}
		\abs{\nabla \chi(t)}(A) = \sup \Big\{  \int_\Omega \bm{f}(\bm{x}) \cdot d \nabla \chi  \colon \bm{f} \in \bm{C}^0_c(\Omega), \supp \bm{f} \subset A, \norm{\bm{f}}_{\infty} \leq 1 \Big\}
		\leq \abs{\mu_t^\Omega}(A).
	\end{equation*}
	Hence, $\abs{\nabla \chi(t)}$ is absolutely continuous with respect to $\abs{\mu_t^\Omega} \llcorner\Omega$ and there exists some $\abs{\mu_t^\Omega}$-measurable function $\theta_t \colon \Omega \to [0, \infty)$ with $\theta_t \leq 1$ for $\abs{\mu_t}$-a.e.\ $\bm{x} \in \Omega$ such that 
	\begin{equation}\label{eq:measurs_abs_cont}
		\abs{\nabla \chi(t)}(A) = \int_A \theta_t(\bm{x}) \, d\abs{\mu^\Omega_t}(\bm{x}). 
	\end{equation}
	In particular, $\theta_t(\bm{x}) = 0$ for $\abs{\mu^\Omega_t}$-almost all $\bm{x} \in \Omega$ with $\bm{x} \not\in \partial^*\mathcal{A}(t)$. 
	After disintegrating $\mu_t^\Omega$ into the non-negative mass measure $\abs{\mu_t^\Omega}$ and a family of probability measures $(\mu_{t, \bm{x}})_{\bm{x} \in \overline{\Omega}}$, see Theorem \ref{thm:disintegration}, and exploiting the compatibility condition \eqref{eq:compatibility_thm} along with the identity \eqref{eq:measurs_abs_cont}, we obtain the first assertion.
\end{proof}

\begin{remark}
	\begin{enumerate}[label = (\roman*), nosep]
		\item The identity \eqref{eq:measurs_abs_cont} implies $\supp \nabla \chi(t) \subseteq \supp \mu_t^\Omega $ and $\abs{\nabla \chi(t)} (\Omega) \leq \abs{\mu_t^\Omega}(\Omega)$. 
		\item Since $\chi(t) = \chi_{\A(t)}$, it follows that $\A(t)$ is a set of finite perimeter for a.e.\ $t \in (0, T_*)$. Hence,  \eqref{eq:normal_proportional} asserts that the expectation of $\mu_{t, \bm{x}}^\Omega$ is proportional to the normal $\bnu(t, \bm{x})$ on $\partial^*\A(t)$ and zero otherwise. 
	\end{enumerate}
\end{remark}

\begin{remark}[Integer varifolds, generalized mean curvature and rectifiability] \label{rem:integer_varifolds}
	Comparing our notion of a weak solution to the concepts proposed in \cite{roeger, hensel_stinson}, one quickly sees that theirs is strictly stronger. For one, they require $(\mu_t^\Omega)_{t \in (0, T_*)}$ to be integer varifolds, entailing, in particular, that apart from zero the function $\theta$ in Corollary~\ref{cor:properties_varifold} may only attain values in $\{\tfrac{1}{n}\}_{n \in \N}$. Moreover, they also demand the existence of a generalized mean curvature vector for the interface $\Gamma(t)$. The proof of these properties mainly relies on estimates for the first variation of the varifolds.  More precisely, one needs to show local boundedness, such that the classical theory by Allard \cite{allard} and the profound results of Schätzle \cite{schaetzle_01} may be applied. Unfortunately, considering the distributional formulation of the Gibbs--Thomson law \eqref{eq:potential_thm}, it is immediately obvious that verifying these bounds is prohibitively difficult in our situation. \\ 
	While rectifiability is another consequence of the insights in these references, there are also other tools available. For example, Garcke and Kwak \cite{kwak} utilized a result by Luckhaus \cite{luckhaus_05} to show rectifiability for a Mullins--Sekerka-type system that was coupled to homogeneous linear elasticity, but still required an additional assumption on the density of the varifolds. There is, however, another caveat concerning this strategy. To implement a similar approach for the presented model, one would have to obtain even higher regularity, more precisely $\bm{H}^3$ estimates, for the displacement $\bm{u}$, but due to the natural (Neumann) boundary conditions involving $\nabla \pt \bm{u}$, and lacking regularity of $\pt \bm{u}$, this also seems out of reach. 
\end{remark}

\begin{corollary}[Modified energy dissipation inequality]\label{cor:energy_dissipation}
	For a solution $(\mu,\chi, \bm{u})$ of the visco-elastic Mullins--Sekerka system in the sense of Theorem \ref{thm:varifold}, it holds for almost all $0 \leq s \leq T \leq T_*$ that 
	\begin{align*}\label{eq:energy_dissipation_cor}
		\F&(\mu, \chi, \bm{u}) (T)\\
		&+ \frac{1}{2}\int_s^T  
		\begin{aligned}[t]
			\Big[  \norm{\pt \chi}_{H^{-1}_{(0)}}^2 
			+  \norm{ \pt \bm{u}}^2_{\bm{H}^1_{\chi(t)}}
			+ \delta \F(\mu, \chi, \bm{u} )[\bm{B}, \bm{v}] 
			& - \frac{1}{2} \norm{\bm{B}\cdot  \nabla \chi}_{H^{-1}_{(0)}}^2 - \frac{1}{2}\norm{\bm{v} - \nabla \bm{u}\bm{B}}_{\bm{H}^1_{\chi(t)}}^2 \Big] \dt 
		\end{aligned}\\ 
		&\!\!\!\!\leq  \F(\mu, \chi, \bm{u}) (s) \leq \F(\chi_0, \bm{u}_0)
		\numberthis
	\end{align*}
	for all measurable maps $\bm{v} \colon (0, T_*) \to  \bm{H}^2(\Omega)$ with $\bm{v}(t) \in \bm{H}^1_\perp (\Omega)$ for all $t \in (0, T_*)$ and satisfying $\norm{\bm{v}}_{\bm{H}^2(\Omega)} \in L^2(0, T_*)$, and all measurable functions $\bm{B} \colon (0, T_*) \to \bm
	C^2(\overline{\Omega})$ such that $\bm{B}(t) \in \mathcal{S}_{\chi(t)}$ for all $t \in (0, T_*)$ and $\norm{\bm{B}}_{\bm{C}^2} \in L^2(0, T_*)$. \\  
	Moreover, it even holds for almost all $0 \leq s \leq T \leq T_*$ that 
	\begin{align*}\label{eq:energy_dissipation_cor_sup}
		\F&(\mu, \chi, \bm{u}) (T)\\
		&
		\begin{aligned}[t]
			&+ \frac{1}{2}\int_s^T   
			\norm{\pt \chi}_{H^{-1}_{(0)}}^2 
			+  \norm{ \pt \bm{u}}^2_{\bm{H}^1_{\chi(t)}} \dt  \\ 
			&+ \frac{1}{2}\int_s^T  
			\sup_{\substack{\bm{B} \in \mathcal{S}_{\chi(t, \cdot)} \\ \bm{v} \in \bm{H}^1_\perp(\Omega) \cap \bm{H}^2(\Omega)}} 
			\Big\{  \delta \F(\mu, \chi, \bm{u} )[\bm{B}, \bm{v}] 
			- \frac{1}{2} \norm{\bm{B}\cdot  \nabla \chi}_{H^{-1}_{(0)}}^2 - \frac{1}{2}\norm{\bm{v} - \nabla \bm{u}\bm{B}}_{\bm{H}^1_{\chi(t)}}^2 \Big\} \dt 
		\end{aligned}\\ 
		&\!\!\!\!\leq  \F(\mu, \chi, \bm{u}) (s) \leq \F(\chi_0, \bm{u}_0). 
		\numberthis
	\end{align*}
\end{corollary}
The proof of this corollary can be found in Section \ref{sec:further_prop}. 
We refer to Section \ref{sec:gradient_flow}, Section \ref{sec:mm} and \eqref{eq:variation_varifold}, where it was shown that 
\begin{align} \label{eq:first_variation_summary} 
	\delta &\F(\mu, \chi, \bm{u} )[\bm{B}, \bm{v}]\\ 
	&=
	\begin{aligned}[t]
		&\int_{\overline{\Omega} \times \mathbb{S}^{d-1}} (\bm{I} - \bm{s} \otimes \bm{s}) \colon \nabla \bm{B} \, d\mu (\bm{x}, \bm{s}) \notag
		\\ 
		&+ \int_\Omega \big[ W(\chi, \E(\bm{u})) \bm{I} -  (\nabla {\bm{u}} )^T W_{, \E} (\chi, \E(\bm{u})  \big)\big] \colon \nabla \bm{B} \dx 
		+ \int_\Omega \tfrac{1}{2} \abs{\nabla^2 \bm{u}}^2\  \div \bm{B} \dx \\ 
		&-\int_\Omega  (\nabla^2 \bm{u})_{ijk} (\nabla \bm{u})_{ip}\,(\nabla^2 \bm{B})_{pjk}  \dx 
		- \int_\Omega(\nabla^2 \bm{u})_{ijk}   (\nabla^2 \bm{u})_{ipk}\,(\nabla\bm{B})_{jp}  \dx \\
		& - \int_\Omega(\nabla^2 \bm{u})_{ijk}  (\nabla^2 \bm{u})_{ijp}\,(\nabla\bm{B})_{kp}  \dx
		+ \int_\Omega \WE (\chi, \E(\bm{u})) \colon \nabla \bm{v} \dx 
		+\int_\Omega \nabla^2 \bm{u} \colon   \nabla^2 \bm{v} \dx. 
	\end{aligned} 
\end{align}
In comparison with the optimal energy dissipation inequality in the finite-dimensional case, cf.\ \cite{steinke} and the references therein, the structural similarity becomes apparent. 
In fact, recent investigations regarding weak-strong uniqueness, see \cite{fischer2024weakstronguniquenessprinciplemullinssekerka} and the references therein, crucially depend on an optimal energy dissipation inequality. Note that in \cite{fischer2024weakstronguniquenessprinciplemullinssekerka}, the authors rely on an inequality of the form \eqref{eq:energy_dissipation_thm} rather than \eqref{eq:energy_dissipation_cor} or \eqref{eq:energy_dissipation_cor_sup}. The formulation in Corollary \ref{cor:energy_dissipation} combined with the comparison to the finite-dimensional setting highlights the underlying principle on which their argument is based.
For similar formulations in related contexts, we also refer to \cite{laux_chambolle, kubin2026veriginproblemphasetransition, hensel_stinson}. \par 
In light of this result, we now examine the connection between our formulation and the general theory of gradient flows, cf.\ \cite{ambrosio}, in particular De Giorgi's notion of metric slopes. The following lemma establishes that the quantities in \eqref{eq:energy_dissipation_cor} that depend on $\bm{B}$ and $\bm{v}$ provide a lower bound for the metric slope, cf.\ \cite[Lem.\ 2]{hensel_stinson}

\begin{lemma}[Relation to metric slopes]\label{lem:relation_metric_slope}
	Let $\chi \in BV(\Omega; \{0, 1\})$, $\mu \in M(\overline{\Omega} \times \S)$, and $\bm{u} \in \bm{H}^2(\Omega)$. Then, it holds that
	\begin{align*}
		\sup_{\bm{B}, \bm{v}} \limsup_{\tau \to 0}
		&\frac{  \big( \F(\mu, \chi, \bm{u}) - \F(\mu^\tau, \chi^\tau, (\bm{u} + \tau \bm{v})^\tau )   \big)_+ }{ \textnormal{d} \big( (\chi, \bm{u}),  (\chi^\tau, (\bm{u} + \tau \bm{v} )^\tau) \big)}\\ 
		&\geq \sup_{\bm{B}, \bm{v}} \bigg(  
		\delta \F(\mu, \chi, \bm{u} )[\bm{B}, \bm{v}] 
		- \frac{1}{2} \norm{\bm{B}\cdot  \nabla \chi}_{H^{-1}_{(0)}}^2 - \frac{1}{2}\norm{\bm{v} - \nabla \bm{u}\bm{B}}_{\bm{H}^1_{\chi}}^2 \bigg),  
	\end{align*}
	where 
	\begin{equation*}
		\textnormal{d} \big( (\chi, \bm{u}), (\tilde{\chi}, \tilde{\bm{u}})) \big)
		= \Big( \norm{\tilde{\chi} - \chi }_{H^{-1}_{(0)}}^2 + \norm{ \tilde{\bm{u}} - \bm{u}}_{\bm{H}^1_{\chi}}^2 \Big)^{\frac{1}{2}}, 
	\end{equation*}
	and the supremum is taken over all $\bm{v}\in \bm{H}^1_\perp(\Omega)\cap \bm{H}^2(\Omega)$, and $\bm{B} \in \mathcal{S}_\chi$ with an associated family of diffeomorphisms  $\Phi(\tau, \cdot)$, see Lemma \ref{lem:diffeomorphisms}, where the superscript $f^\tau$ denotes the composition $f \circ \Phi(-\tau, \cdot)$.
\end{lemma}

For the proof, we refer to Section \ref{sec:further_prop}. 

\begin{remark}
	Hensel and Stinson \cite[Lem.\ 2]{hensel_stinson} are actually able to prove equality in the analogous relation for the pure Mullins--Sekerka system. However, they assume that the first variation of the varifold is given by a generalized mean curvature. While this can be verified in their situation using results by Schätzle \cite{schaetzle_01}, we are missing local bounds for the first variation of the varifold that obstruct the application of these in our case.  
\end{remark}

\begin{remark}[Possible results in case of homogeneous elasticity]
	In Remark~\ref{rem:integer_varifolds}, we explained why stronger results for the proposed system cannot be expected in the most general setting. This situation, however, changes significantly if we assume all elasticity coefficients to be constant, i.e.\, instead of \eqref{eq:elastic_thm} we would consider 
	\begin{equation*}
		\int_0^{T_*}\int_\Omega 
		\big[\C_{\nu} \E ( \pt \bm{u}) 
		+ \WE (\E({\bm{u}}))
		\big] 
		\colon \E(\bm{v}) 
		+ \nabla^2 {\bm{u}} \colon   \nabla^2 \bm{v} \dx \dt  = 0. 
	\end{equation*}  
	In this case, elliptic theory suggests that full $\bm{H}^4$-regularity for $\bm{u}$ is obtainable, such that the strategy of Hensel and Stinson combined with many of the computations from Section \ref{sec:gradient_flow} would entail stronger solutions in the fashion of \cite{hensel_stinson}. Consequently, we expect the techniques developed in  \cite{fischer2024weakstronguniquenessprinciplemullinssekerka} to be applicable, yielding a weak-strong uniqueness result. Note that this is consistent with the theory of diffuse interface systems with elasticity, where, to the best of our knowledge, weak–strong uniqueness results are likewise restricted to the homogeneous case.
\end{remark}

Assuming that there is no loss of perimeter in the limiting process of the minimizing movement scheme, i.e., 
\begin{equation}\label{eq:energy_conservation_hypo}
	\underset{h \searrow 0}{\textnormal{liminf }} \int_0^{T_*} \int_\Omega d\abs{\nabla \chi^h}(t)  \dt 
	\leq 
	\int_0^{T_*}\int_\Omega d\abs{\nabla \chi}(t)  \dt, 
\end{equation}
where $\chi^h$ are approximate solutions corresponding to a time step $h > 0$ in an implicit time discretization, cf.\ Section~\ref{sec:BC_solutions}, we can even prove the following, stronger, statement. 

\begin{theorem}[$BV$-solutions]\label{thm:bv_solutions}
	Let $\Omega \subset \R^d$, $d = 2, 3$, be a bounded, smooth domain. Moreover, suppose that a terminal time $T_* > 0$, and initial data $\chi_0  \in BV (\Omega; \{0, 1\})$ and $\bm{u}_0 \in \bm{H}^1_{\perp}(\Omega) \cap \bm{H}^2(\Omega)$ are given. 
	Assuming that the energy conservation hypothesis \eqref{eq:energy_conservation_hypo} holds, and $\cos \alpha = 0$, then there exist
	\begin{align*}
		\chi &\in L^\infty_{w^*}(0, T_*; BV(\Omega; \{0,1\})), \\ 
		v &\in L^2(0, T_*; H^1_{(0)}(\Omega)), \\ 
		\bm{u} &\in H^1(0, T_*; \bm{H}^1_\perp(\Omega)) \cap L^\infty(0, T_*; \bm{H}^2(\Omega)), 
	\end{align*}
	that solve the regularized visco-elastic Mullins-Sekerka system \eqref{eq:sg_sharp_bulk_ve}-\eqref{eq:sg_initial} in the following sense:\\ 
	The evolution satisfies 
	\begin{equation}\label{eq:pt_bv}
		\int_0^T \int_\Omega \nabla v \cdot \nabla \xi(t) \dx \dt 
		= \int_0^T \int_\Omega {\chi}(t) \partial_t \xi(t) \dx \dt
		+ \int_\Omega  \chi_0 \xi(0) \dx 
	\end{equation}
	for all $\xi \in C^\infty([0, T] \times \overline{\Omega})$ with $\xi(T) = 0$. Moreover, it holds that
	\begin{align*}
		& \int_\Omega \Big( \nabla \cdot \bm{B} - \tfrac{\nabla \chi}{\abs{\nabla \chi}} \cdot  \nabla \bm{B} \tfrac{\nabla \chi}{\abs{\nabla \chi}} \Big) \, d\abs{\nabla \chi}  \numberthis \label{eq:potential_bv_thm} \\ 
		& \quad =  \int_\Omega \chi \  \div \big( v \bm{B} \big) \dx  
		+  \int_\Omega \C_{\nu}(\chi)  \E(\pt \bm{u})
		\colon \nabla^2 \bm{u} \bm{B}  \dx 
		- \int_\Omega \tfrac{1}{2} \abs{\nabla^2 \bm{u}}^2\  \div \bm{B} \dx \\ 
		&\qquad -  \int_\Omega \Big[ W(\chi, \E(\bm{u})) \bm{I} -  ( \nabla \bm{u} )^T \big( \C_{\nu}(\chi)  \E (\pt \bm{u}) +  W_{, \E} (\chi, \E(\bm{u})  \big)\Big] \colon \nabla \bm{B} \dx 
		\\ 
		&\qquad + \int_\Omega  (\nabla^2 \bm{u})_{ijk} (\nabla \bm{u})_{ip}\,(\nabla^2 \bm{B})_{pjk}  \dx 
		+ \int_\Omega(\nabla^2 \bm{u})_{ijk}  \big( (\nabla^2 \bm{u})_{ipk}\,(\nabla\bm{B})_{jp} 
		+ (\nabla^2 \bm{u})_{ijp}\,(\nabla\bm{B})_{kp} \big) \dx
	\end{align*}
	for  all $\bm{B} \in \bm{C}^2(\overline{\Omega})$ with $\bm{B}_{|\partial \Omega} \cdot \bm{n}_{\partial_\Omega} = 0$ and almost all $t \in (0, T_*)$. 
	The displacement $\bm{u}$ satisfies  
	\begin{align}\label{eq:elastic_thm_bv}
		\int_0^{T_*}\int_\Omega 
		\big[\C_{\nu}(\chi) \E ( \pt \bm{u}) 
		+ \WE ({\chi}, \E({\bm{u}}))
		\big] 
		\colon \E(\bm{v}) 
		+ \nabla^2 {\bm{u}}\colon \nabla^2 \bm{v} \dx \dt  = 0
	\end{align}
	for all $\bm{v} \in L^2(0, T_*; \bm{H}^2(\Omega))$. The solution further satisfies the following dissipation inequality for almost all $T \in (0, T_*)$: 
	\begin{align}
		\F(\chi, \bm{u}) (T) 
		+ \frac{1}{2}\int_0^{T}  \norm{v}_{H^{1}}^2 
		+ \norm{\pt \bm{u}}_{\bm{H}^1_{\chi}}^2 \dt 
		\leq 
		\F(\chi, \bm{u})(0) . 
	\end{align}
\end{theorem}

\addtocontents{toc}{\SkipTocEntry}
\subsection*{Structure of this paper}
In Section \ref{sec:prelimiries} we start by introducing more notion along with some function spaces, and continue with a brief summary of relevant definitions and results on measure theory, $BV$-functions and the concept of oriented varifolds.\\ 
We proceed in Section~\ref{sec:gradient_flow} with the formal derivation of the regularized visco-elastic Mullins--Sekerka system as a gradient flow, where we separately introduce the kinetics and energetics, before combining both in the gradient system leading to the strong equations presented above. \\ 
Finally, Section \ref{sec:mm} is dedicated to the proofs of our main results -- Theorem~\ref{thm:varifold} and Theorem~\ref{thm:bv_solutions}. Proceeding as in \cite{hensel_stinson}, we employ a minimizing movement scheme relying on ideas from De Giorgi to construct weak solutions that satisfy our specifications. Assuming energy conservation in the limit and restricting ourselves to contact angles with $\alpha = \tfrac{\pi}{2}$, we further establish the conditional existence of $BV$-solutions in the sense of Luckhaus and Sturzenhecker \cite{sturzenhecker}; see also \cite{Bronsard_Garcke, sturzenhecker_garcke}.

\section{Notation and preliminaries} \label{sec:prelimiries}

\addtocontents{toc}{\SkipTocEntry}
\subsection*{General Notation}

In this text, bold symbols denote vector-valued quantities, while scalar-valued functions are written in standard font.
The standard inner product between two vectors $\bm{a}, \bm{b} \in \R^n$ will be denoted by $\bm{a} \cdot \bm{b}$, while the Frobenius inner product for matrices $\bm{A}, \bm{B} \in \R^{m \times n}$ is given by $\bm{A} \colon \bm{B} \coloneqq \sum_{i,j= 1}^{m, n} A_{ij} B_{ij} $. 
Moreover, we set $[\bm{a} \otimes \bm{b}]_{ij} = a_{i} b_{j}$. As usual, $\bm{A}^T, \bm{a}^T$ represent the transpose of a matrix $\bm{A} \in \R^{m \times n} $ or a vector $\bm{a} \in  \R^{n}$, respectively.\\
We will further encounter higher order tensors, i.e., elements in $\R^{3^n}$ or $\R^{4^n}$. For these, we indicate the contraction with respect to certain indices using the Einstein notation; e.g.\ for $\bm{A}  = (A_{ijkl})_{i,j,k,l= 1}^n\in \R^{n^4}$ and $\bm{B} = (B_{pjk})_{p,j,k = 1}^n \in \R^{n^3}$, 
the contraction with respect to the indices $j$ and $k$ yields a new tensor $\bm{C} =(C_{ipl})_{i, p, l= 1}^n \in \R^{n^3}$ with the pointwise definition 
\begin{equation*}
	C_{ipl}   \coloneqq A_{ijkl} B_{pjk} \coloneqq  \sum_{j,k= 1}^n A_{ijkl} B_{pjk}
\end{equation*}
In other words, we contract over the repeated indices.\\ 
For a topological vector space $V$ we denote by $V'$ its dual and ${}_{V'}\langle \cdot, \cdot \rangle_V$ represents the duality pairing. On a Hilbert space $H$ we write $(\cdot, \cdot)_H$ for the inner product. \\ 
Given a Gâteaux differentiable functional $\F \colon Y \to \R, y \mapsto F(y)$ defined on a Banach space $Y$, we denote the Gâteaux derivative at $y \in Y$ in direction $\xi$ with $\delta \F(\tilde{y})[\xi] = \delta_y \F(\tilde{y})[\xi]$. 

\addtocontents{toc}{\SkipTocEntry}
\subsection*{Function spaces}

For an open set $A \subset \R^d$ or when $A$ is the closure of an open set (i.e., $A = \overline{B}$, $B$ open), we denote by $C^0(A), C^k(A), C^\infty(A)$ the usual continuous, k-times continuously differentiable, and smooth functions, respectively, on $A$. For vector- or tensor-valued function spaces we use the bold font, e.g.\ $\bm{C}^0(A) = \bm{C}^0(A; \R^m)$, and use the subscript $c$, e.g.\ $C^0_c(A)$, for compactly supported functions in $A$.\\ 
If $A \subseteq \R^d$ is measurable, we denote the usual Lebesgue spaces by $L^p(A)$, $1 \leq p \leq \infty$, and use $\bm{L}^p(A; X)$ for strongly measurable $p$-integrable (or essentially bounded) functions, where $X$ is a Banach space (in case that $A$ is an interval of $\R$ we forgo to write $L^p$ in bold). 
Likewise, for an open set $A \subset \R^d$, we use the notation $W^{k, p}(A)$ for the standard Sobolev spaces and $\bm{W}^{k, p}(A)$ for the vector-valued variants. In the Hilbert space case, i.e., $p = 2$, we shorten the notation, setting $H^k(A) = W^{k, 2}(A)$. \\ 
The space of infinitesimally rigid deformation $\bm{H}^{1}_{\textnormal{ird}}(\Omega)$ is defined as 
\begin{align*}
	\bm{H}^{1}_{\textnormal{ird}}(\Omega) &\coloneqq 
	\big\{ 
	\begin{aligned}[t]
		\bm{u} \in \bm{H}^1(\Omega ) \colon 
		&\textnormal{there exists } \bm{b} \in \R^n  
		\textnormal{ and a skew symmetric matrix } \bm{A} \in \R^{n \times n} \\ 
		&\textnormal{ such that } \bm{u}(\bm{x}) = b + \bm{A} \bm{x}  \textnormal{ for all } \bm{x} \in \Omega
		\big\}
	\end{aligned}
	\\
	&= \big\{
	\bm{u} \in \bm{H}^1(\Omega) \colon \E(\bm{u}) = \tfrac{1}{2} \big( \nabla \bm{u} + (\nabla \bm{u})^T\big) = \bm{0}
	\big\}. 
	\numberthis \label{eq:H_ird}
\end{align*}
We denote its orthogonal complement with respect to the $\bm{H}^1$-inner product as $\bm{H}^1_{\perp}(\Omega) \coloneqq (\bm{H}^{1}_{\textnormal{ird}}(\Omega))^{\perp}$, which is closed (by definition through the orthogonal complement) and therefore again a Hilbert space. Note, that the elastic energy $W$ only depends on $\E(\bm{u})$ and not at all on the infinitesimally rigid contributions. For this space, Korn's inequality is applicable, providing a constant $C_K > 0$ such that, see \cite{ciarlet2013linear, zeidler1988nonlinear}, 
\begin{equation}\label{iq:Korn_perp}
	\norm{\bm{u}}_{\bm{H}^1} \leq C_K \norm{\E(\bm{u})}_{\bm{L}^2}\quad  \textnormal{ for all } \bm{u} \in \bm{H}^1_\perp(\Omega). 
\end{equation} 
Observe that this defines an equivalent norm on $\bm{H}^1_\perp(\Omega)$. Moreover, since the subspace $\bm{H}^1_\perp(\Omega) \subset \bm{H}^1(\Omega)$ is closed, by definition through orthogonal complement, it is in particular a Hilbert space. 

\addtocontents{toc}{\SkipTocEntry}
\subsection*{Measures} 

Let $X$ be a locally compact metric space and set $\bm{M}(X; \R^m)$ as the space of all finite $\R^m$-valued Radon measures, with $ M(X) \coloneqq \bm{M}(X; \R)$. For any $\mu \in \bm{M}(X; \R^n)$, the \textit{total variation} measure $\abs{\mu}$ satisfies, see \cite[Prop.\ 1.47]{ambrosio2000functions},  
\begin{equation*}
	\abs{\mu} (A) = \sup \Big\{
	\int_X \bm{f} (\bm{x}) \cdot  d\mu(\bm{x}) \colon \bm{f} \in \bm{C}^0_c(X; \R^m), \ \supp \bm{f} \subset A, \norm{\bm{f}}_{\infty} \leq  1
	\Big\} 
\end{equation*}
for every open set $A \subset X$, and for any arbitrary $E \subset X$ it holds that 
\begin{equation*}
	\abs{\mu}(E) = \inf \Big\{
	\abs{\mu}(A) \colon E \subset A \textnormal{ and } A \textnormal{ is open in } X 
	\Big\} . 
\end{equation*}

Let us continue with a generalized product. Here, we assume $U \subset \R^N, V \subset \R^M$ to be open sets, $\mu$ a positive Radon measure on $U$, and $x \mapsto \omega_x$ mapping all $x \in U$ to a finite $\R^m$-valued Radon measure $\omega_x$ on $V$. Moreover, assume that $x \mapsto \omega_x (W)$ is $\mu$-measurable for any open set $W \subset V$ and suppose 
\begin{equation*}
	\int_{K} \abs{\omega_x} (V) \, d\mu < \infty \quad \textnormal{for all } K \subset U \textnormal{ compact.}
\end{equation*}
Then, $\mu\otimes (\omega_x)_{x \in U}$ given by 
\begin{equation*}
	(\mu\otimes \omega_x) (B) \coloneqq \int_U \Big( \int_V \chi_B (x, y) \, d\omega_x(y) \Big) \, d\mu(x) 
	\quad 
	\textnormal{for all }B \in \mathcal{B} (K \times V), K \subset U \textnormal{ compact}, 
\end{equation*}
gives rise to an $\R^m$-valued Radon measure on $U \times V$.\\ 
Note that, following \cite[Def.\ 2.27]{ambrosio2000functions}, this result is merely stated for open sets $U$ and $V$. However, the same construction remains valid for $U$ being the closure of a bounded domain $\Omega$ and $V = \mathbb{S}^{d-1}$ if we require that $x \mapsto \omega_x (W)$ is $\mu$-measurable for any relatively open set $W \subset V$. Indeed, by extending $\mu$ by zero to a sufficiently large open neighborhood of $U$ and identifying each measure $\omega_x$ with its zero extension to $\R^d$, while setting $\omega_x = 0$ if $x \not\in U$, all assumptions are satisfied. The resulting measure is again supported on $U \times V$, and the restriction to $U \times V$ yields the desired construction. One can argue similarly if $U = (0, T_*)$ and $V = \overline{\Omega} \times \S$.

We write $(\mu_\llcorner A) (B) \coloneqq \mu (A \cap B)$ for the restriction of a measure $\mu$ to a $\mu$-measurable set and denote by $\mathcal{H}^s$, $0 \leq s \leq d$, the $s$-dimensional Hausdorff measure
on $\R^d$. Finally $\mathcal{L}^d$ denotes the $d$-dimensional Lebesgue measure. 

\begin{theorem}{(Disintegration, \cite[Thm.\ 2.28]{ambrosio2000functions})}
	\label{thm:disintegration}
	Suppose that $U \subset \R^N$, $V \subset \R^M$ are open sets, and $\mu \in \bm{M}(U \times V; \R^m)$. Define for all Borel sets $A \subset U$ the measure $\nu(A) = \abs{\mu} (A \times V)$, and further assume $\nu \in  M(U)$. Then, there exists a $\nu$-measurable function $x \mapsto \omega_x$ mapping $U$ to finite, $\R^m$-valued Radon measures with $\abs{\omega_x} (V) = 1$ for $\nu$-a.e. $x \in U$, such that it holds for all $f \in L^1((U \times V, \mu))$ that 
	\begin{align*}
		f(x, \cdot ) \in L^1((V, \abs{\omega_x})) \quad \textit{for $\nu$-a.e.\ } x \in U, \\
		x \mapsto \int_V f (x, y) \, d\omega_x(y) \in L^1((U, \nu)), \\ 
		\int_{U \times V} f(x, y) \, d\mu(x, y) = \int_U \Big( \int_V f(x, y) \, d\omega_x(y) \Big) d\nu(x). 
	\end{align*} 
	If $\mu$ is a positive measure, then $\omega_x = \abs{\omega_x}$ are probability measures $\nu$-a.e.\ in $U$. 
\end{theorem}
Again, this disintegration result remains valid if $U = \overline{\Omega}$ for a bounded domain $\Omega$ and $V = \mathbb{S}^{d-1}$. Indeed, by extending $\mu$ by zero to a sufficiently large open neighborhood, the theorem applies in the extended setting. Since the extended measure is still supported in $U \times V$, it follows that the corresponding $\omega_x$ are indeed supported in $V$ for $\nu$-almost all $x \in U$. Therefore, the resulting disintegration formula immediately restricts to $U \times V$. A similar argument applies in the case $U = (0, T_*)$ and $V = \overline{\Omega} \times \S$. 

\addtocontents{toc}{\SkipTocEntry}
\subsection*{BV functions and varifolds} 

Let $\Omega \subset \R^n$, be an open set. The space of functions of bounded variation, denoted by $BV(\Omega)$, consists of all functions $f \in L^1(\Omega)$, whose distributional gradient is a finite $\R^n$-valued Radon measure. In particular, we set 
\begin{align*}
	BV(\Omega) &= \{f \in  L^1(\Omega)  \colon \nabla f \in \bm{M}(\Omega, \R^n)  \},
	\quad \textnormal{with} \quad 
	\norm{f}_{BV} = \norm{f}_{L^1} + \abs{\nabla f}(\Omega), 
\end{align*}
where $\nabla f$ denotes the distributional derivative. 
Consistent with our notation, the space $BV(\Omega; \{0, 1\})$ denotes the set of functions $\chi \in BV(\Omega)$ which only attain the values in $\{0, 1\}$ almost everywhere in $\Omega$. With any measurable set $\A \subseteq \Omega$ we can associate its \textit{characteristic function} $\chi_\A$, setting $\chi_{\A}(\bm{x}) = 1$ if $\bm{x} \in \A$ and zero otherwise. If $\chi_\A \in BV(\Omega, \{0, 1\})$, then $\A$ is a so-called \textit{set of finite perimeter}, and $\abs{\nabla \chi_\A} = \mathcal{H}^{n-1}{}_\llcorner \partial^*\A$, where $\partial^* \A$ is the \textit{reduced boundary} of $\A$. In particular, the perimeter of $\A$ is given by $\abs{\nabla \chi_\A}(\Omega)$, and one obtains from the definition of $\nabla \chi_{\A}$ that 
\begin{equation}\label{eq:PI_BV}
	\int_\A \div\, \bm{\xi} \dx 
	= -\int_\Omega \bm{\xi} \cdot d\nabla \chi_\A
	= - \int_\Omega \bm{\xi} \cdot \bnu_\A \, d\abs{\nabla \chi_\A} 
	= - \int_{\partial^* \A} \bm{\xi} \cdot \bnu_\A \, d\mathcal{H}^{n-1}
\end{equation} 
for all $\bm{\xi} \in \bm{C}^1_c(\Omega)$, where $\bnu (\bm{x}) =  {\nabla \chi_\A}/ {\abs{\nabla \chi_\A}}$ coincides with the inner unit normal to $\A$ in case that the boundary is smooth. For a detailed treatment of sets of finite perimeter, we refer to \cite[Sec.\ 3.3]{ambrosio2000functions}.

\par 
\medskip 
An \textit{(oriented) varifold} on $\overline{\Omega} \subseteq \R^n$ is a positive Radon measure $V \in M(\overline{\Omega} \times \mathbb{S}^{n-1})$, where $\mathbb{S}^{n-1}$ denotes the unit sphere in $\R^n$. Associated to the varifold is its mass measure $\abs{V} \in M(\overline{\Omega})$ defined by the expression $\abs{V} (U) \coloneqq V ( U  \times \mathbb{S})$ for any measurable set $U \subset\overline{\Omega}$. The \textit{first (tangential) variation} of a varifold $\mu$ (or rather its mass measure) is defined as, cf.\ \cite{allard, hensel_stinson},
\begin{equation}\label{eq:variation_varifold}
	\langle \delta V , \bm{B} \rangle = \int_{\overline{\Omega} \times \mathbb{S}^{n-1}}   (\bm{I} - \bm{s} \otimes \bm{s}) \colon \nabla \bm{B} \, dV(\bm{x, \bm{s}}) 
	\quad \textnormal{for all } \bm{B} \in \bm{C}^1(\overline{\Omega}) \textnormal{ with } \bm{B}_{|\partial\Omega} \cdot \bm{n}_{\partial\Omega} = 0. 
\end{equation}

\section{Derivation of the system} \label{sec:gradient_flow} 

The aim of this section is to  derive the visco-elastic Mullins--Sekerka system \eqref{eq:sg_sharp_bulk_ve}-\eqref{eq:sg_initial} as a (formal) gradient flow of the energy \eqref{eq:energy_ms_regularized}, which is, mentioned previously, obtained by passing to the $\Gamma$-limit in the energy of the Cahn--Larché system \cite{garcke_00}, with further augmentation by the regularizing hyperstress $\F_{\mathfrak{h}}$, and a capillary term $\F_{cap}$.\\ 
We begin by introducing the state space and defining the geometry. In analogy with the visco-elastic Cahn--Larché system, we prescribe a $H^{-1}_{(0)}$-type metric with respect to the characteristic function $\chi$, along with a $\bm{H}^1$-type metric, depending on the state $\chi$, with respect to the displacement function $\bm{u}$. Afterwards, we examine the energy functional in detail and compute its first variation with respect to the arguments $(\chi, \bm{u})$.\\ 
Subsequently, we combine kinetics and energetics to formulate the gradient flow equation, such that, after lengthy computations, we finally derive the strong formulation presented above.    

\subsection{Kinetics}\label{sec:kinetics}
For a fixed $0 < m_0 < \abs{\Omega}$, we set
\begin{gather*}
	\indic \coloneqq \Big\{ \chi \in BV(\Omega; \{0, 1\}) \colon \int_\Omega \chi \dx = m_0 \Big\}. 
\end{gather*}
Then, admissible states are elements of the formal manifold 
\begin{equation*}
	\states \coloneqq  \{ (\chi, \bm{u}) \in \indic \times \bm{H}^1_{\perp}(\Omega) \colon \bm{u} \in \bm{H}^2(\Omega) \cap \bm{H}^4(\Omega\setminus\Gamma(\chi)) \}, 
\end{equation*}
where $\Gamma = \Gamma (\chi)$, satisfying $\abs{\nabla \chi} = \mathcal{H}^{d-1} \llcorner \Gamma$, is the reduced boundary between the two subdomains 
\begin{equation*}
	\A(\chi) \coloneqq \Omega^+ \coloneqq \Omega^+(\chi) \coloneqq \{ \bm{x} \in \Omega \colon \chi(\bm{x}) = 1 \}, 
	\quad \quad 
	\Omega^- \coloneqq \Omega^-(\chi) \coloneqq \{ \bm{x} \in \Omega \colon \chi(\bm{x}) = 0 \} = \Omega \setminus \Omega^+. 
\end{equation*}
Keep in mind that the following arguments are crucially based on the assumption that $\Gamma$ is a sufficiently regular hypersurface. 
To simplify the notation, we set 
\begin{equation*}
	\disp \coloneqq  \bm{H}^1_{\perp}(\Omega) \cap \bm{H}^2(\Omega) \cap \bm{H}^4(\Omega\setminus\Gamma(\chi))
\end{equation*}
and note that by virtue of $\bm{H}^2(\Omega) \hookrightarrow \bm{C}^0(\overline{\Omega})$, the displacement $\bm{u} \in  \bm X$ cannot jump across the interface, i.e. $\jump{\bm{u}} = \bm{0}$ on $\Gamma(\chi)$. Moreover, one has the decomposition $\nabla \bm{u} = \nabla \bm{u} \bnu \otimes \bnu + \nabla_{\Gamma} \bm{u}$, where $\nabla \bm{u} \bnu$ is the normal derivative, and $\nabla_\Gamma \bm{u}$ is the (tangential) surface gradient. As $\bm{u} \in \bm{H}^2(\Omega)$, the gradient $\nabla \bm{u}$ has a trace on $\Gamma$ from both subdomains $\Omega^\pm$, and since the traces $\bm{u}^\pm$ agree, we must find that the tangential derivatives can also not jump across the interface, i.e.,  $\bjump{\nabla_\Gamma \bm{u}} = \bm{0}$. Lastly, we can use integration by parts on $\Omega$, as well as on both subdomains $\Omega^\pm$ separately; upon comparing the results, we must further conclude that the normal derivative may not jump either, i.e., $\jump{\nabla \bm{u}}\bnu = \bm{0}$. 
\\ 
If no confusion seems likely, we will neglect the dependency of $\disp, \Gamma, \Omega^\pm$ on $\chi$ in our notation.
In order to define a metric on this formal manifold, we first need to identify the tangent spaces in every point, which raises the question how to characterize curves in this setting. Obviously, we can vary the displacement by any element $\bm{v} \in \disp$ without leaving $\states$. For the indicator function $\chi$ the situation becomes much more involved. Firstly, we note that $BV(\Omega; \{0, 1\})$ -- and even more so $\indic$ -- is not closed under addition, forcing us to identify those normal velocities of $\Gamma$ which preserve the volume of $\Omega^\pm$. Secondly, we observe that after moving the boundary $\Gamma$, the displacement $\bm{u}$ will no longer satisfy $\bm{H}^4(\Omega \setminus \Gamma )$, forcing us to also account for this by allowing $\bm{u}$ to be transported along with the interface. We emphasize that this is not necessary in the purely elastic setting, where one assumes that the mechanical equilibrium is attained at a much faster time scale compared to the evolution of the phases. \\ 
Following Hensel and Stinson \cite{hensel_stinson}, we define a space of regular, $\chi$-dependent test functions
\begin{equation}\label{eq:def_S_chi}
	\mathcal{S}_{\chi} \coloneqq \Big\{ 
	\bm{B} \in \bm{C}^2(\overline{\Omega}; \R^d) \colon \int_\Omega \chi \nabla \cdot \bm{B} \dx = 0, \bm{B}_{| \partial \Omega} \cdot \bm{n}_{\partial \Omega} = 0 \Big\}, 
\end{equation}
which give rise to infinitesimal, volume preserving inner variations via a family of diffeomorphisms. 
Before stating the corresponding lemma, we want to note that for $\chi \in BV(\Omega; \{0, 1\})$ the restriction in the definition above is equivalent to, cf.\ \eqref{eq:PI_BV}, 
\begin{equation}\label{eq:normal_velocity_preserving}
	\int_\Gamma \bm{B}\cdot\bnu  \dH = 0 . 
\end{equation}

\begin{lemma}\label{lem:diffeomorphisms}
	Let $\chi \in \indic$ and $\bm{B} \in \mathcal{S}_{\chi}$. Then there exists some $\tau_0> 0$ and a family of $C^2$-diffeomorphisms $\Phi_\tau : \overline{\Omega} \rightarrow \overline{\Omega}$ depending differentiably on $\tau \in (-\tau_0, \tau_0)$, such that for all $\bm{x} \in \overline{\Omega}$ one has 
	\begin{align*}
		\Phi_0 (\bm{x}) = \bm{x}, \quad \quad \partial_\tau \Phi_\tau (\bm{x})_{|\tau = 0} = \bm{B}(\bm{x}),  \textrm{ and }\quad 
		\int_{\Omega} \chi \circ \Phi_\tau^{-1} \dx = m_0, \textrm{ for all } \tau \in (-\tau_0, \tau_0). 
	\end{align*}
\end{lemma}

\begin{proof}This follows analogously as in \cite[Lem.~8]{hensel_stinson} combined with the observation that for $\bm{B} \in \bm{C}^2(\overline{\Omega}, \R^n)$, the function $f$ (as defined in \cite{hensel_stinson}) is also twice continuously differentiable, such that the implicit function theorem yields a map that admits the same regularity.
\end{proof}

Let $(\chi, \bm{u}) \in \states$ and choose some $\bm{B} \in \mathcal{S}_\chi$ and $\bm{v} \in \disp$. Then, the following map defines an admissible curve on the formal manifold through the chosen base point:
\begin{equation*}
	(- \tau_0, \tau_0) \rightarrow \states, \quad \tau \mapsto (\chi,  \bm{u} + \tau \bm{v} )\circ \Phi^{-1}_\tau (\tau, \cdot), 
\end{equation*}
where $\Phi$ is as in Lemma \ref{lem:diffeomorphisms}. By considering this as a curve in $\big( (H^{1}(\Omega))', \bm{X} \big)$, we can (formally) differentiate in $\tau = 0$ and obtain the tangent vectors $(- \bm{B}_{| \Gamma} \cdot \bnu, \bm{z})$, where
\begin{equation*}
	(-\bm{B} \cdot \bnu)_{|\Gamma} = (\partial_\tau \Phi_\tau{}_{|\tau = 0} \cdot \bnu)_{|\Gamma} \in (H^{1}(\Omega))' 
	\qquad \textnormal{and} \qquad 
	\bm{z} = \bm{v} - \nabla \bm{u} \bm{B} \in \bm{H}^1(\Omega). 
\end{equation*}
These arguments can be summarized in the following definition of the (formal) tangent space 
\begin{equation*}
	T_{(\chi, \bm{u})} \states = \Big\{ (\mathcal{V}, \bm{z}) \colon \mathcal{V} \in C^2(\Gamma), 
	\bm{z} = \bm{v} - \nabla \bm{u} \bm{B} \textnormal{ with } \bm{v} \in \disp \textnormal{ and }
	\bm{B} \in \mathcal{S}_\chi \textnormal{ such that } (\bm{B} \cdot \bnu)_{|\Gamma} = \mathcal{V} \Big\}. 
\end{equation*}
Due to \eqref{eq:normal_velocity_preserving}, this definition entails that only normal velocities $\mathcal{V}$ with $\int_\Gamma \mathcal{V} \dH = 0 $
are admissible.\\ 
The metric tensor in a point $(\chi, \bm{u})$ is then defined by the following $(H^{-1}, \bm{H}^1)$-type functional in the bulk:  
\begin{align}\label{eq:metric_tensor}
	g_{(\chi, \bm{u})} \big((\mathcal{W}, \bm{z}), (\widetilde{\mathcal{W}}, \tilde{\bm{z}})  \big) \coloneqq 
	\int_\Omega \nabla w \cdot \nabla \widetilde{w} \dx  + \int_{\Omega} \C_{\nu}(\chi) \E(\bm{z}) : \E(\tilde{\bm{z}}) \dx, 
\end{align}
where $w \in H^1_{(0)}(\Omega) \cap H^2(\Omega\setminus\Gamma)$ (resp.\ $\widetilde{w}$) is the unique solution to the elliptic boundary value problem
\begin{equation*}
	\left\{ 
	\begin{aligned}
		\Delta w & = 0 && \textnormal{in } \Omega\setminus \Gamma ,\\ 
		- \jump{\nabla w} \cdot \bnu  &=    \mathcal{W}&& \textnormal{on } \Gamma,\\ 
		\jump{w}   &= 0 && \textnormal{on } \Gamma,\\ 
		\nabla w \cdot \bm{n}_{\partial \Omega} &= 0 && \textnormal{on } \partial \Omega,\\ 
		\int_\Omega w &= 0,
	\end{aligned}
	\right. 
\end{equation*}
which exists under the assumptions that $\Gamma$ (resp.\ $\tilde{\Gamma}$) is suitably regular. 
In particular, it is easy to see that for any positive definite $\C_{\bnu} \in \R^{4n}$ the metric tensor $g_{(\chi, \bm{u})} $ is a positive semi-definite bilinear form on $T_{(\chi, \bm{u})}\states$. 
\par 
Finally, we note that the first integral can be transformed via integration by parts  
\begin{equation}\label{eq:potential_inner_product}
	\int_\Omega \nabla w \cdot \nabla \widetilde{w} \dx
	= \int_{\Omega\setminus\Gamma} w (- \Delta \widetilde{w}) \dx - \int_{\Gamma} \jump{w \nabla \widetilde{w}} \cdot \bnu  \dH 
	=   \int_\Gamma   w \widetilde{\mathcal{W}} \dH . 
\end{equation}

\par 
\medskip 
\subsection{Energetics}\label{sec:energetics}
We proceed by defining an energy functional on the state space $\states$, which is composed of the perimeter of $\Omega^+(t)$, an elastic energy that includes the second gradient $\nabla^2 \bm{u}$, and a capillary energy to enforce a prescribed contact angle between the interface and the boundary. 
More precisely, we set 
\begin{align*}
	\F (\chi, \bm{u}) &\coloneqq  \F_{per}(\chi) + \F_{cap}(\chi)+ \F_{el}(\chi, \bm{u}) + \F_{\mathfrak{h}} (\bm{u})\\
	&\coloneqq
	\int_\Omega c_0 d\abs{\nabla \chi} + \int_{\partial \Omega} c_0 (\cos \alpha) \chi \dH
	+ \int_\Omega W(\chi, \E(\bm{u})) \dx + \int_\Omega \tfrac{1}{2} \abs{\nabla^2 \bm{u}}^2 \dx, 
\end{align*}
where $c_0 > 0$ is a positive constant related to the surface tension. For a sufficiently regular boundary $\Gamma$, the first integral also takes the form
\begin{equation*}
	\int_\Omega c_0 \abs{\nabla \chi} = \int_\Gamma c_0 \dH. 
\end{equation*}
We note that the chosen hyperstress potential is the simplest option possible and that a dependency on the order parameter $\chi$ could also be included.

\subsubsection*{First variation of the energy}
Before we can state the gradient flow equation, we need to compute the first variation of this energy functional with respect to the states $(\chi, \bm{u})$. \\ 
Computing the outer variation of $\F_{el}$ in direction $ \bm{v} \in \disp$ yields 
\begin{equation}\label{eq:variation_el_u}
	\delta_{\bm{u}} \F_{el}(\chi, \bm{u}) [\bm{v}] = \int_\Omega \C(\chi) \big( \E(\bm{u}) - \Tau(\chi) \big)  \colon \E(\bm{v}) \dx 
	= \int_\Omega W_{\E} (\chi, \E(\bm{u})) \colon \nabla \bm{v} \dx . 
\end{equation}
A well-known result, cf.\ e.g.\ \cite[Sec.\ 17.3]{maggi2012sets}, is that the inner variation of the perimeter is given by the mean curvature, i.e., 
\begin{equation}\label{eq:varation_per}
	\delta_\chi \F_{per}(\chi) [\bm{B}] = - \int_\Gamma c_0 \varkappa \bnu \cdot \bm{B} \dH \quad \textnormal{for all } \bm{B} \in \mathcal{S}_\chi. 
\end{equation}
Moreover the inner variation of $\F_{per} +\F_{cap}$ is given by, cf. \cite[Eq.\ (163)]{hensel_stinson},
\begin{align*}
	\delta_{\chi}  \big( \F_{per} (\chi) + \F_{cap} (\chi) \big) [\bm{B}] 
	= &- c_0  \int_\Gamma \varkappa \bnu \cdot \bm{B} \dH\\
	&+ c_0 \int_{ \partial ( \partial \A(\chi) \cap \Omega )} 
	(\bm{\tau}_{\partial \A(\chi)  \cap \Omega}  \cdot \bm{\tau}_{ \partial \A(\chi)  \cap \partial \Omega } - \cos \alpha) ( \bm{\tau}_{ \partial \A(\chi)  \cap \partial \Omega }  \cdot \bm{B} ) \, d \mathcal{H}^{d-2} . 
\end{align*}
Here, $\bm{\tau}_{ \partial \A(\chi)  \cap \Omega }$ is a vector field on the contact points manifold $\partial \overline{(\partial \A(\chi) \cap \Omega)} \subset \partial \Omega $ that is tangent to the interface $\partial \A(\chi) \cap \Omega = \Gamma(\chi)$, normal to the contact points manifold, and pointing away from $\partial \A(\chi) \cap \Omega $ (i.e.\  it is the conormal vector of $\partial \A(\chi) \cap \Omega$). Moreover, the vector field $\bm{\tau}_{ \partial \A(\chi)  \cap \partial \Omega }$ is also defined on the contact points manifold $\partial \overline{(\partial \A(\chi) \cap \Omega)} \subset \partial \Omega $, and denotes the inner conormal of $\partial \A(\chi) \cap \pO$. Therefore  it is tangent to $\A(\chi) \cap \pO$, normal $\partial \overline{\A(\chi) \cap \pO}$, and points toward $\A(\chi) \cap \pO$. Note that due to our choice of orientation, it holds that  $\bm{\tau}_{\partial \A(\chi)  \cap \Omega}  \cdot \bm{\tau}_{ \partial \A(\chi \cap \partial\Omega)} = \bnu \cdot  \bm{n}_{\partial\Omega}$. We refer to \cite[Sec.\ 19.1.2]{maggi2012sets} for the derivation of this formula. To abbreviate the notation, let us set $\bm{\tau}_{\Gamma} \coloneqq \bm{\tau}_{ \partial \A(\chi)  \cap \Omega }$ and $\bm{\tau}_{\po} \coloneqq \bm{\tau}_{ \partial \A(\chi)  \cap \partial \Omega }$ in subsequent arguments, leading to the more concise expression
\begin{align*}
	c_0 \int_{ \partial ( \partial \A(\chi) \cap \Omega )} 
	&(\bm{\tau}_{\partial \A(\chi)  \cap \Omega}  \cdot \bm{\tau}_{ \partial \A(\chi)  \cap \partial \Omega } - \cos \alpha) ( \bm{\tau}_{ \partial \A(\chi)  \cap \partial \Omega }  \cdot \bm{B} ) \, d \mathcal{H}^{d-2} \\ 
	&= c_0 \int_{ \partial \Gamma} 
	(\bnu \cdot  \bm{n}_{\partial\Omega} - \cos \alpha) ( \bm{\tau}_{\po }  \cdot \bm{B} ) \, d \mathcal{H}^{d-2} . 
\end{align*}
When computing the inner variation of $\F_{el}$, it is crucial to respect the structure of the tangent space and take into account that the displacement $\bm{u}$ is also transported along with the interface. In fact, we obtain
\begin{align*}\label{eq:variation_el_chi}
	\delta_\chi \F_{el} (\chi, \bm{u} ) [\bm{B}] 
	&=  \frac{d}{d\tau}_{|\tau= 0}  \int_\Omega W(\chi, \E( \bm{u})) \circ \Phi(- \tau, \cdot)   \dy\\
	&= \int_\Omega \big[ W(\chi, \E(\bm{u})) \bm{I} - (\nabla \bm{u})^T W_{, \E} (\chi, \E(\bm{u}) ) \big] \colon \nabla \bm{B} \dx \numberthis
\end{align*}
for all $\bm{B} \in \mathcal{S}_\chi$, cf. \cite{garcke_00}. 
Additionally, we find for all $\bm{v} \in \disp$ that 
\begin{equation}\label{eq:var_F_h_u}
	\delta_{\bm{u}} \F_{\mathfrak{h}} (\bm{u}) [\bm{v}] = \int_\Omega \nabla^2 \bm{u} \colon  \nabla^2 \bm{v} \dx. 
\end{equation}
The inner variation of $\F_{\mathfrak{h}}$, however, is much more involved, but tedious computations based on the approach in \cite{garcke_00} show that for all $\bm{B} \in \mathcal{S}_\chi$ one has
\begin{align*}\label{eq:var_F_h_chi}
	\delta_{\chi} \F_{\mathfrak{h}} (\bm{u}) [\bm{B}] &=  \frac{d}{d\tau}_{|\tau= 0}  \int_\Omega \tfrac{1}{2}\abs{\nabla^2 (\bm{u} \circ \Phi(-\tau, \cdot))}^2 \dy \numberthis\\
	&= \begin{aligned}[t]
		\int_\Omega \tfrac{1}{2} \abs{\nabla^2 \bm{u}}^2\  \div \bm{B} \dx 
		&- \int_\Omega  (\nabla^2 \bm{u})_{ijk} (\nabla \bm{u})_{ip}\,(\nabla^2 \bm{B})_{pjk}  \dx \\ 
		& - \int_\Omega(\nabla^2 \bm{u})_{ijk}  \big( (\nabla^2 \bm{u})_{ipk}\,(\nabla\bm{B})_{jp} 
		+ (\nabla^2 \bm{u})_{ijp}\,(\nabla\bm{B})_{kp} \big) \dx. 
	\end{aligned}
\end{align*}

\subsection{Gradient flow equation} 

The gradient flow equation with respect to the chosen kinetics and energetics is given by 
\begin{equation}\label{eq:gradient_flow_general}
	0 = g_{(\chi, \bm{u})} ((\mathcal{V}, \pt \bm{u}), (B, \bm{z})) + \delta \F(\chi, \bm{u})  [\bm{B}, \bm{z}] \quad
	\textnormal{for all} \quad 
	(B, \bm{z}) \in T_{(\chi, \bm{u})}\states. 
\end{equation}
Note that, by definition, we can decompose $\bm{z}$ as $\bm{v} - \nabla \bm{u} \bm{B}$ such that $\bm{B}_{|\Gamma} \cdot \bnu = B$. In particular, we are allowed to choose $\bm{v}$ and $\bm{B}$ freely, as long as the sum above is equal to $\bm{z}$. We emphasize that while these vector fields are not uniquely determined, we do not encounter problems when deriving the strong formulation, because during this process we separately set the two variables to zero. 
Moreover, we may rewrite the second part of the metric tensor as
\begin{equation*}
	\int_{\Omega} \C_{\nu} (\chi) \E(\pt \bm{u}) \colon \nabla \bm{z} \dx 
	= \int_{\Omega} \C_{\nu} (\chi) \E(\pt \bm{u}) \colon \nabla \bm{v}  \dx  
	- \int_{\Omega } \C_{\nu} (\chi) \E(\pt \bm{u}) \colon \nabla ( \nabla \bm{u} \bm{B} ) \dx . 
\end{equation*}
Inserting \eqref{eq:metric_tensor}, \eqref{eq:potential_inner_product}, \eqref{eq:variation_el_u}, \eqref{eq:varation_per}, \eqref{eq:variation_el_chi}, \eqref{eq:var_F_h_u} and \eqref{eq:var_F_h_chi} into \eqref{eq:gradient_flow_general} yields 
\begin{align*}
	0 = 
	&\int_\Gamma   w \bnu \cdot \bm{B} \dH 
	- \int_{\Omega } \C_{\nu} (\chi) \E(\pt \bm{u}) \colon \nabla ( \nabla \bm{u} \bm{B} ) \dx \\ 
	&\ -  \int_\Gamma c_0 \varkappa \bnu \cdot \bm{B} \dH
	+c_0 \int_{ \partial \Gamma} 
	(\bnu \cdot  \bm{n}_{\partial\Omega} - \cos \alpha) ( \bm{\tau}_{\po }  \cdot \bm{B} ) \, d \mathcal{H}^{d-2} \\
	&+\int_\Omega \big[ W(\chi, \E(\bm{u})) \bm{I} - (\nabla \bm{u})^T W_{, \E} (\chi, \E(\bm{u}) ) \big] \colon \nabla \bm{B} \dx \\
	&\ +\int_\Omega \tfrac{1}{2} \abs{\nabla^2 \bm{u}}^2\  \div \bm{B} \dx \\ 
	&\ - \int_\Omega  (\nabla^2 \bm{u})_{ijk} (\nabla \bm{u})_{ip}\,(\nabla^2 \bm{B})_{pjk}  \dx 
	- \int_\Omega(\nabla^2 \bm{u})_{ijk}  \big( (\nabla^2 \bm{u})_{ipk}\,(\nabla\bm{B})_{jp} + (\nabla^2 \bm{u})_{ijp}\,(\nabla\bm{B})_{kp} \big) \dx\\
	& \ + \int_{\Omega} \C_{\nu} (\chi) \E(\pt \bm{u}) \colon \nabla \bm{v}  \dx   + \int_\Omega \WE (\chi, \E(\bm{u})) \colon \nabla \bm{v} \dx 
	+\int_\Omega \nabla^2 \bm{u}\colon   \nabla^2 \bm{v} \dx, 
	\numberthis \label{eq:gradient_flow_sq}
\end{align*}
for all $(\bm{B}, \bm{v}) \in \mathcal{S}_\chi \times \disp$ such that $\bm{B}_{|\Gamma} \cdot \bnu = B$ and $\bm{z} = \bm{v} - \nabla \bm{u} \bm{B}$. Moreover, since $\C_{\nu} (\chi) \E(\pt \bm{u})$ and $ \WE (\chi, \E(\bm{u}))$ are symmetric, and the inner product of a symmetric and a skew symmetric matrix is zero, combined with the observation that the second gradient in $\nabla^2 \bm{v}$ annihilates contributions from infinitesimally rigid movements in $\bm{H}^1(\Omega) = \bm{H}^1_{\textnormal{ird}}(\Omega) \oplus \bm{H}^1_\perp(\Omega)$, we can even allow more general test function $\bm{v} \in \bm{H}^2(\Omega) \cap \bm{H}^4(\Omega \setminus \Gamma)$, removing the restriction to $\bm{H}^1_\perp(\Omega)$. \par 
\medskip 
Replicating the approach in \cite{mielke_roub}, we set $\bm{B} \equiv \bm{0}$, and obtain the equation
\begin{equation*}
	\int_{\Omega} \C_{\nu} (\chi) \E(\pt \bm{u}) \colon \nabla \bm{v}  \dx   + \int_\Omega \WE (\chi, \E(\bm{u})) \colon \nabla \bm{v} \dx 
	+\int_\Omega \nabla^2 \bm{u}\colon \nabla^2 \bm{v} \dx = 0, 
\end{equation*}
which transforms via integration by parts in both subdomains $\Omega^\pm$ to 
\begin{align*}
	\int_\Omega -& \div \Big( \C_{\nu} (\chi) \E(\pt \bm{u})   + \WE (\chi, \E(\bm{u})) - \div\, \nabla^2 \bm{u} \Big) \cdot \bm{v} \dx \\ 
	= &\int_{\partial \Omega} \Big[ \C_{\nu} (\chi) \E(\pt \bm{u})   + \WE (\chi, \E(\bm{u})) - \div\, \nabla^2 \bm{u} \Big] \bm{n}_{\po} \cdot \bm{v} \dH 
	\int_{\partial \Omega} \nabla^2 \bm{u} \bm{n}_{\po} \colon \nabla \bm{v} \dH\\ 
	&+	\int_{\Gamma} \Bjump{\C_{\nu} (\chi) \E(\pt \bm{u})   +\WE (\chi, \E(\bm{u})) - \div\, \nabla^2 \bm{u} } \bnu \cdot \bm{v} \dH +  \int_{\Gamma}  \Bjump{\nabla^2 \bm{u} \bnu \colon \nabla \bm{v}}  \dH . 
	\numberthis \label{eq:sg_elastic_1}
\end{align*}
Selecting any $\bm{v}$ that is compactly supported in $\Omega \setminus \Gamma$ now allows us to apply the fundamental lemma of the calculus of variations to deduce 
\begin{equation}\label{eq:second_el_bulk}
	- \div \Big( \C_{\nu} (\chi) \E(\pt \bm{u})   + \WE (\chi, \E(\bm{u})) - \div\, \nabla^2 \bm{u} \Big) = \bm{0} \quad  \textnormal{in } \Omega\setminus \Gamma. 
\end{equation}
Observe that we can decompose $\nabla \bm{v}$ on $\Gamma$ into the surface gradient and the contribution along the normal direction $\bnu$, i.e.\,   $\nabla \bm{v} = \nabla_{\Gamma} \bm{v} + (\nabla \bm{v} \bnu) \otimes \bnu$. Moreover, let us denote by $P_{\Gamma} : \bm{A} \mapsto \bm{A}- \bm{A}\bnu \otimes \bnu$ the projection to the tangential part, and recall the following integration by parts formula, see \cite{mielke_roub} and the references therein, 
\begin{equation*}
	\int_{\Gamma} \bm{A} \colon \nabla_{ \Gamma} \bm{v} \dH 
	= \int_{\Gamma}  (P_{\Gamma} \bm{A}) \colon \nabla_{\Gamma} \bm{v} \dH 
	= - \int_{\Gamma}  \div_{\Gamma} (P_{\Gamma} \bm{A}) \cdot \bm{v} \dH
	+ \int_{\partial \Gamma} (P_\Gamma \bm{A})  \bm{\tau}_{\Gamma} \cdot \bm{v} \, d\mathcal{H}^{d-2}. 
\end{equation*}
For $\bm{A} = \nabla^2 \bm{u} \bnu$, this entails 
\begin{align*}
	\int_{\Gamma} &\bjump{ \nabla^2 \bm{u} \bnu \colon \nabla \bm{v}}  \dH 
	=\int_{\Gamma} \bjump{ \nabla^2 \bm{u}(\bnu \otimes \bnu) \cdot \nabla \bm{v} \nu + \nabla^2 \bm{u}  \bnu \colon \nabla_{\Gamma} \bm{v}}  \dH \\  
	&
	=  \begin{aligned}[t]
		\int_{\Gamma} \jump{ \nabla^2 \bm{u}(\bnu \otimes \bnu) \cdot \nabla \bm{v} \bnu }\dH 
		&- \int_{\Gamma}  \div_{\Gamma} (P_{\Gamma} \jump{ (\nabla^2 \bm{u} \bnu } ) \cdot \bm{v} \dH\\
		&+ \int_{\partial \Gamma} P_\Gamma\jump{  \nabla^2 \bm{u}  \bnu} \bm{\tau}_\Gamma  \cdot \bm{v} \, d\mathcal{H}^{d-2}
		. 
	\end{aligned}
\end{align*}
For clarification, here, the term
\begin{equation*}
	\nabla^2 \bm{u}(\bnu \otimes \bnu) =
	\big( (\nabla^2 \bm{u})_{ijk}  (\bnu \otimes \bnu)_{jk} \big)_{i=1}^n
\end{equation*}
denotes the second normal derivative; we could also write $D^2\bm{u}[\bnu, \bnu]$ instead. Before continuing we need to make the observation that 
\begin{equation}\label{eq:H_2_jump}
	\jump{\nabla \bm{v} \bnu} \equiv \bm{0} \quad \textnormal{on} \quad \Gamma, 
\end{equation}
which follows immediately from the assumed regularity. 
Allowing only test functions $\bm{v}$ that vanish close to the boundary $\partial \Omega$, along with the fact that the bulk integral vanishes due to \eqref{eq:second_el_bulk}, the identity \eqref{eq:sg_elastic_1} therefore simplifies to 
\begin{align*}
	0 = &\int_{\Gamma} \Bjump{\C_{\nu} (\chi) \E(\pt \bm{u})   +\WE (\chi, \E(\bm{u})) - \div\, \nabla^2 \bm{u} } \bnu \cdot \bm{v} \dH\\ 
	&+ \int_{\Gamma} \jump{  \nabla^2 \bm{u}(\bnu \otimes \bnu) }   \cdot \nabla \bm{v} \bnu \dH 
	- \int_{\Gamma}  \div_{\Gamma} (P_{\Gamma} \jump{ (\nabla^2 \bm{u} \bnu } ) \cdot \bm{v} \dH. 
\end{align*}
By choosing test functions $\bm{v}$ that vanish on the interface $\Gamma$ but have arbitrary normal derivatives one deduces that 
\begin{align}\label{eq:interface_sq}
	\bjump{  \nabla^2 \bm{u}  }  (\bnu \otimes \bnu) &= \bm{0} \quad \textnormal{on } \Gamma, \quad  \textnormal{for all } i = 1, \ldots, n, 
\end{align}
which implies $P_{\Gamma} \jump{ \nabla^2 \bm{u} \bnu }= \bjump{  \big(\nabla^2 \bm{u}(\bnu \otimes \bnu) \big) \otimes \bnu}  = \jump{ \nabla^2 \bm{u} \bnu} $. 
Thus, allowing arbitrary test functions, we further obtain 	
\begin{align}\label{eq:normal_stress_jump}
	\Bjump{ \Big( \C_{\nu} (\chi) \E(\pt \bm{u})   + \WE (\chi, \E(\bm{u})) - \div\, \nabla^2 \bm{u} \Big) \bnu - \div_{\Gamma}(\nabla^2 \bm{u} \bnu) } &= \bm{0}
	\quad \textnormal{on } \Gamma. 
\end{align}
Exploiting $P_{\Gamma} \jump{ \nabla^2 \bm{u} \bnu } = \jump{\nabla^2 \bm{u} \bnu}$ and taking the trace on $\partial \Gamma$ it further follows that 
\begin{equation*}
	\int_{\partial \Gamma} P_\Gamma\jump{  \nabla^2 \bm{u}  \bnu} \bm{\tau}_\Gamma  \cdot \bm{v} \, d\mathcal{H}^{d-2}
	= \int_{\partial \Gamma} \jump{ (\nabla^2 \bm{u} \bnu)\bm{\tau}_\Gamma } \cdot \bm{v}  \, d\mathcal{H}^{d-2}. 
\end{equation*}
On the boundary $\po$, we compute analogously
\begin{align*}
	\int_{\partial \Omega} &\nabla^2 \bm{u} \bm{n}_{\po} \colon \nabla \bm{v} \dH \\ 
	&= \begin{aligned}[t]
		&\int_{\po} \nabla^2 \bm{u} (\bm n_{\partial\Omega} \otimes \bm n_{\partial\Omega}) \cdot \nabla \bm{v} \bm{n}_{\po}
		- \div_{\po} \big( P_{\po} \big(   \nabla^2 \bm{u}  \bm{n}_{\po} \big) \big) \cdot \bm{v}  \dH\\ 
		&+ \int_{\partial \Gamma \cap \po } P_{\po} \bjump{ \nabla^2 \bm{u}  \bm{n}_{\po}}  \bm{\tau}_{\po} \cdot \bm{v} \, d \mathcal{H}^{d-2} . 
	\end{aligned}
\end{align*}
Note that the integration by parts can only be done piecewise on $\po$, resulting in the jumps above. 
As an immediate consequence of \eqref{eq:sg_elastic_1} and our deductions, we find that the remaining boundary integrals must also vanish, i.e., 
\begin{align*}
	0  = &\int_{\po } \Big(\Big[ \C_{\nu} (\chi) \E(\pt \bm{u})   + \WE (\chi, \E(\bm{u})) - \div\, \nabla^2 \bm{u} \Big] \bm{n}_{\po} - \div_{\partial \Omega}(P_{\partial \Omega} (\nabla^2 \bm{u} \bm{n}_{\po}) \Big)  \cdot \bm{v}  \dH \\ 
	&+  \int_{\partial \Omega} \nabla^2 \bm{u} (\bm n_{\partial\Omega} \otimes \bm n_{\partial\Omega}) \cdot \nabla \bm{v} \bm{n}_{\po} \dH\\ 
	&+  \int_{\partial \Gamma \cap \po } P_{\po} \bjump{ \nabla^2 \bm{u}  \bm{n}_{\po}}  \bm{\tau}_{\po} \cdot \bm{v} \, d \mathcal{H}^{d-2}
	+ \int_{\partial \Gamma} \jump{ (\nabla^2 \bm{u} \bnu)\bm{\tau}_\Gamma } \cdot \bm{v}  \, d\mathcal{H}^{d-2}. 
\end{align*}
Assuming that $\partial \Gamma \subset \partial \Omega$ is a set of zero measure, we can proceed with analogous arguments as for the interface conditions, concluding that
\begin{align*}
	\nabla^2 \bm{u} (\bm n_{\partial\Omega} \otimes \bm n_{\partial\Omega}) &= \bm{0} \quad \textnormal{on } \partial \Omega,
	\numberthis \label{eq:boundary_sg}\\ 
	\Big[ \C_{\nu} (\chi) \E(\pt \bm{u})   +\WE (\chi, \E(\bm{u})) - \div\, \nabla^2 \bm{u} \Big] \bm{n}_{\po} - \div_{\partial \Omega}(\nabla^2 \bm{u}\bm{n}_{\po}) &= \bm{0}
	\quad \textnormal{on } \po \setminus (\partial \Gamma \cap \po ). \numberthis \label{eq:boundary_normal_stress}
\end{align*} 
Lastly, we obtain 
\begin{equation}\label{eq:sg_conormal}
	\bjump{ \nabla^2 \bm{u}  \bm{n}_{\po}  }\bm{\tau}_{\po} 
	+ \bjump{ \nabla^2 \bm{u} \bnu }\bm{\tau}_\Gamma
	= \bm{0}\quad \textnormal{on } \partial \Gamma \subset \po.   
\end{equation}
\par 
\medskip 
We continue with the case that $\bm{v} \equiv 0$.
In order to use similar arguments we need to allow for more general test function, not just those that are volume preserving with respect to $\Omega^\pm$. This can be achieved by introducing a Lagrange multiplier that is added to the chemical potential $w$. 
\par 
\smallskip  
\noindent
\textit{Claim:} There exists some $\lambda \in \R$ such that
\begin{align*}
	\int_\Gamma  &(w+ \lambda) \bnu \cdot \bm{B} \dH \numberthis \label{eq:langrange_mult_sg}  \\ 
	=&	\int_\Gamma c_0 \varkappa \bnu \cdot \bm{B} \dH
	-  c_0 \int_{ \partial \Gamma} 
	(\bnu \cdot  \bm{n}_{\partial\Omega} - \cos \alpha) ( \bm{\tau}_{\po }  \cdot \bm{B} ) \, d \mathcal{H}^{d-2} \\
	&-\int_\Omega \big[ W(\chi, \E(\bm{u})) \bm{I} - (\nabla \bm{u})^T \big( \C_{\nu} (\chi) \E(\pt \bm{u}) + W_{, \E} (\chi, \E(\bm{u}) ) \big)\big] \colon \nabla \bm{B} \dx \\
	&+\int_{\Omega } \C_{\nu} (\chi) \E(\pt \bm{u}) \colon \nabla^2 \bm{u} \bm{B} \dx -\int_\Omega \tfrac{1}{2} \abs{\nabla^2 \bm{u}}^2\  \div \bm{B} \dx \\ 
	&+ \int_\Omega
	(\nabla^2 \bm u)_{ijk}\,(\nabla \bm u)_{ip}\,(\nabla^2 \bm B)_{pjk}
	\, dx
	+
	\int_\Omega
	(\nabla^2 \bm u)_{ijk}
	\Big(
	(\nabla^2 \bm u)_{ipk}\,(\nabla \bm B)_{jp}
	+
	(\nabla^2 \bm u)_{ijp}\,(\nabla \bm B)_{kp}
	\Big)
	\dx
\end{align*}
for all $\bm{B} \in C^2(\overline{\Omega})$ with $\bm{B}_{|\partial \Omega} \cdot \bm{n}_{\partial\Omega} = 0$.
\par 
\smallskip  
Before we begin with the proof, let us remark that for a sufficiently smooth $\Gamma$ one has
\begin{equation}\label{eq:radon_identity}
	\int_\Gamma  w \bnu \cdot \bm{B} \dH = \int_\Omega w \bm{B} \cdot \bnu \, d\abs{\nabla \chi} 
	= -  \int_\Omega \chi \nabla \cdot ( w \bm{B}) \dx. 
\end{equation} 
\noindent
\textit{Proof of claim:} Let $\bm{B} \in \bm{C}^2(\overline{\Omega})$ with $\bm{B}_{|\partial \Omega} \cdot \bm{n}_{\partial\Omega} = 0$. Observe that for all $\bm{\xi} \in C^\infty(\overline{\Omega})$ with $\int_\Omega \chi \nabla \cdot \bm{\xi} \dx \neq 0$ and $\bm{\xi}_{|\partial \Omega} \cdot \bm{n}_{\partial\Omega} = 0$ the function 
\begin{equation}\label{eq:B_C^1_in_S}
	\widetilde{\bm{B}} \coloneqq \bm{B} - \bm{\xi}\  \frac{ \int_\Omega \chi \nabla \cdot \bm{B} \dx} { \int_\Omega \chi \nabla \cdot \bm{\xi} \dx}
\end{equation}
is well-defined and satisfies $\widetilde{\bm{B}}\in \mathcal{S}_{\chi}$. In particular, it is an admissible test function in \eqref{eq:gradient_flow_sq}. Using the identity \eqref{eq:radon_identity}, it follows that 
\begin{align*}
	- \int_\Omega &\chi \nabla \cdot ( w \bm{B}) \dx 
	+ \tilde{\lambda}  \int_\Omega \chi \nabla \cdot \bm{B} \dx \\ 
	=& \int_\Gamma c_0 \varkappa \bnu \cdot \bm{B} \dH
	-  c_0 \int_{ \partial \Gamma} 
	(\bnu \cdot  \bm{n}_{\partial\Omega} - \cos \alpha) ( \bm{\tau}_{\po }  \cdot \bm{B} ) \, d \mathcal{H}^{d-2} \\
	&-\int_\Omega \big[ W(\chi, \E(\bm{u})) \bm{I} - (\nabla \bm{u})^T \big( \C_{\nu} (\chi) \E(\pt \bm{u}) + W_{, \E} (\chi, \E(\bm{u}) ) \big)\big] \colon \nabla \bm{B} \dx \\
	&+ \int_{\Omega } \C_{\nu} (\chi) \E(\pt \bm{u}) \colon \nabla^2 \bm{u} \bm{B} \dx -\int_\Omega \tfrac{1}{2} \abs{\nabla^2 \bm{u}}^2\  \div \bm{B} \dx \\ 
	&+ \int_\Omega
	(\nabla^2 \bm u)_{ijk}\,(\nabla \bm u)_{ip}\,(\nabla^2 \bm B)_{pjk}
	\, dx
	+
	\int_\Omega
	(\nabla^2 \bm u)_{ijk}
	\Big(
	(\nabla^2 \bm u)_{ipk}\,(\nabla \bm B)_{jp}
	+
	(\nabla^2 \bm u)_{ijp}\,(\nabla \bm B)_{kp}
	\Big)
	\dx
\end{align*}
where 
\begin{align*}
	\tilde{\lambda}& \int_\Omega \chi \nabla \cdot \bm{\xi} \dx\\ 
	=& \int_\Omega  \chi \nabla \cdot (w \bm{\xi}) \dx
	+ \int_\Gamma c_0 \varkappa \bnu \cdot \bm{\xi} \dH 
	- c_0 \int_{ \partial \Gamma} 
	(\bnu \cdot  \bm{n}_{\partial\Omega} - \cos \alpha) ( \bm{\tau}_{\po }  \cdot \bm{\xi} ) \, d \mathcal{H}^{d-2} \\ 
	&- \int_\Omega \big[ W(\chi, \E(\bm{u})) \bm{I} - (\nabla \bm{u})^T \big(  \C_{\nu} (\chi) \E(\pt \bm{u}) + W_{, \E} (\chi, \E(\bm{u}) ) \big)  \big]  \colon \nabla \bm{\xi} \dx\\ 
	&+ \int_\Omega \C_\nu \E(\pt \bm{u}) \colon \nabla^2 \bm{u} \bm{\xi} \dx-\int_\Omega \tfrac{1}{2} \abs{\nabla^2 \bm{u}}^2\  \div \bm{\xi} \dx \\ 
	&+ \int_\Omega
	(\nabla^2 \bm u)_{ijk}\,(\nabla \bm u)_{ip}\,(\nabla^2 \bm \xi)_{pjk}
	\, dx
	+
	\int_\Omega
	(\nabla^2 \bm u)_{ijk}
	\Big(
	(\nabla^2 \bm u)_{ipk}\,(\nabla \bm \xi)_{jp}
	+
	(\nabla^2 \bm u)_{ijp}\,(\nabla \bm \xi)_{kp}
	\Big)
	\dx
\end{align*}
and we obtain \eqref{eq:langrange_mult_sg} for $\lambda = -\tilde{\lambda}$. It remains to choose a suitable $\bm{\xi}$ such that $\int_\Omega \chi \nabla \cdot \bm{\xi} \dx \neq 0$ and $\bm{\xi}_{|\partial \Omega} \cdot \bm{n}_{\partial\Omega} = 0$. For a rigorous proof, we refer to \cite[Proof of Lem.\ 9]{hensel_stinson}. Note that Hensel and Stinson there only derive bounds for $\bm{\xi}$ in $\bm{C}^1(\overline{\Omega})$; by using either higher order Schauder estimates or elliptic regularity, one can further obtain bounds for $\bm{\xi}$ in $\bm{C}^2(\overline{\Omega})$.
\hfill$\diamondsuit$\\ 
Using integration by parts, one further finds 
\begin{align*}
	-&\int_\Omega \big[ W(\chi, \E(\bm{u})) \bm{I} - (\nabla \bm{u})^T \big(  \C_{\nu} (\chi) \E(\pt \bm{u}) + W_{, \E} (\chi, \E(\bm{u}) ) \big)  \big] \colon \nabla \bm{B} \dx  \\ 
	&= \int_\Gamma \jump{W(\chi, \E(\bm{u})) \bm{I} - (\nabla \bm{u})^T \big(  \C_{\nu} (\chi) \E(\pt \bm{u}) + W_{, \E} (\chi, \E(\bm{u}) ) \big) } \bnu \cdot \bm{B} \dH\\ 
	&\quad +\int_{\po} \Big( W(\chi, \E(\bm{u})) \bm{I} - (\nabla \bm{u})^T \big(  \C_{\nu} (\chi) \E(\pt \bm{u}) + W_{, \E} (\chi, \E(\bm{u}) ) \big)\Big)  \bm{n}_{\po} \cdot \bm{B} \dH\\ 
	& \quad + \int_{\Omega\setminus\Gamma}  \nabla \cdot \big[ W(\chi, \E(\bm{u})) \bm{I} - (\nabla \bm{u})^T \big(  \C_{\nu} (\chi) \E(\pt \bm{u}) + W_{, \E} (\chi, \E(\bm{u}) ) \big)  \big] \cdot  \bm{B} \dx ,
\end{align*}
where the divergence in $\Omega\setminus\Gamma$ given by
\begin{align*}
	&\nabla \cdot \big[ W(\chi, \E(\bm{u})) \bm{I} - (\nabla \bm{u})^T \big(  \C_{\nu} (\chi) \E(\pt \bm{u}) + W_{, \E} (\chi, \E(\bm{u}) ) \big)  \big] \\ 
	&\quad = \nabla W(\chi, \E(\bm{u})) - \nabla (\nabla \bm{u})^T \WE(\chi, \E(\bm{u}))  
	- (\nabla \bm{u})^T  \nabla \cdot \big(  \C_{\nu} (\chi) \E(\pt \bm{u}) + W_{, \E} (\chi, \E(\bm{u}) ) \big) \\ 
	& \qquad - \nabla (\nabla \bm{u})^T  (\C_{\nu} (\chi) \E(\pt \bm{u}) ) . \numberthis \label{eq:computation_div}
\end{align*}
\par  
\noindent
Assume that $\bm{W}$ is a symmetric matrix. Then, it holds that 
\begin{equation*}
	(\partial_i (\nabla \bm{u})^T) \colon \bm{W}
	= \partial_i \tfrac{1}{2} \big( (\nabla \bm{u})^T + \nabla \bm{u} \big) \colon \bm{W}
	= \partial_i \E(\bm{u}) \colon \bm{W}. 
\end{equation*}
On the other hand, we compute $\partial_i W( \E(\bm{u}) ) = \WE(\E(\bm{u})) \colon \partial_i \E(\bm{u})$, and noting that $\WE$ is a symmetric matrix, it follows that
\begin{align}\label{eq:tensors_cancel}
	\nabla W(\chi, \E(\bm{u})) &- \nabla (\nabla \bm{u})^T \WE(\chi, \E(\bm{u}))   \\
	&= \big(\partial_i \E(\bm{u})\colon \WE(\chi, \E(\bm{u})) \big)_{i = 1}^d
	- \nabla (\nabla \bm{u})^T  \WE(\chi, \E(\bm{u}))  = 0. \notag
\end{align}
We also find that two more terms cancel:
\begin{equation}\label{eq:tensor_cancel_2}
	\int_{\Omega \setminus \Gamma} \C_{\nu} (\chi) \E(\pt \bm{u}) \colon \nabla^2 \bm{u} \bm{B} \dx   
	- \int_{\Omega \setminus\Gamma} \nabla (\nabla \bm{u})^T \colon (\C_{\nu} (\chi) \E(\pt \bm{u}) )  \cdot \bm{B} \dx = 0. 
\end{equation}
\par 
\noindent
In summary, we obtain that the identity \eqref{eq:langrange_mult_sg} is equivalent to 
\begin{align*}
	\int_\Gamma  & (w + \lambda) \bnu \cdot \bm{B} \dH  \numberthis \label{eq:jump_second_5}\\ 
	=&\int_\Gamma c_0 \varkappa \bnu \cdot \bm{B} \dH 
	-  c_0 \int_{ \partial \Gamma} 
	(\bnu \cdot  \bm{n}_{\partial\Omega} - \cos \alpha) ( \bm{\tau}_{\po }  \cdot \bm{B} ) \, d \mathcal{H}^{d-2} \\
	&+\int_\Gamma \bjump{W(\chi, \E(\bm{u})) \bm{I} - (\nabla \bm{u})^T \big(  \C_{\nu} (\chi) \E(\pt \bm{u}) + W_{, \E} (\chi, \E(\bm{u}) ) \big) } \bnu \cdot \bm{B} \dH \\
	&+ \int_{\po} \Big( W(\chi, \E(\bm{u})) \bm{I} - (\nabla \bm{u})^T \big(  \C_{\nu} (\chi) \E(\pt \bm{u}) + W_{, \E} (\chi, \E(\bm{u}) ) \big)\Big)  \bm{n}_{\po} \cdot \bm{B} \dH\\ 
	&- \int_\Omega (\nabla \bm{u})^T  \nabla \cdot \big(  \C_{\nu} (\chi) \E(\pt \bm{u}) + W_{, \E} (\chi, \E(\bm{u}) ) \big)\cdot \bm{B}  \dx
	-\int_\Omega \tfrac{1}{2} \abs{\nabla^2 \bm{u}}^2\  \div \bm{B} \dx \\ 
	&+ \int_\Omega
	(\nabla^2 \bm u)_{ijk}\,(\nabla \bm u)_{ip}\,(\nabla^2 \bm B)_{pjk}
	\, dx
	+
	\int_\Omega
	(\nabla^2 \bm u)_{ijk}
	\Big(
	(\nabla^2 \bm u)_{ipk}\,(\nabla \bm B)_{jp}
	+
	(\nabla^2 \bm u)_{ijp}\,(\nabla \bm B)_{kp}
	\Big)
	\dx, 
\end{align*}
and it remains to rewrite the last three terms as integrals over $\Gamma$. Integration by parts immediately yields 
\begin{align*}\label{eq:jump_second_2}
	-&\int_\Omega \tfrac{1}{2} \abs{\nabla^2 \bm{u}}^2\  \div \bm{B} \dx\\
	&= \int_\Gamma \bjump{\tfrac{1}{2} \abs{\nabla^2 \bm{u}}^2} \bnu \cdot \bm{B} \dH 
	+ \int_{\po}  \tfrac{1}{2} \abs{\nabla^2 \bm{u}}^2 \bm{n}_{\po} \cdot \bm{B} \dH  
	+ \int_\Omega \tfrac{1}{2} \nabla \abs{\nabla^2 \bm{u}}^2 \cdot  \bm{B} \dx, 
	\numberthis
\end{align*} 
and a quick computation further shows 
\begin{equation}\label{eq:jump_second_4}
	\int_\Omega \tfrac{1}{2} \nabla \abs{\nabla^2 \bm{u}}^2 \cdot  \bm{B} \dx 
	= \int_\Omega  \big( (\nabla^3 \bm{u})_{ijkp}  (\nabla^2 \bm{u})_{ijk} \big)_{p= 1}^d\cdot \bm{B} \dx. 
\end{equation}
We continue with the first term in the very last integral in \eqref{eq:jump_second_5} and calculate \\ 
\begin{align*}
	\int_\Omega &(\nabla^2 \bm u)_{ijk}
	(\nabla^2 \bm u)_{ipk}\,(\nabla \bm B)_{jp} \dx 
	\numberthis \label{eq:jump_second_1}\\ 
	& =
	-\int_\Gamma \bjump{ \big( (\nabla^2 \bm{u})_{ijk}  (\nabla^2 \bm{u} \bnu)_{ik}\big)_{j = 1}^d  } \cdot \bm{B} \dH 
	-\int_{\po} \Big(  (\nabla^2 \bm{u})_{ijk}  (\nabla^2 \bm{u})_{ipk}  \bm{n}_{\po} \Big)_{j = 1}^d \cdot \bm{B} \dH 
	\\ 
	&\, \quad - \int_\Omega \Big( (\nabla^3 \bm{u})_{ijkp}  (\nabla^2 \bm{u})_{ijk} \Big)_{p = 1}^d \cdot \bm{B} +  \Big( (\nabla^2 \bm{u})_{ijk} (\div \nabla^2\bm{u} )_{ik} \Big)_{j= 1}^d \cdot \bm{B}\dx. 
	\numberthis \label{eq:compu_gradient_nabla_2}
\end{align*}
\par \noindent
We proceed with the first integral in the last line of \eqref{eq:jump_second_5} and compute 
\begin{align*}
	&\int_\Omega (\nabla^2 \bm u)_{ijk}\,(\nabla \bm u)_{ip}\,(\nabla^2 \bm B)_{pjk} \, dx\\
	& \ \  = 
	\begin{aligned}[t]
		&-\int_\Gamma \bjump{  (\nabla \bm{u})^T (\nabla^2 \bm{u} \bnu )} \colon \nabla \bm{B} \dH 
		-\int_{\po}  \big(  (\nabla \bm{u})^T (\nabla^2 \bm{u} \bm{n}_{\po}) \big) \colon \nabla \bm{B} \dH 
		\\ 
		&- \int_\Omega \Big( (\div \nabla^2 \bm{u})_{ik}   (\nabla \bm{u})_{ip}+ (\nabla^2 \bm{u})_{ijk}  (\nabla^2 \bm{u})_{ijp} \Big)  (\nabla \bm{B})_{pk} \dx. 
	\end{aligned}
\end{align*}
Let us take a closer look at the last term, and observe that the matrix $\big((\nabla^2 \bm{u})_{ijk}  (\nabla^2 \bm{u})_{ijp} \big)_{k,p= 1}^d$ is symmetric due to the symmetry of $\nabla^2 \bm{u}$ with respect to the third index. Hence, we can transpose the matrix $\nabla \bm{B}$ and obtain 
\begin{align*}
	(\nabla^2 \bm{u})_{ijk}  (\nabla^2 \bm{u})_{ijp}  (\nabla \bm{B})_{pk}
	= (\nabla^2 \bm{u})_{ijk} (\nabla^2 \bm{u})_{ijp}  (\nabla \bm{B})_{kp} . 
\end{align*}
Moreover, it holds that 
\begin{align*}
	- &\int_\Omega (\div \nabla^2 \bm{u})_{ik}   (\nabla \bm{u})_{ip} (\nabla \bm{B})_{pk} \dx\\
	& =  \begin{aligned}[t]
		&\int_\Gamma  \bjump{ (\nabla \bm{u})^T (\div \nabla^2 \bm{u} \bnu)} \cdot  \bm{B} \dH 
		+ \int_{\po}  \big( (\nabla \bm{u})^T (\div \nabla^2 \bm{u}\bm{n}_{\po})  \big) \cdot \  \bm{B} \dH \\ 
		&+ \int_\Omega \big(  (\nabla \bm{u})^T  \div^2 \nabla^2\bm{u}  \big) \cdot \bm{B} 
		+ \big((\nabla^2 \bm{u})_{ijk} (\div \nabla^2 \bm{u})_{ik} \big)_{j = 1}^d \cdot \bm{B} \dx. 
	\end{aligned}
\end{align*}
Thus, we arrive at 
\begin{align*}
	&\int_\Omega(\nabla^2 \bm u)_{ijk}\,(\nabla \bm u)_{ip}\,(\nabla^2 \bm B)_{pjk} \dx\numberthis \label{eq:jump_second_3} \\ 
	&\quad= \begin{aligned}[t]
		&-\int_\Gamma \bjump{  (\nabla \bm{u})^T (\nabla^2 \bm{u} \bnu )} \colon \nabla \bm{B} \dH -\int_{\po}  \big(  (\nabla \bm{u})^T (\nabla^2 \bm{u} \bm{n}_{\po}) \big) \colon \nabla \bm{B} \dH \\ 
		&+ \int_\Gamma  \bjump{ (\nabla \bm{u})^T (\div \nabla^2 \bm{u} \bnu)} \cdot  \bm{B} \dH
		+ \int_{\po}  \big( (\nabla \bm{u})^T (\div \nabla^2 \bm{u}\bm{n}_{\po})  \big) \cdot \  \bm{B} \dH\\ 
		&+ \int_\Omega \big(  (\nabla \bm{u})^T  \div^2 \nabla^2\bm{u}  \big) \cdot \bm{B}
		+ \big((\nabla^2 \bm{u})_{ijk} (\div \nabla^2 \bm{u})_{ik} \big)_{j = 1}^d \cdot \bm{B}  \dx 
		- \int_\Omega (\nabla^2 \bm{u})_{ijk}  (\nabla^2 \bm{u})_{ijp}  (\nabla \bm{B})_{kp} \dx. 
	\end{aligned}
\end{align*}
\par 
Combining \eqref{eq:jump_second_2}, \eqref{eq:jump_second_4} and \eqref{eq:jump_second_1} with \eqref{eq:jump_second_3} leads to 
\begin{align*}
	-&\int_\Omega \tfrac{1}{2} \abs{\nabla^2 \bm{u}}^2\  \div \bm{B} \dx \\ 
	&+ \int_\Omega
	(\nabla^2 \bm u)_{ijk}\,(\nabla \bm u)_{ip}\,(\nabla^2 \bm B)_{pjk}
	\, dx
	+
	\int_\Omega
	(\nabla^2 \bm u)_{ijk}
	\Big(
	(\nabla^2 \bm u)_{ipk}\,(\nabla \bm B)_{jp}
	+
	(\nabla^2 \bm u)_{ijp}\,(\nabla \bm B)_{kp}
	\Big)
	\dx\\ 
	\pagebreak[0]
	&= \begin{aligned}[t]
		&\int_\Gamma \bjump{\tfrac{1}{2} \abs{\nabla^2 \bm{u}}^2} \bnu \cdot \bm{B} \dH -\int_\Gamma \bjump{ \big( (\nabla^2 \bm{u})_{ijk}  (\nabla^2 \bm{u} \bnu)_{ik}\big)_{j = 1}^d  } \cdot \bm{B} \dH  
		\\ 
		&-\int_\Gamma \bjump{  (\nabla \bm{u})^T (\nabla^2 \bm{u} \bnu )} \colon \nabla \bm{B} \dH 
		+ \int_\Gamma  \bjump{ (\nabla \bm{u})^T (\div \nabla^2 \bm{u} \bnu)} \cdot  \bm{B} \dH \\ 
		&+  \int_{\po}  \tfrac{1}{2} \abs{\nabla^2 \bm{u}}^2 \bm{n}_{\po} \cdot \bm{B} \dH  
		-\int_{\po} \Big(  (\nabla^2 \bm{u})_{ijk}  (\nabla^2 \bm{u})_{ipk}  \bm{n}_{\po} \Big)_{j = 1}^d \cdot \bm{B} \dH
		\\ 
		&-\int_{\po}  \big(  (\nabla \bm{u})^T (\nabla^2 \bm{u} \bm{n}_{\po}) \big) \colon \nabla \bm{B} \dH
		+ \int_{\po}  \big( (\nabla \bm{u})^T (\div \nabla^2 \bm{u}\bm{n}_{\po})  \big) \cdot \  \bm{B} \dH \\  
		&+\int_\Omega (\nabla \bm{u})^T  \div^2 \nabla^2\bm{u}  \big) \cdot \bm{B} 
		\dx,
	\end{aligned}
\end{align*}
which we insert into \eqref{eq:jump_second_5}, and by invoking \eqref{eq:second_el_bulk} arrive at
\begin{align*}
	&\int_\Gamma  (w + \lambda) \bnu \cdot \bm{B} \dH  \\ 
	&=\int_\Gamma c_0 \varkappa \bnu \cdot \bm{B} \dH  -  c_0 \int_{ \partial \Gamma} 
	(\bnu \cdot  \bm{n}_{\partial\Omega} - \cos \alpha) ( \bm{\tau}_{\po }  \cdot \bm{B} ) \, d \mathcal{H}^{d-2}  \\ 
	&+\int_\Gamma \bjump{ \big( \tfrac{1}{2} \abs{\nabla^2 \bm{u}}^2 + W(\chi, \E(\bm{u})) \big) \bm{I} - (\nabla \bm{u})^T \big(  \C_{\nu} (\chi) \E(\pt \bm{u}) + W_{, \E} (\chi, \E(\bm{u}) ) - \div \nabla^2 \bm{u}\big) } \bnu \cdot \bm{B} \dH\\ 
	&+\int_{\po} \Big(\big(\tfrac{1}{2} \abs{\nabla^2 \bm{u}}^2 + W(\chi, \E(\bm{u}))\big) \bm{I} - (\nabla \bm{u})^T \big(  \C_{\nu} (\chi) \E(\pt \bm{u}) + W_{, \E} (\chi, \E(\bm{u}) ) - \div \nabla^2 \bm{u}\big) \Big) \bm{n}_{\po } \cdot \bm{B} \dH\\ 
	&-\int_\Gamma \bjump{ \big( (\nabla^2 \bm{u})_{ijk}  (\nabla^2 \bm{u} \bnu)_{ik}\big)_{j = 1}^d  } \cdot \bm{B} \dH 
	-\int_{\po} \Big(  (\nabla^2 \bm{u})_{ijk}  (\nabla^2 \bm{u})_{ipk}  \bm{n}_{\po} \Big)_{j = 1}^d \cdot \bm{B} \dH
	\\ 
	&-\int_\Gamma \bjump{  (\nabla \bm{u})^T (\nabla^2 \bm{u} \bnu )} \colon \nabla \bm{B} \dH 
	-\int_{\po}  \big(  (\nabla \bm{u})^T (\nabla^2 \bm{u} \bm{n}_{\po}) \big) \colon \nabla \bm{B} \dH. 
\end{align*}
Since all remaining terms are surface integrals over the interface, we can now proceed to derive the jump conditions between the two phases.
After splitting $\nabla \bm{B} = \nabla_\Gamma \bm{B} + \nabla \bm{B} \bnu \otimes \bnu$, one obtains
\begin{align*}
	-&\int_\Gamma \bjump{  (\nabla \bm{u})^T (\nabla^2 \bm{u} \bnu )} \colon \nabla \bm{B} \dH 	\\ 
	&= -\int_\Gamma \bjump{  (\nabla \bm{u})^T (\nabla^2 \bm{u} \bnu )} \colon \nabla_\Gamma \bm{B} \dH 	
	-\int_\Gamma  \bjump{  (\nabla \bm{u})^T \nabla^2 \bm{u} (\bnu \otimes \bnu )}  \cdot \nabla \bm{B}\bnu   \dH	\\ 
	& = \begin{aligned}[t]
		&\int_\Gamma  \div_{\Gamma} \Big(  P_{\Gamma}  \bjump{    (\nabla \bm{u})^T (\nabla^2 \bm{u} \bnu ) } \Big)  \cdot \bm{B} 
		- \int_{\partial \Gamma}  P_{\Gamma}  \bjump{   (\nabla \bm{u})^T (\nabla^2 \bm{u} \bnu ) } \bm{\tau}_\Gamma   \cdot \bm{B} \, d \mathcal{H}^{d-2}\\
		&-\int_\Gamma  \bjump{  (\nabla \bm{u})^T \nabla^2 \bm{u} (\bnu \otimes \bnu )}  \cdot \nabla \bm{B}\bnu   \dH
		.
	\end{aligned}
\end{align*} 
As noted before $\jump{\nabla^2 \bm{u} (\bnu \otimes \bnu)} =\bm{0}$, see \eqref{eq:interface_sq}; likewise, the gradient $\nabla \bm{u} = \nabla \bm{u} \bnu + \nabla_{ \Gamma} \bm{u}$ cannot jump across the interface either, cf.\ Section \ref{sec:kinetics}. Hence the last integral is zero. Further recall that the projection is defined as $P_\Gamma \colon \bm{A} \mapsto \bm{A} - \bm{A} \bnu \otimes \bnu$. Then we compute 
\begin{align*}
	P_{\Gamma}  \bjump{   (\nabla \bm{u})^T (\nabla^2 \bm{u} \bnu ) } 
	&= 	\bjump{    (\nabla \bm{u})^T (\nabla^2 \bm{u} \bnu )}  - \bjump{  (\nabla \bm{u})^T \big( \nabla^2 \bm{u}  (\bnu \otimes \bnu )\big) \otimes \bnu }  = \bjump{    (\nabla \bm{u})^T (\nabla^2 \bm{u} \bnu )}, 
\end{align*}
so we can neglect the projection in this term. Thus, the middle term simplifies to 
\begin{equation*}
	\int_{\partial \Gamma}  P_{\Gamma} \bjump{(\nabla \bm{u})^T (\nabla^2 \bm{u} \bnu ) }\tau_\Gamma    \cdot \bm{B} \, d \mathcal{H}^{d-2} =  \int_{\partial \Gamma}   \bjump{(\nabla \bm{u})^T (\nabla^2 \bm{u} \bnu ) }\bm{\tau}_\Gamma    \cdot \bm{B} \, d \mathcal{H}^{d-2}. 
\end{equation*}
Since we have just shown that we can ignore the tangential projection on $\bjump{ (\nabla \bm{u})^T (\nabla^2 \bm{u} \bnu )}$, linearity of the divergence entails 
\begin{align*}
	\div_{\Gamma} \Big(  P_{\Gamma}  \bjump{    (\nabla \bm{u})^T (\nabla^2 \bm{u} \bnu )  } \Big) 
	&= 	\Bjump{ \big((\nabla_\Gamma \nabla \bm{u})_{ijk}  (\nabla^2 \bm{u} \bnu)_{ik}  \big)_{j= 1}^d+   (\nabla \bm{u})^T  \div_\Gamma (\nabla^2 \bm{u} \bnu )\big) }. 
\end{align*}
Analogous arguments for the terms on $\po$ then lead to
\begin{align*}
	\int_\Gamma  & (w + \lambda) \bnu \cdot \bm{B} \dH  \\ 
	=&\int_\Gamma c_0 \varkappa \bnu \cdot \bm{B} \dH 
	-  c_0 \int_{ \partial \Gamma} 
	(\bnu \cdot  \bm{n}_{\partial\Omega} - \cos \alpha) ( \bm{\tau}_{\po }  \cdot \bm{B} ) \, d \mathcal{H}^{d-2} \\ 
	&+\int_\Gamma \bjump{\tfrac{1}{2} \abs{\nabla^2 \bm{u}}^2 + W(\chi, \E(\bm{u}))  } \bnu \cdot \bm{B} \dH\\ 
	&-\int_\Gamma  (\nabla \bm{u})^T \Bjump{ \Big( \big(  \C_{\nu} (\chi) \E(\pt \bm{u}) + W_{, \E} (\chi, \E(\bm{u}) ) - \div \nabla^2 \bm{u}\big) \bnu -   \div_\Gamma (\nabla^2 \bm{u} \bnu)\Big) } \cdot \bm{B} \dH\\ 
	&-\int_{\po} (\nabla \bm{u})^T  \Big( \big(  \C_{\nu} (\chi) \E(\pt \bm{u}) + W_{, \E} (\chi, \E(\bm{u}) ) - \div \nabla^2 \bm{u}\big)\bm{n}_{\po } - \div_{\po} (\nabla^2 \bm{u} \bm{n}_{\po}) \Big)  \cdot \bm{B} \dH\\ 
	&+ \int_\Gamma \Bjump{  \big( ( \nabla_\Gamma \nabla \bm{u} - \nabla^2 \bm{u})_{ijk}   (\nabla^2 \bm{u} \bnu )_{ik} \big)_{j = 1}^d } \cdot \bm{B} \dH \\
	&+\int_{\po} \Big(   (\nabla_{\po} \nabla \bm{u} - \nabla^2 \bm{u})_{ijk}   (\nabla^2 \bm{u}\bm{n}_{\po} )_{ik} \Big)_{j = 1}^d  \cdot \bm{B} \dH 
	\\ 
	&
	- \int_{\partial \Gamma} (\nabla \bm{u})^T  \Big(\jump{ \nabla^2 \bm{u} \bm{n}_{\po}}\bm{\tau}_{\po}  + \jump{ \nabla^2 \bm{u} \bnu }\bm{\tau}_{\Gamma}  \Big) \cdot \bm{B} \, d\mathcal{H}^{d-2}
	\numberthis \label{eq:potential_derivation}
\end{align*}
for all $\bm{B} \in \bm{C}^2(\overline{\Omega})$ with $\bm{B}_{|\po} \cdot \bm{n}_{\po}  = 0$. By a quick computation, one further finds  that 
\begin{equation}\label{eq:sg_normal_boundary}
	\Big(   (\nabla_{\po} \nabla \bm{u} - \nabla^2 \bm{u})_{ijk}   (\nabla^2 \bm{u}\bm{n}_{\po} )_{ik} \Big)_{j = 1}^d  \cdot \bm{B}
	= - \big( (\nabla^2 \bm{u} \bm{n}_{\po}) \bm{B} \big) \cdot \big(  \nabla^2 \bm{u}  (\bm{n}_{\po} \otimes \bm{n}_{\po} )  \big). 
\end{equation}
Along with \eqref{eq:boundary_sg}, this implies 
\begin{align*}
	\int_{\po} &\Big(   (\nabla_{\po} \nabla \bm{u} - \nabla^2 \bm{u})_{ijk}   (\nabla^2 \bm{u}\bm{n}_{\po} )_{ik} \Big)_{j = 1}^d  \cdot \bm{B}  \dH \\ 
	&= \int_{\po} - (\nabla^2 \bm{u} \bm{n}_{\po}) \bm{B}  \cdot \big(  \nabla^2 \bm{u}(\bm{n}_{\po} \otimes \bm{n}_{\po} )  \big)  \dH  = 0. 
\end{align*}
Analogously, but relying on \eqref{eq:interface_sq}, we deduce 
\begin{align*}
	\int_\Gamma &\Bjump{  \big( ( \nabla_\Gamma \nabla \bm{u} - \nabla^2 \bm{u})_{ijk}   (\nabla^2 \bm{u} \bnu )_{ik} \big)_{j = 1}^d } \cdot \bm{B} \dH = \int_\Gamma \Bjump{ \nabla^2 \bm{u} \bnu}  \bm{B}  \cdot \big(  \nabla^2 \bm{u} (\bnu \otimes \bnu )  \big)  \dH = 0,
\end{align*}
since we have already shown in \eqref{eq:interface_sq} that the normal part of $\nabla^2 \bm{u}$ has no jump across the interface.
Together with \eqref{eq:normal_stress_jump}, \eqref{eq:boundary_normal_stress}, \eqref{eq:sg_conormal}, the identity \eqref{eq:potential_derivation} simplifies to
\begin{align*}
	\int_\Gamma   (w + \lambda) \bnu \cdot \bm{B} \dH  
	&+  c_0 \int_{ \partial \Gamma} 
	(\bnu \cdot  \bm{n}_{\partial\Omega} - \cos \alpha) ( \bm{\tau}_{\po }  \cdot \bm{B} ) \, d \mathcal{H}^{d-2} \\
	&=\int_\Gamma  \Big( c_0 \varkappa \bnu 
	+ \bjump{\tfrac{1}{2} \abs{\nabla^2 \bm{u}}^2 + W(\chi, \E(\bm{u}))  } \bnu \Big) \cdot \bm{B} \dH. 
\end{align*}
Suppose that $\mathcal{B} \in C^1_c(\Gamma)$ and denote by $\bm{\mathcal{B}} \in \bm{C}^1_c(\Omega, \R^n)$ a vector field satisfying $\bm{\mathcal{B}}(\bm{x}) = \mathcal{B}(\bm{x})\bnu$ for all $\bm{x} \in \Gamma$. Moreover, suppose that this extension is (in a neighborhood of $\Gamma$) constant in normal direction, i.e., $\nabla \bm{\mathcal{B}} \bnu = \bm{0}$. Since $\bm{\mathcal{B}}$ is an admissible test function, we deduce by the fundamental lemma of the calculus of variations the jump condition
\begin{alignat*}{1}
	\tilde{w} =
	c_0 \varkappa  
	+ \bjump{\tfrac{1}{2} \abs{\nabla^2 \bm{u}}^2 + W(\chi, \E(\bm{u}))  }  
\end{alignat*}
on $\Gamma$, where $\tilde{w} = w + \lambda$. Moreover, we further obtain that 
\begin{equation*}
	\cos \alpha = \bnu \cdot \bm{n}_{\po} \quad \textnormal{on } \partial \Gamma. 
\end{equation*}

\section{Implicit time discretization} \label{sec:mm}
\noindent
Before we start with the construction of solutions, let us recall some of the notation. We write
\begin{gather*}
	\indic \coloneqq \Big\{ \chi \in BV(\Omega; \{0, 1\}) \colon \int_\Omega \chi \dx = m_0 \Big\}
\end{gather*}
for the space of admissible characteristic functions, and 
\begin{equation*}
	\disp  \coloneqq \bm{H}^1_{\perp}(\Omega) \cap \bm{H}^2(\Omega)
\end{equation*}
for all admissible displacements. Combined, these define the state space $\states \coloneqq   \indic \times \disp$, that can be seen as a formal manifold. \\ 
Assuming that the surface tension coefficient $c_0$ is equal to $1$, we consider the energy functional 
\begin{equation*}
	\F \colon \mathcal{M} \rightarrow \R, 
	\quad 
	(\tilde{\chi}, \tilde{\bm{u}}) \mapsto \int_\Omega \abs{\nabla \tilde{\chi}} 
	+ \int_{\partial \Omega}  (\cos \alpha) \tilde{\chi} \dH
	+ \int_\Omega W(\tilde{\chi}, \E(\tilde{\bm{u}})) \dx 
	+ \int_\Omega \tfrac{1}{2} \abs{\nabla^2 \tilde{\bm{u}}}^2 \dx. 
\end{equation*}
Moreover, we introduce the inner product 
\begin{align*}
	(\chi, \tilde{\chi})_{H^{-1}_{(0)}} \coloneqq \int_\Omega \nabla v \cdot \nabla \tilde{v} \dx, \quad \chi, \tilde{\chi} \in H^{-1}_{(0)}(\Omega), 
\end{align*}
where $v$ (respectively $\tilde{v}$) is the unique solution to the elliptic boundary value problem
\begin{equation}\label{eq:distance_elliptic}
	\left\{ 
	\begin{aligned}
		\Delta v & = \chi && \textnormal{in } \Omega ,\\ 
		\nabla v \cdot \bm{n}_{\partial \Omega} &= 0 && \textnormal{on } \partial \Omega,\\ 
		\int_\Omega v &= 0. 
	\end{aligned}
	\right. 
\end{equation}
Using the Lax-Milgram theorem, it is easy to see that this is well-defined and therefore gives rise to a norm on the space $H^{-1}_{(0)}(\Omega)$. Since the solution operator for the elliptic problem above is nothing but the Riesz isomorphism to the associated Hilbert space $H^1_{(0)} (\Omega)$, this inner product actually induces the natural Hilbert space structure on $H^{-1}_{(0)}(\Omega)$. \\ 
Given any $\chi \in \indic$, we can further define an inner product on the space $\bm{H}^1_{\perp}(\Omega)$ by 
\begin{equation}\label{eq:H^1_norm_chi}
	(\bm{u}, \tilde{\bm{u}})_{\bm{H}^1_{\chi}} \coloneqq \int_\Omega \C_{\nu}(\chi) \E(\bm{u}) \colon \E(\tilde{\bm{u}}) \dx.  
\end{equation}
Recalling that $\C_{\nu}$ is uniformly positive definite and applying Korn's inequality, we find that the norm induced by this inner product is equivalent to the standard norm on the Hilbert space $\bm{H}^1_{\perp}(\Omega)$, see \cite[Sec.\ 6.15]{ciarlet2013linear}, \cite[Sec.\ 2.1]{kwak}. 
\par 
\medskip 
Given a fixed time interval $[0, T_*]$ and initial conditions $(\chi_0, \bm{u}_0) \in \states$, along with a fixed time step size $h = \tfrac{T_*}{n} > 0$, $n \in \N$, we want to inductively select approximate solutions $(\chi_n^h, \bm{u}_n^h)$ at times $nh$ for $k \in \N$ where $k \leq n$. To this end, set $(\chi_0^h, \bm{u}_0^h) = (\chi_0, \bm{u}_0)$ and choose 
\begin{equation}\label{eq:mm_problem}
	(\chi_k^h, \bm{u}_k^h) \in \underset{ (\tilde{\chi}, \tilde{\bm{u}}) \in \states}{\textnormal{argmin}} 
	\Big(
	\F(\tilde{\chi}, \tilde{\bm{u}}) 
	+ \tfrac{1}{2h}\norm{ \chi_{k-1}^h - \tilde{\chi} }^2_{H^{-1}_{(0)}}  
	+\tfrac{1}{2h} \norm{\bm{u}_{k-1}^h - \tilde{\bm{u}} }^2_{\bm{H}^1_{\chi^h_{k-1}}}
	\Big).
\end{equation}

\begin{proposition}\label{prop:minimizer}
	For any admissible $h > 0$ and all admissible $k \in \N$, there exists a minimizer of 
	\begin{equation*}
		\F^h_k \colon \mathcal{M} \rightarrow \R, 
		\quad 
		(\tilde{\chi}, \tilde{\bm{u}}) \mapsto \F(\tilde{\chi}, \tilde{\bm{u}}) 
		+ \tfrac{1}{2h}\norm{ \chi_{k-1}^h - \tilde{\chi} }^2_{H^{-1}_{(0)}}  
		+\tfrac{1}{2h} \norm{\bm{u}_{k-1}^h - \tilde{\bm{u}} }^2_{\bm{H}^1_{\chi^h_{k-1}}}. 
	\end{equation*}
\end{proposition}

This assertion is an immediate consequence of the direct method of the calculus of variations.

\subsection{Interpolants}

We proceed by constructing approximate solutions that are defined on the whole interval $[0, T_*]$ and not limited to discrete points.\\ 
\textit{Piecewise constant} interpolants will be denoted by $f^h$ (where $f$ is to be substituted by either $\chi$ or $\bm{u}$), which are defined as 
\begin{equation*}
	f^h(t) \coloneqq f ^h_{\lfloor \tfrac{t}{h} \rfloor} \quad \textnormal{for all } t \in [0, T]. 
\end{equation*} 
Moreover, we write $\hat{f}^h$ for \textit{piecewise linear} interpolants 
\begin{equation}\label{eq:linear_interpolants}
	\hat{f}^h (t) \coloneqq \tfrac{kh -t }{h} f^h_{k -1} + \tfrac{t- (k-1)h}{h} f^h_k \quad \textnormal{for all } t \in [(k-1)h, kh]. 
\end{equation}
Lastly, we use the notation $(\bar{\chi}^h, \bar{\bm{u}}^h)$ for the (variational) De Giorgi interpolants, which are defined by 
\begin{align*}
	(\bar{\chi}^h, \bar{\bm{u}}^h)( (k-1) h) &\coloneqq (\bar{\chi}^h((k-1)h), \bar{\bm{u}}^h((k-1)h))  =  ({\chi}^h_{k-1}, {\bm{u}}^h_{k-1}), \quad k \in \N; \\
	(\bar{\chi}_k^h, \bar{\bm{u}}_k^h)(t) &\in \underset{ (\tilde{\chi}, \tilde{\bm{u}}) \in \states}{\textnormal{argmin}} 
	\Big(
	\F(\tilde{\chi}, \tilde{\bm{u}}) 
	+ \tfrac{1}{2 (t - (k-1)h )}\norm{\tilde{\chi} - \chi_{k-1}^h   }^2_{H^{-1}_{(0)}}  
	+\tfrac{1}{2 (t - (k-1)h )} \norm{\tilde{\bm{u}}  - \bm{u}_{k-1}^h  }^2_{\bm{H}^1_{\chi^h_{k-1}}}
	\Big), \\
	& \quad \textnormal{for all } t \in ((k-1) h, kh). \numberthis \label{eq:degiorgi}
\end{align*}
The existence of such an interpolant follows immediately from Proposition \ref{prop:minimizer}. This variational approach will allow us to derive an optimal energy dissipation inequality.  \\ 
Concerning measureability of the De Giorgi interpolants, we refer to \cite[App.\ B.1]{abels2025weaksolutionssharpinterface}, where a similar assertion is proven based on the Kuratowski--Ryll--Nardzewski Measurable Selection Theorem.
\par 
\medskip 
For rigorous limit passage, especially in the energy, the $BV$-formulation will not suffice to capture the evolution in sufficient detail. This problem can be remedied by replacing the current perimeter energy with a more robust formulation. Specifically, we define the oriented varifolds 
\begin{equation}\label{eq:def_mu_Omega}
	\mu^{h, \Omega}_t \coloneqq \abs{\nabla \bar{\chi}^h(t, \cdot)} \llcorner \Omega  \otimes \big( \delta_{\tfrac{\nabla \bar{\chi}^h(t, \bm{x})}{\abs{\nabla \bar{\chi}^h(t, \bm{x})}}} \big)_{\bm{x} \in \Omega}, 
\end{equation}
where $\delta_{\bm{s}}$ denotes the Dirac measure at a point $\bm{s}\in \mathbb{S}^{d-1}$, and 
\begin{equation}\label{eq:def_mu_boundary}
	\mu^{h, \partial \Omega}_t \coloneqq \bar{\chi}^h (t, \cdot) \llcorner \partial \Omega  \otimes \big( \delta_{\bm{n}_{\pO}} \big)_{\bm{x} \in \pO}, 
\end{equation}
with $\bm{n}_{\pO}$ being the inner normal on $\partial \Omega$ with respect to $\Omega$. Finally, we set 
\begin{equation}\label{eq:def_mu_new}
	\mu^{h}_t   \coloneqq   \mu^{h, \Omega}_t  + (\cos \alpha) \mu^{h, \partial \Omega}_t  \in M(\overline{\Omega} \times \mathbb{S}^{d-1}), 
\end{equation}
obtaining an oriented varifold on $\overline{\Omega}$ whose total mass is naturally associated with $\bar{\chi}^h(t, \cdot)$. In particular, it holds that 
\begin{equation}\label{eq:perimeter_varifold}
	\F_{per}(\bar{\chi}) +  \F_{cap}(\bar{\chi}) = \int_\Omega \abs{\nabla \bar{\chi}} + \int_{\pO} (\cos \alpha) \bar{\chi} \dH = \int_{\overline{\Omega} \times \mathbb{S}^{d-1}} d\mu_{\bar{\chi}} (\bm{x}, \bm{s}). 
\end{equation}

\subsection{Approximate energy dissipation inequality}

For any admissible $k \in \N$, we can use the selected minimizer of $\F^h_{k-1}$ as a competitor in the optimization problem associated to $\F_k^h$, and obtain as an immediate consequence of a simple telescoping argument that 
\begin{align}\label{eq:a_piori_trivial}
	\F(\chi^h, \bm{u}^h) (T) 
	+ \tfrac{1}{2h^2}\int_0^{T_*}  \norm{\chi^h(t+h) - \chi(t)}_{H^{-1}_{(0)}}^2 
	+ \norm{\bm{u}^h(t+h) - \bm{u}^h(t)}_{\bm{H}^1_{\chi^h}}^2 \dt 
	\leq 
	\F(\chi^h, \bm{u}^h)(0) 
\end{align}
for all $T \in \N h$. However, this energy dissipation inequality is not optimal and the goal of this section is to show the following, improved, assertion.
\par 
\medskip 
\noindent 
\textit{Claim:} For $\tau, \kappa, s, T$ such that $0 < s < \kappa < \tau < T < T_*$, and any admissible $0 < h < \min\{T- \tau, \kappa - s \}$ it holds that 
\begin{align*}
	&\F(\bar{\chi}^h, \bar{\bm{u}}^h) (T)\\
	&\quad + \frac{1}{2} \int_\kappa^\tau  \begin{aligned}[t]
		&\norm{\pt \hat{\chi}^h}_{H^{-1}_{(0)}}^2 
		+  \norm{ \pt \hat{\bm{u}}^h}^2_{\bm{H}^1_{\bar{\chi}^h ((\lfloor \sfrac{t}{h} \rfloor)h) }} 
		+ \norm{ \tfrac{ \bar{\chi}^h- \bar{\chi}^h ( \lfloor \sfrac{t}{h} \rfloor h) }{t - \lfloor \sfrac{t}{h} \rfloor h } }^2_{H^{-1}_{(0)}} 
		+ \norm{ \tfrac{\bar{\bm{u}}^h   -  \bar{\bm{u}}^h ( \lfloor \sfrac{t}{h} \rfloor h)}{ {t - \lfloor \sfrac{t}{h} \rfloor h } }}^2_{\bm{H}^1_{\bar{\chi}^h ((\lfloor \sfrac{t}{h} \rfloor)h) }} \dt 
	\end{aligned}\\ 
	&\ \  \leq \F(\bar{\chi}^h, \bar{\bm{u}}^h) (s) \leq \F(\bar{\chi}^h, \bar{\bm{u}}^h) (0)
	\numberthis \label{eq:dissipation_approximate}
\end{align*}

While this assertion can be deduced from general results on minimizing movements, cf.\ \cite[Thm.\ 3.1.4]{ambrosio}, we also include a proof here that follows the approach in \cite{hensel_stinson, laux_chambolle}. We begin with the observation that 
\begin{equation}\label{eq:energy_decreasing}
	\F(\bar{\chi}^h, \bar{\bm{u}}^h) (t) \leq \F(\bar{\chi}^h, \bar{\bm{u}}^h) ((k-1) h) 
	\quad \textnormal{for all } t \in ((k-1)h, kh), 
\end{equation}
which is an immediate consequence of the minimality of the interpolant \eqref{eq:degiorgi} at time $t$. Before turning to the continuous inequality \eqref{eq:dissipation_approximate}, we will establish the discrete version  
\begin{align*}
	&\F(\bar{\chi}^h, \bar{\bm{u}}^h) (kh)
	\begin{aligned}[t]
		&+ \tfrac{1}{2h}  \norm{ \chi^h_k - \chi^h_{k-1} }_{H^{-1}_{(0)}}^2 
		+  \tfrac{1}{2h} \norm{ \bm{u}^h_k - \bm{u}^h_{k-1}}^2_{\bm{H}^1_{\bar{\chi}^h ((k-1)h)}} 	\\ 
		& + \tfrac{1}{2} \int_{(k-1)h}^{kh}
		\norm{ \tfrac{ \bar{\chi}^h(t) - \bar{\chi}^h ((k-1)h }{t - (k-1)h } }^2_{H^{-1}_{(0)}} 
		+ \norm{ \tfrac{\bar{\bm{u}}^h (t)  -  \bar{\bm{u}}^h ((k-1)h)  }{ {t -(k-1)h } }}^2_{\bm{H}^1_{\bar{\chi}^h ((k-1)h)}}  \dt
	\end{aligned}
	\\ 
	&\leq \F(\bar{\chi}^h, \bar{\bm{u}}^h) ((k-1)h) \leq \F(\bar{\chi}^h, \bar{\bm{u}}^h) (0)
	\numberthis \label{eq:dissipation_approximate_disc}
\end{align*}
for all admissible $k \in \N$. Without loss of generality, it suffices to derive this relation for the case $k = 1$; in subsequent arguments we will also neglect the superscript $h$ as no confusion seems likely. \\ 
Consider the mapping 
\begin{equation*}
	f(t) \coloneqq \F(\bar{\chi}, \bar{\bm{u}}) (t) 
	+ \tfrac{1}{2t} \norm{\bar{\chi} (t) - \chi_0}_{H^{-1}_{(0)}}^2
	+ \tfrac{1}{2t} \norm{\bar{\bm{u}} (t) - \bm{u}_0}^2_{\bm{H}^1_{\chi_0}}, 
	\quad t \in (0, h); 
\end{equation*}
We will establish that $f$ is locally Lipschitz in $(0, h)$ with 
\begin{equation}\label{eq:f_derivative}
	\tfrac{d}{dt}  f(t) = - \tfrac{1}{2t^2} \norm{\bar{\chi} (t) - \chi_0}_{H^{-1}_{(0)}}^2
	- \tfrac{1}{2t^2} \norm{\bar{\bm{u}} (t) - \bm{u}_0}^2_{\bm{H}^1_{\chi_0}}
	\quad \textnormal{for a.e. } t \in (0, h). 
\end{equation}
To this end, we prove the following properties 
\begin{enumerate}[noitemsep, label=(f.\arabic*)]
	\item $(0, h] \ni t \mapsto \norm{\bar{\chi} (t) - \chi_0}_{H^{-1}_{(0)}}^2
	+  \norm{\bar{\bm{u}} (t) - \bm{u}_0}^2_{\bm{H}^1_{\chi_0}}$ is non-decreasing; \label{eq:f_nd} 
	\item  $(0, h] \ni t \mapsto f(t)$ is non-increasing. \label{eq:f_ni}
\end{enumerate}
As a first step towards this goal, we deduce for all $0 < s < t \leq h$ by using minimality \eqref{eq:degiorgi} in $s$, adding zero, and finally exploiting minimality in $t$,  that  
\begin{align*}
	f(s) &=  \F(\bar{\chi}, \bar{\bm{u}}) (s) 
	+ \tfrac{1}{2s} \norm{\bar{\chi} (s) - \chi_0}_{H^{-1}_{(0)}}^2
	+ \tfrac{1}{2s} \norm{\bar{\bm{u}} (s) - \bm{u}_0}^2_{\bm{H}^1_{\chi_0}}\\ 
	& \leq \F(\bar{\chi}, \bar{\bm{u}}) (t) 
	+ \tfrac{1}{2s} \norm{\bar{\chi} (t) - \chi_0}_{H^{-1}_{(0)}}^2
	+ \tfrac{1}{2s} \norm{\bar{\bm{u}} (t) - \bm{u}_0}^2_{\bm{H}^1_{\chi_0}}\\ 
	& \leq \begin{aligned}[t]
		&\F(\bar{\chi}, \bar{\bm{u}}) (t) 
		+ \tfrac{1}{2t} \norm{\bar{\chi} (t) - \chi_0}_{H^{-1}_{(0)}}^2
		+ \tfrac{1}{2t} \norm{\bar{\bm{u}} (t) - \bm{u}_0}^2_{\bm{H}^1_{\chi_0}}\\ 
		&+\big( \tfrac{1}{2s} - \tfrac{1}{2t} \big) \Big[ \norm{\bar{\chi} (t) - \chi_0}_{H^{-1}_{(0)}}^2
		+ \norm{\bar{\bm{u}} (t) - \bm{u}_0}^2_{\bm{H}^1_{\chi_0}} \Big]
	\end{aligned}\\ 
	& \leq \begin{aligned}[t]
		&\F(\bar{\chi}, \bar{\bm{u}}) (s) 
		+ \tfrac{1}{2t} \norm{\bar{\chi} (s) - \chi_0}_{H^{-1}_{(0)}}^2
		+ \tfrac{1}{2t} \norm{\bar{\bm{u}} (s) - \bm{u}_0}^2_{\bm{H}^1_{\chi_0}}\\ 
		&+\big( \tfrac{1}{2s} - \tfrac{1}{2t} \big)
		\Big[ \norm{\bar{\chi} (t) - \chi_0}_{H^{-1}_{(0)}}^2
		+ \norm{\bar{\bm{u}} (t) - \bm{u}_0}^2_{\bm{H}^1_{\chi_0}} \Big]. 
	\end{aligned}
\end{align*}
Rewriting the inequality above yields
\begin{align*}
	\norm{\bar{\chi} (s) - \chi_0}_{H^{-1}_{(0)}}^2
	+ \norm{\bar{\bm{u}} (s) - \bm{u}_0}^2_{\bm{H}^1_{\chi_0}}   
	\leq \norm{\bar{\chi} (t) - \chi_0}_{H^{-1}_{(0)}}^2
	+ \norm{\bar{\bm{u}} (t) - \bm{u}_0}^2_{\bm{H}^1_{\chi_0}}, 
\end{align*}
which already implies \ref{eq:f_nd}.\par  
Moreover, by applying the minimality \eqref{eq:degiorgi} in $s$, it follows that 
\begin{equation*}
	f(t) - f(s) \geq
	\tfrac{s-t}{2ts} \Big[  \norm{\bar{\chi} (t) - \chi_0}_{H^{-1}_{(0)}}^2
	+ \norm{\bar{\bm{u}} (t) - \bm{u}_0}^2_{\bm{H}^1_{\chi_0}}  \Big], 
\end{equation*}
which entails 
\begin{equation}\label{eq:f_2}
	\tfrac{f(t) - f(s)}{t-s} \geq 
	- \tfrac{1}{2ts} \Big[  \norm{\bar{\chi} (t) - \chi_0}_{H^{-1}_{(0)}}^2
	+ \norm{\bar{\bm{u}} (t) - \bm{u}_0}^2_{\bm{H}^1_{\chi_0}}  \Big]
\end{equation}
for $0 < s < t \leq h$. 
Similarly, using minimality in $t$ leads to the inverse inequality 
\begin{equation*}
	f(t) - f(s) 
	\leq \tfrac{s-t}{2ts} \Big[  \norm{\bar{\chi} (s) - \chi_0}_{H^{-1}_{(0)}}^2
	+ \norm{\bar{\bm{u}} (s) - \bm{u}_0}^2_{\bm{H}^1_{\chi_0}}  \Big], 
\end{equation*}
and thus for all $0 < s < t \leq h$ 
\begin{equation}\label{eq:f_3}
	\tfrac{f(t) - f(s)}{t-s} 
	\leq - \tfrac{1}{2ts} \Big[  \norm{\bar{\chi} (s) - \chi_0}_{H^{-1}_{(0)}}^2
	+ \norm{\bar{\bm{u}} (s) - \bm{u}_0}^2_{\bm{H}^1_{\chi_0}}  \Big]
	\leq 0. 
\end{equation}
Since $t- s > 0$, this immediately implies \ref{eq:f_ni}. For any $0 < s < t \leq h$, \ref{eq:f_ni} and \eqref{eq:f_3} further imply
\begin{equation}\label{eq:f_4}
	\tfrac{ \abs{f(t) - f(s)} }{\abs{t-s}} = \tfrac{f(s) - f(t)}{t-s}
	\leq \tfrac{1}{2ts} \Big[  \norm{\bar{\chi} (s) - \chi_0}_{H^{-1}_{(0)}}^2
	+ \norm{\bar{\bm{u}} (s) - \bm{u}_0}^2_{\bm{H}^1_{\chi_0}}  \Big] , 
\end{equation}
and along with \ref{eq:f_nd}, we conclude that $f$ is continuous and also locally Lipschitz. Thus $f$ is differentiable almost everywhere and we deduce \eqref{eq:f_derivative} from \eqref{eq:f_4} combined with the lower bound \eqref{eq:f_2} by sending $s \nearrow t$ and $t \searrow s$. Here, we also note that the mapping in \ref{eq:f_nd} is monotone and therefore has at most countably many discontinuities.\\ 
Integrating \eqref{eq:f_derivative} over $(s,t)$ further yields 
\begin{equation*}
	f(t) - f(s) 
	= - \int_s^t \tfrac{1}{2\tau^2} \norm{\bar{\chi} (\tau) - \chi_0}_{H^{-1}_{(0)}}^2
	+ \tfrac{1}{2\tau^2} \norm{\bar{\bm{u}} (\tau) - \bm{u}_0}^2_{\bm{H}^1_{\chi_0}} \, d\tau, 
\end{equation*}
and observing that minimality of \eqref{eq:degiorgi} at $s$ entails $f(s) \leq \F(\chi_0, \bm{u}_0)$, along with \ref{eq:f_ni}, leads to 
\begin{equation*}
	f(h) + \int_s^t \tfrac{1}{2\tau^2} \norm{\bar{\chi} (\tau) - \chi_0}_{H^{-1}_{(0)}}^2
	+ \tfrac{1}{2\tau^2} \norm{\bar{\bm{u}} (\tau) - \bm{u}_0}^2_{\bm{H}^1_{\chi_0}} \, d\tau \leq  \F(\chi_0, \bm{u}_0). 
\end{equation*}
By inserting the definition of $f$ and sending $s \searrow 0$ while $t \nearrow h$, we arrive at 
\begin{align*}
	\F(\bar{\chi}, \bar{\bm{u}}) (h) 
	&+ \tfrac{1}{2h} \norm{\bar{\chi} (h) - \chi_0}_{H^{-1}_{(0)}}^2
	+ \tfrac{1}{2h} \norm{\bar{\bm{u}} (h) - \bm{u}_0}^2_{\bm{H}^1_{\chi_0}} \\ 
	&+ \int_0^h \tfrac{1}{2\tau^2} \norm{\bar{\chi} (\tau) - \chi_0}_{H^{-1}_{(0)}}^2
	+ \tfrac{1}{2\tau^2} \norm{\bar{\bm{u}} (\tau) - \bm{u}_0}^2_{\bm{H}^1_{\chi_0}} \, d\tau \leq  \F(\chi_0, \bm{u}_0). 
\end{align*}
As the same is true for any admissible $k \in \N$ (and not merely $k = 1$), one has
\begin{align*}
	&\F(\bar{\chi}^h, \bar{\bm{u}}^h) (kh) 
	+ \tfrac{h}{2} \norm{\tfrac{\chi^h_k - \chi^h_{k-1}}{h}}_{H^{-1}_{(0)}}^2
	+ \tfrac{h}{2} \norm{\tfrac{\bm{u}^h_k - \bm{u}^h_{k-1}}{h}}^2_{\bm{H}^1_{\chi^h_{k-1}}} \\ 
	&\quad + \tfrac{1}{2} \int_{(k-1)h}^{kh} 
	\norm{\tfrac{\bar{\chi}^h (t) - \bar{\chi}^h ((k-1)h)}{t - (k-1)h }}_{H^{-1}_{(0)}}^2
	+  \norm{\tfrac{\bar{\bm{u}}^h (t) - \bar{\bm{u}}^h((k-1)h) }{t - (k-1)h }}^2_{\bm{H}^1_{\chi^h_{k-1}}}  \dt
	\leq  \F(\bar{\chi}^h, \bar{\bm{u}}^h) ((k-1)h). 
\end{align*}
Telescoping these inequalities and applying the definition of piecewise linear interpolants, we deduce that for all admissible $k,m \in \N$ with $m < k$ 
\begin{align*}
	\F(\bar{\chi}^h, \bar{\bm{u}}^h) (kh)  \numberthis \label{eq:dissipatoin_discrete} 
	+  \frac{1}{2} \int_{mh}^{kh}
	&+  \norm{ \pt \hat{\chi}^h }_{H^{-1}_{(0)}}^2
	+  \norm{ \pt \hat{\bm{u}}^h }^2_{\bm{H}^1_{\chi^h_{k-1}}} \\ 
	& +
	\norm{ \tfrac{ \bar{\chi}^h (t) - \bar{\chi}^h ( \lfloor\sfrac{t}{h}\rfloor h )}{t - \lfloor\sfrac{t}{h}\rfloor h}}_{H^{-1}_{(0)}}^2 
	+\norm{ \tfrac{\bar{\bm{u}}^h (t) - \bar{\bm{u}}^h(\lfloor\sfrac{t}{h}\rfloor h ) }{t - \lfloor\sfrac{t}{h}\rfloor h}}^2_{\bm{H}^1_{\chi^h_{ \lfloor\sfrac{t}{h}\rfloor} }} \dt
	\leq  \F(\bar{\chi}^h, \bar{\bm{u}}^h) (mh).
\end{align*}
Finally, it remains to generalize the discrete dissipation inequality to the continuous case, as asserted in \eqref{eq:dissipation_approximate}. Once again employing minimality of \eqref{eq:degiorgi} at $t$, it follows for all $(k-1)h < s < t \leq kh$ that 
\begin{align*}
	\F(\bar{\chi}^h,& \bar{\bm{u}}^h) (t) + \tfrac{1}{2( t- (k-1)h )} 
	\Big[
	\norm{\bar{\chi}^h (t) - \bar{\chi}^h ((k-1)h)}_{H^{-1}_{(0)}}^2
	+  \norm{\bar{\bm{u}}^h (t) - \bar{\bm{u}}^h((k-1)h) }^2_{\bm{H}^1_{\chi^h_{k-1}}}
	\Big]\\ 
	&\leq \F(\bar{\chi}^h, \bar{\bm{u}}^h) (s) + \tfrac{1}{2( t- (k-1)h )} 
	\Big[
	\norm{\bar{\chi}^h (s) - \bar{\chi}^h ((k-1)h)}_{H^{-1}_{(0)}}^2
	+  \norm{\bar{\bm{u}}^h (s) - \bar{\bm{u}}^h((k-1)h) }^2_{\bm{H}^1_{\chi^h_{k-1}}}
	\Big]. 
\end{align*}
Exploiting that the mapping 
\begin{equation}
	((k-1)h, kh] \ni t \mapsto \norm{\bar{\chi}^h (t) - \bar{\chi}^h ((k-1)h)}_{H^{-1}_{(0)}}^2
	+  \norm{\bar{\bm{u}}^h (t) - \bar{\bm{u}}^h((k-1)h) }^2_{\bm{H}^1_{\chi^h_{k-1}}}
\end{equation}
is non-decreasing, cf.\ \ref{eq:f_nd}, along with \eqref{eq:energy_decreasing}, we deduce that 
\begin{equation}\label{eq:F_non-increasing} 
	(0, T_*) \ni t \mapsto \F(\bar{\chi}^h, \bar{\bm{u}}^h) (t) 
	\textnormal{ is non-increasing for all } h \in (0, 1). 
\end{equation}
Moreover, for any admissible $ 0 < h < \min\{T- \tau, \kappa - s \}$, we can find $k_0, m_0$ such that $\tau < k_0 h < T$ and $s < m_0 h < \kappa$. Then, it follows from \eqref{eq:F_non-increasing} and the positivity of the integrand that 
\begin{align*}
	\F(\bar{\chi}^h&, \bar{\bm{u}}^h) (T) +
	\begin{aligned}[t]
		\frac{1}{2} \int_{\kappa}^{\tau}
		&  \norm{ \pt \hat{\chi}^h }_{H^{-1}_{(0)}}^2
		+  \norm{ \pt \hat{\bm{u}}^h }^2_{\bm{H}^1_{\chi^h_{k-1}}} \\ 
		&+
		\norm{\tfrac{\bar{\chi}^h (t) - \bar{\chi}^h ( \lfloor\sfrac{t}{h}\rfloor h )}{t - \lfloor\sfrac{t}{h}\rfloor h}}_{H^{-1}_{(0)}}^2
		+\norm{\tfrac{\bar{\bm{u}}^h (t) - \bar{\bm{u}}^h(\lfloor\sfrac{t}{h}\rfloor h ) }{t - \lfloor\sfrac{t}{h}\rfloor h}}^2_{\bm{H}^1_{\chi^h_{\lfloor\sfrac{t}{h}\rfloor}}} \dt
	\end{aligned} \\ 
	& \leq 
	\F(\bar{\chi}^h, \bar{\bm{u}}^h) (k_0h) +
	\begin{aligned}[t]
		\frac{1}{2} \int_{m_0 h}^{k_0h}
		&  \norm{ \pt \hat{\chi}^h }_{H^{-1}_{(0)}}^2
		+  \norm{ \pt \hat{\bm{u}}^h }^2_{\bm{H}^1_{\chi^h_{k-1}}} \\ 
		&+
		\norm{\tfrac{\bar{\chi}^h (t) - \bar{\chi}^h ( \lfloor\sfrac{t}{h}\rfloor h )}{t - \lfloor\sfrac{t}{h}\rfloor h}}_{H^{-1}_{(0)}}^2
		+\norm{\tfrac{\bar{\bm{u}}^h (t) - \bar{\bm{u}}^h(\lfloor\sfrac{t}{h}\rfloor h ) }{t - \lfloor\sfrac{t}{h}\rfloor h}}^2_{\bm{H}^1_{\chi^h_{ \lfloor\sfrac{t}{h}\rfloor}}}  \dt. 
	\end{aligned} 
\end{align*}
Furthermore, we have $\F(\bar{\chi}^h, \bar{\bm{u}}^h) (s)  \geq \F(\bar{\chi}^h, \bar{\bm{u}}^h) (m_0h) $ since $s < m_0h$. Combining	 these two inequalities with \eqref{eq:dissipatoin_discrete} finally yields \eqref{eq:dissipation_approximate}.  
\hfill $\diamondsuit$

\subsection{The Euler-Lagrange equations}\label{sec:euler_lagrange}

In analogy with Proposition \ref{prop:minimizer}, we define for all $t \in (0, T_*) \setminus \N h$ the functional 
\begin{equation*}
	\F^h_t \colon \mathcal{M} \rightarrow \R, 
	\quad 
	(\tilde{\chi}, \tilde{\bm{u}}) \mapsto \F(\tilde{\chi}, \tilde{\bm{u}}) 
	+ \tfrac{1}{2(t - \lfloor \sfrac{t}{h} \rfloor h )}\norm{\tilde{\chi} - \chi_{\lfloor \sfrac{t}{h} \rfloor}^h}^2_{H^{-1}_{(0)}}  
	+\tfrac{1}{2(t - \lfloor \sfrac{t}{h} \rfloor h)} \norm{ \tilde{\bm{u}} - \bm{u}_{\lfloor \sfrac{t}{h} \rfloor }^h}^2_{\bm{H}^1_{\chi^h_{\lfloor \sfrac{t}{h} \rfloor }}} 
	\numberthis \label{discrete_functional_c}
\end{equation*}
and note that for all admissible $h > 0$, and all $t \in (0, T_*) \setminus \N h$ a minimizer exists. The aim of this section is to derive the corresponding Euler-Lagrange equations and therefore the equations which are satisfied by $(\bar{\chi}^h, \bar{\bm{u}}^h) (t)$. \\ 
We recall the definition of the last term 
\begin{equation*}
	(2(t - \lfloor \sfrac{t}{h} \rfloor h)) \mathcal{D}_{\bm{u}}^{h, t} (\tilde{\bm{u}}) \coloneqq   \norm{ \tilde{\bm{u}} - \bm{u}_{\lfloor \sfrac{t}{h} \rfloor }^h}^2_{\bm{H}^1_{\chi^h_{\lfloor \sfrac{t}{h} \rfloor }}}  
	= \int_\Omega \C_{\nu}(\chi^h_{\lfloor \sfrac{t}{h} \rfloor }) \E(\tilde{\bm{u}} - \bm{u}_{\lfloor \sfrac{t}{h} \rfloor }^h) \colon \E(\tilde{\bm{u}} - \bm{u}_{\lfloor \sfrac{t}{h} \rfloor }^h) \dx 
\end{equation*}
and compute the first variation in direction of $\bm{v} \in \disp$, cf.\ \cite[Lem.\ 3.2]{garcke_00}, 
\begin{equation}\label{var:D_u_u}
	\delta_{\tilde{\bm{u}}} \ \mathcal{D}_{\bm{u}}^{h, t} (\tilde{\bm{u}})   [\bm{v}]
	=  \int_\Omega \C_{\nu}(\chi^h_{\lfloor \sfrac{t}{h} \rfloor }) \E\big(\tfrac{\tilde{\bm{u}} - \bm{u}_{\lfloor \sfrac{t}{h} \rfloor }^h}{t - \lfloor \sfrac{t}{h} \rfloor h }\big) \colon \E(\bm{v}) \dx . 
\end{equation}
The (inner) variation with respect to $\tilde{\chi}$ is more involved, and a tedious computation based on the approach in \cite{garcke_00}, see also Section \ref{sec:energetics}, shows that
\begin{equation}\label{var_D_u_chi}
	\delta_{\tilde{\chi}} \ \mathcal{D}_{\bm{u}}^{h, t} (\tilde{\bm{u}})   [\bm{B}]
	= - \int_\Omega \C_{\nu}(\chi^h_{\lfloor \sfrac{t}{h} \rfloor })  \E\big(\tfrac{\tilde{\bm{u}} - \bm{u}_{\lfloor \sfrac{t}{h} \rfloor }^h}{t - \lfloor \sfrac{t}{h} \rfloor h}\big)
	\colon \nabla^2 \tilde{\bm{u}} \bm{B}  
	+ (\nabla \tilde{\bm{u}})^T  \C_{\nu}(\chi^h_{\lfloor \sfrac{t}{h} \rfloor })  \E\big(\tfrac{\tilde{\bm{u}} - \bm{u}_{\lfloor \sfrac{t}{h} \rfloor }^h}{t - \lfloor \sfrac{t}{h} \rfloor h} \big)  \colon \nabla \bm{B} \dx. 
\end{equation}
Proceeding to the second dissipative term
\begin{equation*}
	\mathcal{D}^{h, t}_{\chi} (\tilde{\chi}) 
	= \tfrac{1}{2(t - \lfloor \sfrac{t}{h} \rfloor h )}\norm{\tilde{\chi} - \chi_{\lfloor \sfrac{t}{h} \rfloor}^h}^2_{H^{-1}_{(0)}}  
	=  \tfrac{t - \lfloor \sfrac{t}{h} \rfloor h }{2}  \norm{\tfrac{\tilde{\chi} - \chi_{\lfloor \sfrac{t}{h} \rfloor}^h}{t - \lfloor \sfrac{t}{h} \rfloor h }}^2_{H^{-1}_{(0)}}   
	= \tfrac{t - \lfloor \sfrac{t}{h} \rfloor h }{2} \int_\Omega \abs{\nabla {w}}^2 \dx, 
\end{equation*}
where $w \in H^1_{(0)} (\Omega)$ solves the corresponding problem of the form \eqref{eq:distance_elliptic} with $\chi$ replaced by $\tfrac{\tilde{\chi} - \chi_{\lfloor \sfrac{t}{h} \rfloor}^h}{t - \lfloor \sfrac{t}{h} \rfloor h }$. More precisely, for all $t \in (0, T_*)$ we define the potential $w^h(t, \cdot)$ associated with the variational interpolant $\bar{\chi}^h$ as the unique solution of 
\begin{equation}\label{eq:w_approximate}
	\left\{ 
	\begin{aligned}
		\Delta w^h(t, \cdot ) & =  \tfrac{\bar{\chi}^h(t, \cdot) - \chi_{\lfloor \sfrac{t}{h} \rfloor}^h}{t - \lfloor \sfrac{t}{h} \rfloor h }&& \textnormal{in } \Omega ,\\ 
		\nabla w^h(t, \cdot ) \cdot \bm{n}_{\partial \Omega} &= 0 && \textnormal{on } \partial \Omega,\\ 
		\int_\Omega w^h(t, \cdot) &= 0. 
	\end{aligned}
	\right. 
\end{equation}
It is well-known, cf.\ \cite{roeger, sturzenhecker}, \cite[Lem.\ 10]{hensel_stinson}, that 
\begin{equation}\label{var:D_chi_chi}
	\delta_{\tilde{\chi}} \mathcal{D}^{h, t}_{\chi} (\tilde{\chi})  [\bm{B}]
	= - \int_\Omega \tilde{\chi} \  \div \big( w \bm{B} \big) \dx . 
\end{equation}
Recall that the energy functional consists of four contributions $\F = \F_{per} + \F_{cap}+ \F_{el} + \F_{\mathfrak{h}}$. 
For the variation of $\F_{el} + \F_{\mathfrak{h}}$ we have already seen in Section~\ref{sec:gradient_flow} that 
\begin{align*}
	\delta_{(\tilde{\chi}, \tilde{\bm{u}})} \big( \F_{el}&((\tilde{\chi}, \tilde{\bm{u}})) + \F_{\mathfrak{h}} (\tilde{\bm{u}}) \big)
	[\bm{B}, \bm{v}] \\
	= &\int_\Omega \big[ W(\tilde{\chi}, \E(\tilde{\bm{u}})) \bm{I} -  (\nabla {\tilde{\bm{u}}} )^T W_{, \E} (\tilde{\chi}, \E(\tilde{\bm{u}})  \big)\big] \colon \nabla \bm{B} \dx 
	+ \int_\Omega \tfrac{1}{2} \abs{\nabla^2 \tilde{\bm{u}}}^2\  \div \bm{B} \dx \\ 
	&-\int_\Omega (\nabla^2\tilde{ \bm{u}})_{ijk}   (\nabla\tilde{ \bm{u}})_{ip}  (\nabla^2 \bm{B})_{pjk}  \dx 
	-\int_\Omega (\nabla^2 \tilde{\bm{u}})_{ijk} \Big( (\nabla^2 \tilde{\bm{u}})_{ipk}  (\nabla\bm{B})_{jp} 
	+ (\nabla^2 \tilde{\bm{u}})_{ijp}  (\nabla{\bm{B}})_{kp}  \Big) \dx\\ 
	&   + \int_\Omega \WE (\tilde{\chi}, \E(\tilde{\bm{u}})) \colon \nabla \bm{v} \dx 
	+\int_\Omega \nabla^2 \tilde{\bm{u}}\colon   \nabla^2 \bm{v} \dx	. 
	\numberthis  \label{var:F_el_h_u}
\end{align*}
The variation of the perimeter and capillary energy in the varifold formulation is given by the first tangential variation
\begin{equation}\label{var:per_varifold}
	\delta_{\tilde{\chi}} \big( \F_{per} (\tilde{\chi}) + \F_{cap} (\tilde{\chi}) \big) [\bm{B}]  
	= \int_{\overline{\Omega} \times \mathbb{S}^{d-1}} (\bm{I} - \bm{s} \otimes \bm{s}) \colon \nabla \bm{B} \, d\mu_{\tilde{\chi}} (\bm{x}, \bm{s}). 
\end{equation}
\par 
\smallskip  
We observe that the variational interpolants $(\bar{\chi}, \bar{\bm{u}})$ are minimizers of the functional \eqref{discrete_functional_c}, and therefore satisfy the Euler--Lagrange equations. Hence, the identities \eqref{var_D_u_chi}, \eqref{var:D_chi_chi} and \eqref{var:F_el_h_u} give rise to 
\begin{align*}
	-\delta_{\tilde{\chi}} &(\F_{per} (\bar{\chi}^h)  + \F_{cap} (\bar{\chi}^h) )[\bm{B}]  \\ 
	= &- \int_\Omega \bar{\chi}^h \  \div \big( w^h \bm{B} \big) \dx  
	- \int_\Omega \C_{\nu}(\chi^h_{\lfloor \sfrac{t}{h} \rfloor })  \E\big(\tfrac{\bar{\bm{u}}^h - \bm{u}_{\lfloor \sfrac{t}{h} \rfloor }^h}{t - \lfloor \sfrac{t}{h} \rfloor h}\big)
	\colon \nabla^2 \bar{\bm{u}}^h \bm{B}  \dx 
	+ \int_\Omega \tfrac{1}{2} \abs{\nabla^2\bar{ \bm{u}}^h}^2\  \div \bm{B} \dx \\ 
	&+ \int_\Omega \Big[ W(\bar{\chi}^h, \E(\bar{\bm{u}}^h)) \bm{I} -  ( \nabla \bar{\bm{u} } )^T \big( \C_{\nu}(\chi^h_{\lfloor \sfrac{t}{h} \rfloor })  \E\big(\tfrac{\bar{\bm{u}} - \bm{u}_{\lfloor \sfrac{t}{h} \rfloor }^h}{t - \lfloor \sfrac{t}{h} \rfloor h} \big) +  W_{, \E} (\bar{\chi}^h, \E(\bar{\bm{u}}^h)  \big)\Big] \colon \nabla \bm{B} \dx 
	\\ 
	&-\int_\Omega
	(\nabla^2 \bar{\bm u}^h)_{ijk}\,(\nabla \bar{\bm u}^h)_{ip}\,(\nabla^2 \bm B)_{pjk}
	\, \dx
	-\int_\Omega
	(\nabla^2 \bar{\bm u}^h)_{ijk}
	\Big(
	(\nabla^2 \bar{\bm u}^h)_{ipk}\,(\nabla \bm B)_{jp}
	+
	(\nabla^2 \bar{\bm u}^h)_{ijp}\,(\nabla \bm B)_{kp}
	\Big)
	\, \dx
\end{align*}
for all $\bm{B} \in \mathcal{S}_{\bar{\chi}^h}$ and a.e. $t \in (0, T_*)$. Following the proof of \eqref{eq:langrange_mult_sg}, the same identity holds even for all $\bm{B} \in \bm{C}^2(\overline{\Omega})$ with $\bm{B} \cdot \bm{n}_{\partial \Omega}$ if we replace $w^h$ by $\widetilde{w}^h = w^h + \lambda^h$, where 
\begin{align*}
	\lambda^h &\int_\Omega \bar{\chi}^h \nabla \cdot \bm{\xi} \dx \\
	=& - \int_\Omega  \bar{\chi}^h \nabla \cdot (w^h \bm{\xi}) \dx 
	- \int_\Omega \C_{\nu}(\chi^h_{\lfloor \sfrac{t}{h} \rfloor })  \E\big(\tfrac{\bar{\bm{u}}^h - \bm{u}_{\lfloor \sfrac{t}{h} \rfloor }^h}{t - \lfloor \sfrac{t}{h} \rfloor h}\big)
	\colon \nabla^2 \bar{\bm{u}}^h \bm{\xi}  \dx 
	+ \int_\Omega \tfrac{1}{2} \abs{\nabla^2\bar{ \bm{u}}^h}^2\  \div \bm{\xi} \dx \\ 
	&+ \int_\Omega \Big[ W(\bar{\chi}^h, \E(\bar{\bm{u}}^h)) \bm{I} -  ( \nabla \bar{\bm{u} } )^T \big( \C_{\nu}(\chi^h_{\lfloor \sfrac{t}{h} \rfloor })  \E\big(\tfrac{\bar{\bm{u}} - \bm{u}_{\lfloor \sfrac{t}{h} \rfloor }^h}{t - \lfloor \sfrac{t}{h} \rfloor h} \big) +  W_{, \E} (\bar{\chi}^h, \E(\bar{\bm{u}}^h)  \big)\Big] \colon \nabla \bm{\xi} \dx 
	\\ 
	&-\int_\Omega
	(\nabla^2 \bar{\bm u}^h)_{ijk}\,(\nabla \bar{\bm u}^h)_{ip}\,(\nabla^2 \bm \xi)_{pjk}
	\dx
	-\int_\Omega
	(\nabla^2 \bar{\bm u}^h)_{ijk}
	\Big(
	(\nabla^2 \bar{\bm u}^h)_{ipk}\,(\nabla \bm \xi)_{jp}
	+
	(\nabla^2 \bar{\bm u}^h)_{ijp}\,(\nabla \bm \xi)_{kp}
	\Big)
	\dx \\ 
	&-\delta_{\tilde{\chi}}(\F_{per} (\bar{\chi}^h)  + \F_{cap} (\bar{\chi}^h) ) [\bm{\xi}] ,
	\numberthis \label{eq:lagrange_mult_mm}
\end{align*}
with $\bm{\xi} \in \bm{C}^{\infty}(\overline{\Omega})$ satisfying $\bm{\xi} \cdot \bm{n}_{\partial \Omega}  = 0$ and  
\begin{equation}\label{eq:C^2_estiamte_xi}
	\norm{\bm{\xi}}_{\bm{C}^2} \leq C(\Omega, m_0, \abs{\Gamma}), 
	\quad \textnormal{and} \quad 
	\int_\Omega \chi \nabla \cdot \bm{\xi} \dx  \geq C(m_0, \Omega). 
\end{equation}
In particular, we obtain the estimate 
\begin{equation}\label{estimate:lambda_implicit}
	\abs{\lambda^h} 
	\leq C\big(\Omega, m_0, \abs{\mu^h_t} (\overline{\Omega})\big)
	\Big(\norm{\nabla w^h}_{L^1} 
	+ \norm{\tfrac{\bar{\bm{u}}^h - \bm{u}_{\lfloor \sfrac{t}{h} \rfloor }^h}{t - \lfloor \sfrac{t}{h} \rfloor h}}_{\bm{H}^1}   \norm{\bar{\bm{u}}^h}_{\bm{H}^2}+ 
	\norm{\bar{\bm{u}}^h}_{\bm{H}^2}^2 
	+ \abs{\mu^h_t} (\overline{\Omega})
	\Big),
\end{equation}
and note that, by definition, one has $\abs{\mu^h_t} (\Omega)=  \abs{\nabla \bar{\chi}^h(t, \cdot)} (\Omega)$.
\par 
\medskip 
To conclude this section, let us summarize our results. First of all, it holds for all $\bm{B} \in \bm{C}^2(\overline{\Omega})$ with $\bm{B}_{|\po} \cdot \bm{n}_{\partial \bm{n}} \equiv 0$ that 
\begin{align*}
	&\int_{\Omega \times \mathbb{S}^{d-1}} (\bm{I} - \bm{s} \otimes \bm{s}) \colon \nabla \bm{B} \, d\mu_{\bar{\chi}} (\bm{x}, \bm{s}) \numberthis \label{el_w_appoximate} \\ 
	& \quad = \int_\Omega \bar{\chi}^h \  \div \big( \widetilde{w}^h \bm{B} \big) \dx  
	+\int_\Omega \C_{\nu}(\chi^h)  \E\big(\tfrac{\bar{\bm{u}}^h - \bm{u}_{\lfloor \sfrac{t}{h} \rfloor }^h}{t - \lfloor \sfrac{t}{h} \rfloor h}\big)
	\colon \nabla^2 \bar{\bm{u}}^h \bm{B}  \dx 
	- \int_\Omega \tfrac{1}{2} \abs{\nabla^2\bar{ \bm{u}}^h}^2\  \div \bm{B} \dx \\ 
	&\qquad - \int_\Omega \Big[ W(\bar{\chi}^h, \E(\bar{\bm{u}}^h)) \bm{I} -  ( \nabla \bar{\bm{u}}^h  )^T \big( \C_{\nu}(\chi^h)  \E\big(\tfrac{\bar{\bm{u}} - \bm{u}_{\lfloor \sfrac{t}{h} \rfloor }^h}{t - \lfloor \sfrac{t}{h} \rfloor h} \big) +  W_{, \E} (\bar{\chi}^h, \E(\bar{\bm{u}}^h)  \big)\Big] \colon \nabla \bm{B} \dx 
	\\ 
	&\qquad +\int_\Omega
	(\nabla^2 \bar{\bm u}^h)_{ijk}\,(\nabla \bar{\bm u}^h)_{ip}\,(\nabla^2 \bm B)_{pjk}
	\dx
	+\int_\Omega
	(\nabla^2 \bar{\bm u}^h)_{ijk}
	\Big(
	(\nabla^2 \bar{\bm u}^h)_{ipk}\,(\nabla \bm B)_{jp}
	+
	(\nabla^2 \bar{\bm u}^h)_{ijp}\,(\nabla \bm B)_{kp}
	\Big)
	\dx, 
\end{align*}
Moreover, \eqref{var:D_u_u} and \eqref{var:F_el_h_u} yield 
\begin{align}\label{el_elastic_approximate_1}
	\int_\Omega 
	\big[ 
	\C_{\nu}(\chi^h_{\lfloor \sfrac{t}{h} \rfloor }) \E\big(\tfrac{\bar{\bm{u}} - \bm{u}_{\lfloor \sfrac{t}{h} \rfloor }^h}{t - \lfloor \sfrac{t}{h} \rfloor h }\big)
	+ \WE ({\bar{\chi}}^h, \E({\bar{\bm{u}}^h}))
	\big] 
	\colon \E(\bm{v}) 
	+ \nabla^2 {\bar{\bm{u}}}^h\colon \nabla^2 \bm{v} \dx = 0. 
\end{align}
for all $\bm{v} \in \disp$ and any $t \in (0, T_*) \setminus \N h$. Moreover, note that $\C_{\nu} (\chi) \E(\pt \bm{u})$ and $ \WE (\chi, \E(\bm{u}))$ are symmetric and that the inner product of a symmetric and a skew symmetric matrix is zero. Combined with the observation that the second gradient in $\nabla^2 \bm{v}$ annihilates any infinitesimally rigid movement (since $\bm{H}^1(\Omega) = \bm{H}^1_{\textnormal{ird}}(\Omega) \oplus \bm{H}^1_\perp(\Omega)$), we can even allow test function $\bm{v} \in \bm{H}^2(\Omega)$, without the restriction to $\bm{H}^1_\perp(\Omega)$. 
It is also easy to see that 
\begin{equation*}
	\delta_{\tilde{u}} \Big( \tfrac{1}{2h} \norm{\bm{u}_{k-1}^h - \tilde{\bm{u}} }^2_{\bm{H}^1_{\chi^h_{k-1}}} \Big) [\bm{v}]
	= \int_\Omega \C_\nu (\chi^h_{k-1}) \E(\tfrac{\tilde{u} - \bm{u}^h_{k-1}}{h }) \colon \E(\bm{v}) \dx 
\end{equation*}
which implies that for all $t \in (0, T_*)$ and all $\bm{v} \in  \bm{H}^2(\Omega)$
\begin{align}\label{el_elastic_approximate}
	\int_\Omega 
	\big[\C_{\nu}(\chi^h) \E ( \pt \hat{\bm{u}}^h) 
	+ \WE ({\chi}^h, \E({\bm{u}^h}))
	\big] 
	\colon \E(\bm{v}) 
	+ \nabla^2 {\bm{u}}^h\colon \nabla^2 \bm{v} \dx = 0. 
\end{align}
Finally, we see that the energy-dissipation inequality \eqref{eq:dissipation_approximate} is equivalent to 
\begin{align*}
	\F(\bar{\chi}^h, \bar{\bm{u}}^h) (T)
	&+ \tfrac{1}{2} \int_\kappa^\tau 
	\norm{\pt \hat{\chi}^h}_{H^{-1}_{(0)}}^2 
	+  \norm{ \pt \hat{\bm{u}}^h}^2_{\bm{H}^1_{{\chi}^h(t) }} 
	+ \norm{ \nabla w^h  }^2_{L^2} 
	+ \norm{ \tfrac{\bar{\bm{u}}^h   -  \bar{\bm{u}}^h ( \lfloor \sfrac{t}{h} \rfloor h)  }{ {t - \lfloor \sfrac{t}{h} \rfloor h } }}^2_{\bm{H}^1_{{\chi}^h(t) }} \dt \\ 
	&\leq \F(\bar{\chi}^h, \bar{\bm{u}}^h) (s) \leq \F(\bar{\chi}^h, \bar{\bm{u}}^h) (0)
	\numberthis \label{eq:dissipation_approximate_w}
\end{align*}
for all $\tau, \kappa, s, T$ such that $0 < s < \kappa < \tau < T < T_*$ and any admissible $0 < h < \min\{T- \tau, \kappa - s \}$.

\subsection{Compactness}

The approximate energy dissipation inequality \eqref{eq:dissipation_approximate_w} immediately the following uniform estimate

\begin{align*}
	\norm{ \abs{\mu^h_t}(\overline{\Omega}) }_{L^\infty(0, T_*)} 
	+ \norm{ \bar{\bm{u}}^h }_{L^\infty(0, T_*; \bm{H}^2)} 
	&+ 	\norm{\pt \hat{\chi}^h}_{L^2(0, T_*;H^{-1}_{(0)})}
	+  \norm{ \pt \hat{\bm{u}}^h}^2_{L^2(0, T_*;\bm{H}^1) } \\
	&+ \norm{ \nabla w^h  }_{L^2(0, T_*;L^2)} 
	+ \norm{ \tfrac{\bar{\bm{u}}^h   -  \bar{\bm{u}}^h ( \lfloor \sfrac{\cdot}{h} \rfloor h)  }{ {(\cdot\,  - \lfloor \sfrac{\cdot}{h} \rfloor) h } }}_{L^2(0, T_*;\bm{H}^1) } 
	\leq C,
	\numberthis
	\label{est:a_priori_mm}
\end{align*}
where $C < \infty$ is a constant that does not depend on $h > 0$. Here, we used that the norms $\norm{\cdot  }_{\bm{H}^1_\chi}$ are equivalent to the standard norm on $\bm{H}^1(\Omega)$, implying the existence of constants $c, C > 0$ that do not depend on $\chi$ such that 
\begin{equation*}
	c \norm{\bm{v}}_{\bm{H}^1_\chi} \leq \norm{\bm{v}}_{\bm{H}^1} \leq C \norm{\bm{v}}_{\bm{H}^1_\chi}
\end{equation*}
for all $\bm{v} \in \bm{H}^1_{\perp}(\Omega)$. 

\subsubsection*{Potential}

Appealing to the estimates \eqref{estimate:lambda_implicit} and \eqref{est:a_priori_mm}, we find that 
\begin{align*}
	\norm{\lambda^h}_{L^2(L^2)}^2
	&\leq \int_0^{T_*}  C\big(\Omega, m_0, \abs{\mu^h_t} (\Omega)\big) \Big( 
	\norm{\nabla w^h}_{L^2}^2 
	+  \norm{\bar{\bm{u}}^h}_{\bm{H}^2}^2 \norm{\tfrac{\bar{\bm{u}}^h - \bm{u}_{\lfloor \sfrac{t}{h} \rfloor }^h}{t - \lfloor \sfrac{t}{h} \rfloor h}}_{\bm{H}^1}^2 + 
	\norm{\bar{\bm{u}}^h}_{\bm{H}^2}^4
	+ \abs{\mu^h_t}^2(\Omega)
	\Big)\\ 
	&\leq C < \infty
\end{align*}
uniformly in $h > 0$. Using the Poincaré-Wirtinger inequality furthermore yields that $\norm{w^h}_{L^2(H^1)} \leq C $ for some $C < \infty$ and all $h > 0$, allowing us to conclude that 
\begin{align*}
	w^h  &\rightharpoonup \tilde{w}  \quad \textnormal{in } L^2(0, T_*; H^1(\Omega)),
	\label{eq:conv_w} \numberthis
	\\ 
	\lambda^h &\rightharpoonup \lambda \quad \textnormal{in } L^2(0, T_*; L^2(\Omega)), 
\end{align*}
along a subsequence of $h \searrow 0$. By linearity of weak limits these convergences entail $w^h + \lambda^h \rightharpoonup \tilde{w} + \lambda \eqqcolon w$ in $L^2(0, T_*; H^1(\Omega))$ along a not relabeled subsequence.

\subsubsection*{Order parameter}

Our next goal is to establish strong convergence of $\chi^h, \bar{\chi}^h$ and $\hat{\chi}^h$ in $L^2(\Omega_{T_*})$ to a common limit: 
\begin{equation}\label{eq:chi_conv_strong}
	\chi^h, \bar{\chi}^h, \hat{\chi}^h \to \chi \quad \textnormal{in } L^2(0, T_{*}; L^2(\Omega))
\end{equation}
In order to apply the Aubin-Lions-Simon compactness result, we first need to verify that the differences of translations in time of $\chi^h$ converge uniformly to zero:\par 
\medskip
\noindent
\textit{Claim:} 
\begin{equation}\label{eq:translations_chi}
	\int_0^{T_*- \delta} \norm{\chi^h(t+ \delta) - \chi^h}^2_{H^{-1}_{(0)}} \dt \to 0 
	\quad 
	\textnormal{uniformly in } h> 0 \textnormal{ as } \delta \to 0. 
\end{equation}
The subsequent computations follow \cite[Lem.\ 4.2.7]{stinson21}. \par 
\medskip 
\noindent
\textit{Case 1: }($\delta < h$) Suppose that $n = \sfrac{T_*}{h}$, then it holds that 
\begin{align*}
	\int_0^{T_*- \delta} \norm{\chi^h(t+ \delta) - \chi^h(t)}^2_{H^{-1}_{(0)}} \dt
	= \sum_{k = 1}^{n-1} \int_{kh - \delta}^{kh} \norm{ \chi_k^h - \chi_{k-1}^h}^2_{H^{-1}_{(0)}} \dt 
	= \sum_{k = 1}^{n-1}  \delta\norm{ \chi_k^h - \chi_{k-1}^h}^2_{H^{-1}_{(0)}}. 
\end{align*}
Recalling that $\chi_k^h - \chi_{k-1}^h = h\, \pt \hat{\chi}^h(t) = \int_{(k-1)h}^{kh} \pt \hat{\chi}^h(t)  \dt$, we proceed by using the properties of Bochner integrals and Hölder's inequality to obtain 
\begin{align*}
	\int_0^{T_*- \delta} \norm{\chi^h(t+ \delta) - \chi^h(t)}^2_{H^{-1}_{(0)}} \dt 
	&= \sum_{k = 1}^{n-1} \delta \norm{ \int_{(k-1)h}^{kh} \pt \hat{\chi}^h(t)  \dt}_{H^{-1}_{(0)}}^2
	\leq \delta \sum_{k = 1}^{n-1}   \Big( \int_{(k-1)h}^{kh} \norm{ \pt \hat{\chi}^h(t) }_{H^{-1}_{(0)}}  \dt\Big)^2 	\\
	&\leq \delta \sum_{k = 1}^{n-1}   h \int_{(k-1)h}^{kh} \norm{ \pt \hat{\chi}^h(t) }_{H^{-1}_{(0)}}^2  \dt  
	\leq \delta h \int_0^{T_*}  \norm{ \pt \hat{\chi}^h(t) }_{H^{-1}_{(0)}}^2  \dt  
	\leq C \delta, 
\end{align*}
where $C > 0$ is independent of $h \in (0, 1)$ due to \eqref{est:a_priori_mm}.\par 
\medskip
\noindent 
\textit{Case 2: }($\delta \geq h$) For this case, let us define $[\delta]_{h} \coloneqq \delta \textnormal{ mod } h$ and note that then there exists some $l \in \N$ such that $lh = \delta - [\delta]_h$. As before we suppose that $n = \sfrac{T_*}{h}$ and split the resulting subinterval of $(0, T_*)$ in the following way
\begin{align*}
	&\int_0^{T_*- \delta} \norm{\chi^h(t+ \delta) - \chi^h(t)}^2_{H^{-1}_{(0)}} \dt \\ 
	& \quad \leq 
	\sum_{k = 1}^{n-l} \int_{(k-1) h}^{kh - [\delta]_h} \norm{\chi^h_{(k-1)+l} - \chi^h_{k-1}}_{H^{-1}_{(0)}}^2  \dt 
	+ \sum_{k = 1}^{n-l- 1}  \int_{kh - [\delta]_h}^{kh } \norm{\chi^h_{k+l} - \chi^h_{k-1}}_{H^{-1}_{(0)}}^2  \dt \\ 
	& \quad \leq \sum_{k = 1}^{n-l} (h - [\delta]_h) \bigg( \sum_{i = k-1}^{k+l-2} \norm{\chi_{i+1}^h - \chi_i}_{H^{-1}_{(0)}}^2 \bigg)
	+  \sum_{k = 1}^{n-l- 1}  [\delta]_h \bigg( \sum_{i = k-1}^{k+l-1}  \norm{\chi_{i+1}^h - \chi_i}_{H^{-1}_{(0)}}^2\bigg)  \\ 
	& \quad = \sum_{k = 1}^{n-l} (h - [\delta]_h)\bigg( \sum_{i = k-1}^{k+l-2} \norm{ \int_{ih}^{(i+1)h} \pt \hat{\chi}^h(t) \dt }_{H^{-1}_{(0)}}^2 \bigg) 
	+ \sum_{k = 1}^{n-l- 1}  [\delta]_h \bigg( \sum_{i = k-1}^{k+l-1}  \norm{\int_{ih}^{(i+1)h} \pt \hat{\chi}^h(t) \dt}_{H^{-1}_{(0)}}^2\bigg)\\ 
	& \quad \leq\begin{aligned}[t]
		&\sum_{k = 1}^{n-l} (h - [\delta]_h)\bigg( \sum_{i = k-1}^{k+l-2} \Big( \int_{ih}^{(i+1)h} \norm{ \pt \hat{\chi}^h(t)  }_{H^{-1}_{(0)}} \dt \Big)^2\bigg)\\ 
		&\quad + \sum_{k = 1}^{n-l- 1}  [\delta]_h \sum_{i = k-1}^{k+l-1} \bigg( \Big( \int_{ih}^{(i+1)h}  \norm{ \pt \hat{\chi}^h(t)}_{H^{-1}_{(0)}} \dt \Big)^2\bigg)
	\end{aligned} \\ 
	& \quad \leq \begin{aligned}[t]
		\sum_{k = 1}^{n-l} (h - [\delta]_h)  l h  \int_{(k-1)h}^{(k+l-2)h} &\norm{\pt \hat{\chi}^h(t)}_{H^{-1}_{(0)}}^2  \dt 
		+ \sum_{k = 1}^{n-l- 1}  [\delta]_h l h  \int_{(k-1h)}^{(k-1+l)h}  \norm{ \pt \hat{\chi}^h(t)}_{H^{-1}_{(0)}}^2 \dt
	\end{aligned} \\ 
	& \quad \leq \big(\sum_{k = 1}^{n-l} (h- [\delta]_h + [\delta]_h\big) l h \Big(\int_{0}^{T_*} \norm{ \pt \hat{\chi}^h(t)}_{H^{-1}_{(0)}}^2 \dt \Big)\\ 
	&\quad  \leq C h (n-l) l h \leq C \delta, 
\end{align*} 
where the last inequality is due to $nh - lh < T_*$ and $l h \leq \delta$. 
\hfill$\diamondsuit$\par 
\medskip
Using the definition of $w^h$ \eqref{eq:w_approximate}, and the \textit{a priori} estimate \eqref{est:a_priori_mm}, we further observe that the difference between $\chi^h$ and $\bar{\chi}^h$ in $L^2(0, T_*; H^{-1}_{(0)})$ vanishes as $h \searrow 0$, i.e., 
\begin{equation}\label{eq:difference_chi}
	\int_{0}^{T^*} \norm{ \bar{\chi}^h - \chi^h}_{ H^{-1}_{(0)}}^2 \dt  
	= \int_{0}^{T^*} \norm{  \Delta w^h (t- \lfloor \sfrac{t}{h} \rfloor h) }_{ H^{-1}_{(0)}}^2 \dt  
	\leq  h^2 \int_{0}^{T^*} \norm{  w^h }_{ H^{1}_{(0)}}^2 \dt  
	\leq C_\chi h^2  
\end{equation}
for some $C_\chi > 0$ independent of $h > 0$. Exploiting this fact, we can deduce that \eqref{eq:translations_chi} also holds for the De Giorgi interpolants $\bar{\chi}^h$ along some subsequence of $h \searrow 0$. To this end, fix some arbitrary $\eps > 0$ and choose $\bar{h} = \bar{h}(\eps) > 0$ such that $C_\chi h^2 < \eps$. Moreover, due to the uniform convergence established in \eqref{eq:translations_chi}, we can find some $\bar{\delta}' = \bar{\delta}'(\eps) > 0$ with 
\begin{equation*}
	\int_0^{T_*- \delta} \norm{\chi^h(t+ \delta) - \chi^h}^2_{H^{-1}_{(0)}} \dt < \eps
\end{equation*} 
for all $0 < \delta < \bar{\delta}'$ and all $h > 0$. Thus, the triangle inequality and Jensen's inequality readily yield
\begin{align*}
	&\int_0^{T_*- \delta} \norm{\bar{\chi}^h(t+ \delta) - \bar{\chi}^h}^2_{H^{-1}_{(0)}} \dt \numberthis \label{eq:3_eps_translation}
	\\
	&\quad \leq  3 \Big( \int_\delta^{T_*} \norm{\bar{\chi}^h - \chi^h}^2_{H^{-1}_{(0)}} \dt 
	+\int_0^{T_*- \delta} \norm{\chi^h(t+ \delta) - \chi^h}^2_{H^{-1}_{(0)}} \dt
	+ \int_0^{T_*- \delta} \norm{\chi^h - \bar{\chi}^h}^2_{H^{-1}_{(0)}} \dt  \Big)
	\leq 9\eps
\end{align*}
for all $0 < \delta < \bar{\delta}'$ and $0 < h < \bar{h}$. Note that this estimate still depends on $\overline{h}(\eps)$. To obtain the required uniform convergence, we recall that for all fixed $h > 0$ it holds, see e.g.\ \cite[Lem.\ 4.15]{alt2013lineare}, 
\begin{equation*}
	\int_0^{T_*- \delta} \norm{\bar{\chi}^h(t+ \delta) - \bar{\chi}^h(t)}^2_{H^{-1}_{(0)}} \dt \to 0 \quad  \textnormal{ as } \quad  \delta \searrow 0. 
\end{equation*}
By passing to a subsequence $(h_k)_{k \in \N}$ of $h \searrow 0$, we find that there are only finitly many $h_k > \bar{h}$ and we can extract some $\bar{\delta}^*  > 0$ such that \eqref{eq:3_eps_translation} holds for all $0 < \delta < \bar{\delta} \coloneqq \min \{\bar{\delta}', \bar{\delta}^*\}$ and all $(h_k)_{k \in \N}$. We observe that for $d \leq 3$ one has $BV(\Omega) \hookrightarrow\hookrightarrow  L^{\sfrac{6}{5}} (\Omega) \hookrightarrow H^{-1}_{(0)}(\Omega)$, cf.\ \cite[Cor.\ 3.49]{ambrosio2000functions}. Invoking the Aubin-Lions-Simon lemma, see \cite{simon1986compact}, in combination with Lebesgue's generalized dominated convergence result yields the existence of $\chi \in L^2(0, T_*; BV(\Omega, \{0, 1\})) \cap H^1(0, T_*; H^{-1}_{ (0) })$ with
\begin{equation*}
	\chi^h, \bar{\chi}^h \to \chi \quad \textnormal{in } L^2(0, T_{*}; L^2(\Omega)),
\end{equation*}
where the fact that the limits of $\chi^h$ and $\bar{\chi}^h$ coincide follows from \eqref{eq:difference_chi}. It remains establish strong convergence of the piecewise linear interpolant $\hat{\chi}^h$. Recalling the \textit{a priori} estimate \eqref{est:a_priori_mm}, it is obvious that Aubin-Lions-Simon is applicable and we only need to confirm that the limit is indeed $\chi$. Using definition \eqref{eq:linear_interpolants} of $\hat{\chi}^h$, we find for $i = \lfloor \tfrac{t}{h} \rfloor$ that
\begin{align*}
	\norm{\chi^h(t) - \hat{\chi}^h(t)}_{H^{-1}_{(0)}} 
	&=  \norm{\chi^h_i - \hat{\chi}^h(t)}_{H^{-1}_{(0)}} 
	\leq \int_{ih}^t \norm{\pt \hat{\chi}^h(t)}_{H^{-1}_{(0)}} \, ds   \\
	&\leq  h^{\sfrac{1}{2}}\Big( \int_{ih}^{(i+1)h}  \norm{\pt \hat{\chi}^h}_{H^{-1}_{(0)}} ^2 \, ds   \Big)^{\sfrac{1}{2}} 
	\leq h^{\sfrac{1}{2}}    \norm{\pt \hat{\chi}^h}_{L^2(H^{-1}_{(0)})}, 
\end{align*}
such that after integration in time, we arrive at 
\begin{equation*}
	\norm{\chi^h(t) - \hat{\chi}^h(t)}_{L^1(H^{-1}_{(0)})} \leq  h^{\sfrac{1}{2}} \, T_*   \norm{\pt \hat{\chi}^h}_{L^2(H^{-1}_{(0)})}, 
\end{equation*}
which shows that $\chi^h$ and $\hat{\chi}^h$ have the same limit.\par 
Lastly, note that we also claim that $\chi$ admits a time derivative in $H^{-1}_{(0)}(\Omega)$. Indeed, weak compactness implies the existence of $\mathfrak{X} \in L^2(0, T_*; H^{-1}_{(0)})$ such that $\pt \hat{\chi}^h \rightharpoonup \mathfrak{X}$ in $L^2(0, T_*; H^{-1}_{(0)})$. By passing to the limit in the weak formulation, we find 
\begin{equation}\label{eq:conv_dt_chi}
	\int_0^{T_*} \Big( {}_{H^{-1}_{(0)}}\langle \mathfrak{X}, \zeta \rangle_{H^1_{(0)}} + \int_\Omega {\chi}\,  \pt \zeta \dx \Big) \dt = - \int_\Omega \chi_0 \zeta(0, \bm{x}) \dx
\end{equation}
for all test functions $\zeta \in C^1_c([0, T_*) \times \Omega) \cap H^1(0, T_*; H^1_{(0)})$, and conclude that $\chi$ is weakly differentiable in time with $\pt \chi = \mathfrak{X}$. Moreover, since $H^1(0, T_*; H^{-1}_{(0)})$ admits a trace at $t = 0$ in $H^{-1}_{(0)}$, the identity above implies 
\begin{equation*}
	\textnormal{Tr}_{|t = 0} \chi = \chi_0 \quad \textnormal{in } H^{-1}_{(0)}.  
\end{equation*}
\par 
\medskip 
This concludes the arguments concerning convergence in Sobolev spaces. It remains to show that the limit is indeed in $L^\infty_{w^*}(0, T_*; BV(\Omega; \{0, 1\}))$. Due to \eqref{est:a_priori_mm} we already have a uniform estimate for $\nabla \bar{\chi}^h$ in $L^\infty(0, T; \bm{M}(\Omega))$, such that invoking the Banach--Alaoglu theorem yields the existence of some $ \bm{\lambda}  \in L^\infty(0, T_*; \bm{M}(\Omega))$ with
\begin{equation*}
	\nabla \bar{\chi}^h \xrightharpoonup{*} \bm{\lambda} \quad \textnormal{in} \quad  L^\infty_{w^*}(0, T_*; \bm{M}(\Omega))
\end{equation*}
along a not relabeled subsequence.\\ 
Let $\xi \in C^\infty_c(0, T_*)$ and $\bm{\eta} \in C^\infty_c(\Omega)$. Due to the properties of $BV$, it holds that
\begin{equation*}
	\int_0^{T_*}  \xi(t)  \int_\Omega  \bar{\chi}^h(t, \bm{x}) \nabla \cdot \bm{\eta}(\bm{x})\,  \dx \dt  = - \int_0^{T_*} \xi(t) \int_\Omega    \bm{\eta}(\bm{x}) \cdot  d\nabla \chi \dt, 
\end{equation*}
such that passing to the limit using the weak$^*$ and strong convergence results from above implies 
\begin{equation*}
	\int_0^{T_*}  \xi(t)  \int_\Omega  {\chi}(t, \bm{x}) \nabla \cdot \bm{\eta}(\bm{x})\,  \dx \dt  = - \int_0^{T_*} \xi(t) \int_\Omega    \bm{\eta}(\bm{x}) \cdot  d\bm{\lambda}(t) \dt. 
\end{equation*}
By appealing to the fundamental lemma of the calculus of variations, we obtain for almost all $t \in (0, T_*)$ that
\begin{equation*}
	\int_\Omega  {\chi}(t, \bm{x}) \nabla \cdot \bm{\eta}(\bm{x})\,  \dx  = -  \int_\Omega    \bm{\eta}(\bm{x}) \cdot d\bm{\lambda}(t). 
\end{equation*}
Exploiting separability of $\bm{C}^\infty_c(\Omega)$ and invoking a continuity argument, we conclude that $\chi$ is indeed in $L_{w^*}^\infty(0, T_*; BV(\Omega; \{0, 1\}))$ with $\nabla \chi  = \bm{\lambda} $ in $\bm{M}(\Omega)$. In particular, 
\begin{equation}\label{conv:nabla_chi_measure}
	\nabla \bar{\chi}^h \xrightharpoonup{*} \nabla \chi \quad \textnormal{in} \quad L^\infty_{w^*}(0, T_*; \bm{M}(\Omega)). 
\end{equation}

\subsubsection*{Displacement}

Analogously to the order parameter, we claim that the differences between translations in time and the original displacement functions $\bm{u}^h$ vanish uniformly in $h > 0$. \par 
\medskip 
\noindent
\textit{Claim:}
\begin{equation}\label{eq:translations_u}
	\int_0^{T_*- \delta} \norm{\bm{u}^h(t+ \delta) - \bm{u}^h}^2_{\bm{H}^1} \dt \to 0 
	\quad 
	\textnormal{uniformly in } h> 0 \textnormal{ as } \delta \to 0. 
\end{equation}

\noindent
\textit{Proof:} Recalling the \textit{a priori} estimates \eqref{est:a_priori_mm} for $\bm{u}^h$ and $\pt \hat{\bm{u}}^h$, this assertion follows analogously to \eqref{eq:translations_chi}; see also \cite[Lem.\ 4.2.7]{stinson21}. \hfill$\diamondsuit$
\par 
\medskip 
\noindent
Analogously to \eqref{eq:difference_chi}, we observe that the \textit{a priori} estimate further \eqref{est:a_priori_mm} implies 
\begin{equation}\label{eq:u_u_bar_difference}
	\int_0^{T_*} \norm{\bar{\bm{u}}^h - \bm{u}^h}_{\bm{H^1}}^2 \dt 
	= 	\int_0^{T_*}  (t - \lfloor \tfrac{t}{h} \rfloor h)^2  \norm{\tfrac{\bar{\bm{u}}^h - \bar{\bm{u}}^h (\lfloor \tfrac{t}{h} \rfloor h )} { t - \lfloor \tfrac{t}{h} \rfloor h}}_{\bm{H^1}}^2 \dt 
	\leq C_{\bm{u}} h^2. 
\end{equation}
Hence, \eqref{eq:translations_u} and a similar argument as for $\bar{\chi}^h$ implies that along a subsequence of $h \searrow 0$ one has 
\begin{equation*}
	\int_0^{T_*- \delta} \norm{\bar{\bm{u}}^h(t+ \delta) -\bar{ \bm{u}}^h}^2_{\bm{H}^1} \dt \to 0 
	\quad 
	\textnormal{uniformly in } h> 0 \textnormal{ as } \delta \to 0. 
\end{equation*}
Moreover, for the distance between $\bm{u}^h$ and $\hat{\bm{u}}^h$, we obtain 
\begin{align*}
	\norm{\bm{u}^h(t) - \hat{\bm{u}}^h(t)}_{\bm{H}^1} 
	&= \norm{\bm{u}^h_i - \hat{\bm{u}}^h(t)}_{\bm{H}^1}  
	\leq \int_{ih}^t \norm{\pt \hat{\bm{u}}^h}_{\bm{H}^1}  \dt \\
	&\leq h^{\sfrac{1}{2}}\Big( \int_{ih}^{(i+1)h}  \norm{\pt \hat{\bm{u}}^h}_{\bm{H}^1} ^2 \dt  \Big)^{\sfrac{1}{2}} 
	\leq h^{\sfrac{1}{2}}  \norm{\pt \hat{\bm{u}}^h}_{L^2(\bm{H}^1)}, 
	\numberthis \label{eq:u_u_hat_approach}
\end{align*}
which tends to $0$ as $h \searrow 0$ due to the uniform estimate \eqref{est:a_priori_mm} on $\norm{\pt \hat{\bm{u}}^h}_{L^2(\bm{H}^1)}$. The same \textit{a piori} estimate includes a uniform bound for the functions $\bar{\bm{u}}^h$ in $L^2(0, T; \bm{H}^2(\Omega))$.  Likewise bounded are the piecewise constant and piecewise linear interpolants $\bm{u}^h, \hat{\bm{u}}^h$, which is the assertion of the first energy dissipation inequality \eqref{eq:a_piori_trivial}. Thus, Aubin-Lions-Simon yields the existence of a limit $\bm{u} \in L^2(0, T_*; \bm{H}^2(\Omega)) \cap H^1(0, T_*; \bm{H}^1_{\perp} (\Omega))$ such that 
\begin{equation}\label{eq:u_conv_strong}
	\bm{u}^h, \bar{\bm{u}}^h, \hat{\bm{u}}^h \to \bm{u} \quad 
	\textnormal{in } 
	L^2(0, T; \bm{H}^1_{\perp}(\Omega)), 
\end{equation}
where \eqref{eq:u_u_bar_difference} and \eqref{eq:u_u_hat_approach} imply that the three individual limits coincide. Note that, analogously to the argument for the order parameter $\chi$, one can show $\pt \hat{\bm{u}}^h \rightharpoonup \pt \bm{u}$ in $L^2(0, T_*; \bm{H}^1_{\perp}(\Omega))$, and since $H^1(0, T_*; \bm{H}^1_{\perp} (\Omega))$ admits a trace at $t = 0$ in $\bm{H}^1_{\perp} (\Omega)$, we deduce 
\begin{equation*}
	\textnormal{Tr}_{|t = 0} \bm{u} = \bm{u}_0 
	\quad 
	\textnormal{in } \bm{H}^1_{\perp} (\Omega). 
\end{equation*}
\par 
\medskip 
In the next step, we want to establish strong convergence of $\bar{\bm{u}}^h$ in $L^2(0, T_*; \bm{H}^2(\Omega)$. Testing and integrating \eqref{el_elastic_approximate_1} in time yields 
\begin{equation}\label{eq:elastic_apprximate_3}
	\int_0^{T_*} \int_\Omega 
	\big[
	\C_{\nu}(\chi^h_{\lfloor \sfrac{t}{h} \rfloor }) \E\big(\tfrac{\bar{\bm{u}} - \bm{u}_{\lfloor \sfrac{t}{h} \rfloor }^h}{t - \lfloor \sfrac{t}{h} \rfloor h }\big) 
	+ \WE ({\bar{\chi}}^h, \E({\bar{\bm{u}}^h}))
	\big] 
	\colon \E(\bm{v}) 
	+ \nabla^2 {\bar{\bm{u}}}^h\colon  \nabla^2 \bm{v} \dx \dt  = 0
\end{equation}
for all $h > 0$ and all $\bm{v} \in L^2(0, T_*; \bm{H}^2(\Omega))$. The \textit{a priori} estimate \eqref{est:a_priori_mm} and the Banach-Alaoglu theorem further imply  
\begin{align*}
	\bar{\bm{u}}^h &\xrightharpoonup{*} \bm{u} \quad \textnormal{in } L^\infty(0, T_*; \bm{H}^2(\Omega)),\\ 
	\tfrac{\bar{\bm{u}} - \bm{u}_{\lfloor \sfrac{t}{h} \rfloor }^h}{t - \lfloor \sfrac{t}{h} \rfloor h }
	&\rightharpoonup
	\mathfrak{\bm{U}} 
	\quad \textnormal{in } L^2(0, T_*; \bm{H}^1_{\perp}(\Omega)) \numberthis \label{eq:u_pot_conv}
\end{align*}
for some function $\mathfrak{U} \in L^2(0, T; \bm{H}^1_{\perp}(\Omega))$. In particular, these convergences, along with \eqref{eq:chi_conv_strong}, allow us to pass to the limit in \eqref{eq:elastic_apprximate_3}, leading to 
\begin{equation}\label{eq:elastic_limit_1}
	\int_0^{T_*} \int_\Omega 
	\big[
	\C_{\nu}(\chi) \E(\mathfrak{\bm{U}}) 
	+ \WE (\chi, \E(\bm{u}))
	\big] 
	\colon \E(\bm{v}) 
	+ \nabla^2 \bm{u}  \colon \nabla^2 \bm{v} \dx \dt  = 0
\end{equation}
for all $\bm{v} \in L^2(0, T_*; \bm{H}^2(\Omega))$. Hence, we can first test both this equation and \eqref{eq:elastic_apprximate_3} with the difference $\bm{u} - \bar{\bm{u}}^h \in L^2(0, T_*; \bm{H}^2(\Omega))$ and arrive after subtracting the resulting identities from each other at
\begin{align*}
	&\int_0^{T_*} \int_\Omega  \abs{\nabla^2( \bm{u} - \bar{\bm{u}}^h )}^2 \dx \dt = 
	\int_0^{T_*} \int_\Omega 
	\begin{aligned}[t]
		\Big[
		\C_{\nu}&(\chi^h_{\lfloor \sfrac{t}{h} \rfloor }) \E\big(\tfrac{\bar{\bm{u}} - \bm{u}_{\lfloor \sfrac{t}{h} \rfloor }^h}{t - \lfloor \sfrac{t}{h} \rfloor h }\big)
		- \C_{\nu}(\chi) \E(\mathfrak{\bm{U}})\\
		&+ \WE ({\bar{\chi}}^h, \E({\bar{\bm{u}}^h}))
		- \WE (\chi, \E(\bm{u})) 
		\Big]  
		\colon \E( \bm{u} - \bar{\bm{u}}^h )\dx \dt . 
	\end{aligned}		
\end{align*} 
Observe that the first two terms on the right-hand side are bounded in $L^2(0, T_*; \bm{L}^2(\Omega))$ due to \eqref{est:a_priori_mm}.
Moreover, we have already shown $\E( \bm{u} - \bar{\bm{u}}^h ) \to \bm{0}$ in $L^2(0, T_*; \bm{L}^2(\Omega))$, entailing that  
\begin{equation}\label{eq:u_conv_H2}
	\bar{\bm{u}}^h \to \bm{u} \quad \textnormal{in } L^2(0, T; \bm{H}^2(\Omega)). 
\end{equation}
Finally, let us conclude that $\mathfrak{\bm{U}}$ and $\pt \bm{u}$ actually coincide. As we have done for \eqref{el_elastic_approximate}, we may also test \eqref{el_elastic_approximate}, integrate in time and pass to the limit, obtaining 
\begin{align*}
	\int_0^{T_*}\int_\Omega 
	\big[\C_{\nu}(\chi) \E ( \pt \bm{u}) 
	+ \WE ({\chi}, \E({\bm{u}}))
	\big] 
	\colon \E(\bm{v}) 
	+ \nabla^2 {\bm{u}}\colon \nabla^2 \bm{v} \dx \dt  = 0
\end{align*}
for all $\bm{v} \in L^2(0, T_*; \bm{H}^2(\Omega))$. Subtracting \eqref{eq:elastic_limit_1} from this identity therefore leads to 
\begin{equation}\label{eq:u_limit_reduced}
	\int_0^{T_*}\int_\Omega 
	\C_{\nu}(\chi)\big[  \E ( \pt \bm{u} - \mathfrak{\bm{U}}) 
	\big] 
	\colon \E(\bm{v}) 
	\dx \dt  = 0
\end{equation}
for all $\bm{v} \in L^2(0, T_*; \bm{H}^2(\Omega))$. Since $\bm{H}^2(\Omega) \cap \bm{H}^1_{\perp}(\Omega)$ is dense in $\bm{H}^1_{\perp}(\Omega)$ and separable, the uniqueness of solutions to elliptic problems yields that $\pt \bm{u}(t)$ and $\mathfrak{\bm{U}}(t)$ must already coincide in $\bm{H}^1_{\perp}(\Omega)$ for a.e.\ $t \in (0, T_*)$. 

\subsubsection*{Limit varifold}

Due to the uniform estimate \eqref{est:a_priori_mm}, weak$^*$-compactness of finite Radon measures, cf.\ \cite[Thm.\ 1.54]{ambrosio2000functions}, implies that 
\begin{equation}\label{eq:conv_mu_measure}
	\begin{alignedat}{2}
		\mathcal{L}^1{}_{\llcorner}(0, T_*) \otimes (\mu^{\Omega, h}_t)_{t \in (0, T_*)} 
		&\xrightharpoonup{*} 
		\mu^{\Omega} 
		\quad 
		&&\textnormal{in } 
		M((0, T_*) \times {\overline{\Omega}} \times \mathbb{S}^{d-1}), \\
		\mathcal{L}^1{}_{\llcorner}(0, T_*) \otimes (\mu^{\partial \Omega, h}_t)_{t \in (0, T_*)} 
		&\xrightharpoonup{*} 
		\mu^{\pO}
		\quad 
		&&\textnormal{in } 
		M((0, T_*) \times \pO \times \mathbb{S}^{d-1}), 
	\end{alignedat}
\end{equation} 
along a not relabeled subsequence of $h \searrow 0$. At this point, however, it is not entirely obvious that the measures $\mu^\Omega, \mu^{\pO}$ can be disintegrated into slices in time such that $\mu^{\Omega} = \mathcal{L}^1{}_{\llcorner}(0, T_*) \otimes (\mu^{\Omega}_t)_{t \in(0, T_*)}$ and $\mu^{\partial \Omega} = \mathcal{L}^1{}_{\llcorner}(0, T_*) \otimes (\mu^{\partial \Omega}_t)_{t \in(0, T_*)}$, respectively . To see this, let us suppose that there exists a non-negative real function $t \mapsto e_\mu (t) \in L^1(0, T_*)$ such that 
\begin{equation}\label{eq:conv_point_mass_measure}
	| \mu^h_t |(\overline{\Omega})  \to e_\mu (t) \quad \textnormal{for a.e.\  } t \in (0, T_*). 
\end{equation}
That this assumption is indeed true, along a suitably chosen subsequence, will be shown after \eqref{eq:mu_mass_point}. The assertion then follows from similar arguments as in \cite[Lem.\ 2]{MR4904499}. Focusing on $\mu^\Omega$, we repeat the proof here for the sake of the reader. \\
Let $\eta \in C^0_c((0, T_*))$ be non-negative, i.e., $\eta \geq 0$, and $\zeta \in C^0_c(\overline{\Omega} \times \mathbb{S}^{d-1})$ with $\zeta  \in [0, 1]$. 
The fact that $\zeta \leq 1$ immediately implies 
\begin{equation*}
	\int_{(0, T_*) \times \overline{\Omega} \times \mathbb{S}^{d-1}} \eta(t) \zeta(\bm{x}, \bm{s}) \, d(\mathcal{L}^1{}_{\llcorner}(0, T_*) \otimes \mu^{\Omega, h}_t ) (t, \bm{x}, \bm{s}) 
	\leq 
	\int_0^{T_*} \eta(t)\,  d |\mu^h_t| (\overline{\Omega}) \dt, 
\end{equation*} 
where we also exploit that the second marginal of $\mu^{\Omega, h}_t$ is a probability measure on $\mathbb{S}^{d-1}$. Due to \eqref{eq:conv_mu_measure}, we can pass to the limit on the left-hand side; for the right-hand side, the pointwise convergence \eqref{eq:conv_point_mass_measure} and the dominated convergence theorem then entail 
\begin{equation*}
	\int_{(0, T_*) \times \overline{\Omega} \times \mathbb{S}^{d-1}} \eta(t) \zeta(\bm{x}) \xi(\bm{s}) \, d\mu^{\Omega}(t, \bm{x}, \bm{s}) 
	\leq 
	\int_0^{T_*} \eta(t) e_{\mu}(t) \dt. 
\end{equation*}
Moreover, by monotone convergence \cite[Thm.\ 1.19]{ambrosio2000functions}, it even holds that 
\begin{equation}\label{eq:abs_cont_measures}
	\int_{(0, T_*) \times \overline{\Omega} \times \mathbb{S}^{d-1}} \eta(t)  \, d\mu^\Omega(t, \bm{x}, \bm{s}) 
	\leq 
	\int_0^{T_*} \eta(t) e_{\mu}(t) \dt. 
\end{equation}
Observe that the left-hand side in this inequality is the marginal of $\mu^\Omega$ with respect to $\overline{\Omega}\times \mathbb{S}^{d-1}$. Noting that the right-hand side is finite for any non-negative $\eta \in C^0_c((0, T_*))$, we deduce that this marginal is a Radon measure on $(0, T_*)$. Applying standard disintegration results, see Theorem \ref{thm:disintegration}, we find that there exists a Radon measure $\sigma$ on $(0, T_*)$ and a family of Radon probability measures $(\omega_t)_{t \in (0; T_*)}$ on $\overline{\Omega} \times \mathbb{S}^{d-1}$ such that $\mu^{\Omega} = \sigma \otimes (\omega_t)_{t \in (0; T_*)}$. The inequality \eqref{eq:abs_cont_measures} therefore asserts that $\sigma \leq e \mathcal{L}^1{}_\llcorner (0, T_*)$; in particular, $\sigma$ is absolutely continuous with respect to the Lebesgue measure, allowing us to apply the Radon--Nikod$\acute{\textnormal{y}}$m theorem \cite[Thm.\ 1.28]{ambrosio2000functions} to find an $\mathcal{L}^1$-measurable function $\rho$ on $(0, T_*)$ with $\rho(t) \leq e_\mu (t)$ for a.e.\  $t \in (0, T_*)$ such that $\sigma = \rho \mathcal{L}^1{}_\llcorner (0, T_*)$. Hence, the limit measure $\mu^\Omega$ is indeed of the form $\mu^{\Omega} = \mathcal{L}^1{}_\llcorner (0, T_*) \otimes  (\mu_t)_{t \in(0, T_*)}$ with $\mu_t^\Omega = \rho(t) \omega_t$. \\
Likewise the assertion follows for $\mu^{\partial \Omega}$. Applying the disintegration theorem once more, we find that the time slices $\mu^\Omega_t, \mu^{\pO}_t$ may further be decomposed as 
\begin{alignat*}{2}
	\mu_t^\Omega &= |\mu_t^\Omega| \otimes (\mu^\Omega_{t, \bm{x}})_{\bm{x} \in \overline{\Omega}}
	\quad && \in  M(\overline{\Omega} \times\S),\\ 
	\mu_t^{\pO} &= |\mu_t^{\pO}| \otimes (\mu^{\pO}_{t, \bm{x}})_{\bm{x} \in \partial \Omega}
	\quad && \in M(\po \times \S), 
\end{alignat*}
where $(\mu^\Omega_{t, \bm{x}})_{\bm{x} \in \overline{\Omega}}$, $(\mu^\Omega_{t, \bm{x}})_{\bm{x} \in \pO}$ are  families of Radon probability measures on $\mathbb{S}^{d-1}$. 
Using this information about $\mu^{\pO}$, we can even show that there exists a non-negative function $g \in L^\infty((0, T_*) \times \po; [0, 1])$ such that $\mu^{\partial \Omega} = \mathcal{L}^1{}_{\llcorner}(0, T_*) \otimes g_t \mathcal{H}^{d-1} \llcorner\po \otimes (\delta_{\bm{n}_{\po}(\bm{x})} )_{\bm{x} \in \po}$. Indeed, it holds that, 
$BV(\Omega) \hookrightarrow L^1(\partial \Omega)$, \cite[Thm.\ 3.87]{ambrosio2000functions}, and since $\bar{\chi}^h \in BV(\Omega; \{0, 1\})$ is a characteristic function, one deduces that the trace is non-negative and satisfies $\bar{\chi}^h(\bm{x}) = \textnormal{Tr}_{\po} \bar{\chi}^h(\bm{x}) \leq 1$ for almost all $\bm{x} \in \po$, see \cite[Eq.\ (3.86)]{ambrosio2000functions}. This in turn implies that $\bar{\chi}^h \in L^\infty( (0, T_*); L^\infty(\po; [0, 1]))$ for all $h > 0$, which already entails a uniform bound. Thus, there exists some $g \in  L^\infty( (0, T_*); L^\infty(\po; [0, 1]))$ such that 
\begin{equation*}
	\bar{\chi}^h \xrightharpoonup{*} g \quad \textnormal{in} \quad L^\infty(0, T; L^p(\po; [0, 1])), 
	\quad \textnormal{for any fixed } 1 < p < \infty
\end{equation*}
along a not relabeled subsequence. We note that $0 \leq g(t, \bm{x}) \leq 1$ is an immediate consequence of the fact that in Banach spaces, closed and convex sets are also closed with respect to the weak$^*$ topology.\\ 
For any $\xi \in C^0( (0, {T_*}) \times \po \times \S)$ the convergence of $\mu^{\po, h}$ in measure and the disintegrated representation shown above yield
\begin{equation*}
	\int_0^{T_*} \int_{\po} \int_{\S} \xi(t, \bm{x}, \bm{s}) \, d\mu^{\po, h}_t(\bm{x}, \bm{s}) \dt 
	\to 
	\int_0^{T_*} \int_{\po} \int_{\S} \xi(t, \bm{x}, \bm{s}) \, d\mu^{\po}_{t, \bm{x}} (\bm{s})\  d |\mu^{\po}_t|(\bm{x}) \dt . 
\end{equation*}
Conversely, by the definition \eqref{eq:def_mu_boundary} of $\mu^{h, \po}$ and the convergence of the traces of $\bar{\chi}^h$ we also have 
\begin{align*}
	\int_0^{T_*}\int_{\po} \int_{\S} \xi(t, \bm{x}, \bm{s}) \, d\mu^{\po, h}_t(\bm{x}, \bm{s}) \dt  
	&= \int_0^{T_*} \int_{\po} \xi(t, \bm{x}, \bm{n}_{\po(\bm{x})}) \bar{\chi}^h (t, \bm{x}) \dH \dt \\
	\to  \int_0^{T_*}  \int_{\po} \xi(t, \bm{x}, \bm{n}_{\po(\bm{x})}) g (t, \bm{x}) \dH \dt 
	&=  \int_0^{T_*}  \int_{\po} \int_{\S} \xi(t, \bm{x},\bm{s}) \, d\delta_{ \bm{n}_{\po(\bm{x})}} \,  g_t (\bm{x})  \dH \dt. 
\end{align*}
Comparing these two limits, we deduce the asserted structure of $\mu^{\po}$.\\
Finally, we extend $\mu_t^{\po}$ to a measure in $M(\overline{\Omega} \times \S)$ by setting $\mu^{\po}_t(A) \coloneqq \mu_t^{\po} (A \cap (\po \times \S))$ for any Borel set $A \subset \overline{\Omega} \times \S$ and define 
\begin{equation*}
	\mu \coloneqq \mathcal{L}^1{}_{\llcorner}(0, T_*) \otimes (\mu_t)_{t \in (0, {T_*})}
	\quad \textnormal{with}\quad 
	\mu_t = \mu_t^\Omega + (\cos \alpha) \mu_t^{\pO} \in M(\overline{\Omega} \times \S).  
\end{equation*}
In particular, by the weak* convergence \eqref{eq:conv_mu_measure} and linearity of these limits one can further show that 
\begin{equation}\label{eq:conv_mu_sum}
	\mu^h \xrightharpoonup{*} \mu \quad \textnormal{in} \quad M((0, T_*) \times \overline{\Omega} \times \S). 
\end{equation}

\subsubsection*{Energy}
We consider the following family of non-increasing (cf.\ \eqref{eq:F_non-increasing}) mappings
\begin{equation*}
	(0, T_*) \ni t 
	\mapsto 
	e^h(t)
	\coloneqq
	\int_{\overline{\Omega}} d|\mu^h_t| 
	+ \int_\Omega W(\bar{\chi}^h, \bar{\bm{u}}^h)(t) \dx 
	+ \int_\Omega \tfrac{1}{2} | \nabla^2 \bar{\bm{u}}^h |^2(t) \dx.
\end{equation*}
By Helly's selection theorem, we may find a non-increasing map $e : (0, T_*) \to [0, \infty)$ such that 
\begin{equation}\label{eq:energy_conv_point}
	e^h(t) \to e(t), \quad \textnormal{for all } t \in (0, T_*),
\end{equation}
along a not relabeled subsequence of $h \searrow 0$. Note that this can be done without compromising the compactness results we have derived above, and observe that \eqref{eq:chi_conv_strong}, \eqref{eq:u_conv_strong} and \eqref{eq:u_conv_H2} already entail 
\begin{equation*}
	\int_\Omega W(\bar{\chi}^h, \bar{u})^h(t) \dx 
	+ \int_\Omega \tfrac{1}{2} | \nabla^2 \bar{\bm{u}}^h |^2(t) \dx 
	\xrightarrow{h \to 0}
	\int_\Omega W({\chi}, {u})(t) \dx 
	+ \int_\Omega \tfrac{1}{2} | \nabla^2 {\bm{u}} |^2(t) \dx
	\quad \textnormal{for a.e.\ } t \in (0, T_*). 
\end{equation*}
Thus, it follows that 
\begin{equation}\label{eq:mu_mass_point}
	\int_{\overline{\Omega}} d|\mu^h_t|  \xrightarrow{h \to 0} e_\mu(t) \coloneqq e(t) 
	- \int_\Omega W({\chi}, {u})(t) \dx 
	- \int_\Omega \tfrac{1}{2} | \nabla^2 {\bm{u}} |^2(t) \dx \quad \textnormal{for a.e.\ } t \in (0, T_*). 
\end{equation}
The definition of $W$ and the fact that $\bm{u} \in L^\infty(0, T_*; \bm{H}^2(\Omega))$ then imply $e_\mu \in L^\infty(0, T_*)$.
Hence, the time-sliced representation of $\mu$ as $\mathcal{L}^1{}_{\llcorner}(0, T_*) \otimes (\mu_t)_{t \in(0, T_*)}$ is a consequence of the arguments following \eqref{eq:conv_point_mass_measure}, but that merely allows us to conclude that $e_\mu(t) \geq |\mu_t^\Omega| (\overline{\Omega})$ and $e_\mu(t) \geq |\mu_t^{\po}| (\overline{\Omega})$ for almost every  $t \in (0, T_*)$. To obtain equality for the sum, let $\eta \in C^\infty_c(0, T_*)$; using that $|\mu^h_t| (\overline{\Omega})$ is uniformly bounded and converges pointwise almost everywhere, see \eqref{est:a_priori_mm} and \eqref{eq:mu_mass_point}, respectively, we can invoke dominated convergence to obtain 
\begin{equation*}
	\int_0^{T_*} |\mu^h_t| (\overline{\Omega}) \, \eta(t) \dt 
	\to 
	\int_0^{T_*} e_\mu(t)\, \eta(t) \dt. 
\end{equation*}
Conversely, weak convergence of $\mu^{\Omega, h}$ and $\mu^{\po, h}$ as measures, see \eqref{eq:conv_mu_measure}, and the definition of $\mu$ yield
\begin{align*}
	\int_0^{T_*} |\mu^h_t| (\overline{\Omega})\,  \eta(t) \dt  
	&=\begin{aligned}[t]
		&  \int_{ (0, T_*) \times \Omega \times \mathbb{S}^{d-1}} \eta(t) \, 
		d(\mathcal{L}^1{}_{\llcorner}(0, T_*) \otimes \mu^{h, \Omega}_t ) (t, \bm{x}, \bm{s})\\ 
		&+  (\cos \alpha) \int_{ (0, T_*) \times \po \times \mathbb{S}^{d-1}} \eta(t) \, 
		d(\mathcal{L}^1{}_{\llcorner}(0, T_*) \otimes \mu^{h, \po }_t ) (t, \bm{x}, \bm{s})
	\end{aligned}
	\\
	&\to
	\int_0^{T_*}   \abs{\mu^\Omega_t}(\overline{\Omega}) +  (\cos \alpha) \abs{\mu^{\po}_t}(\overline{\Omega}) \dt 
	= \int_0^{T_*} |\mu_t| (\overline{\Omega})\,  \eta(t) \dt . 
\end{align*}
By comparing these two results and appealing to the fundamental lemma of the calculus of variations, we conclude $ \abs{\mu^\Omega_t}(\overline{\Omega}) +  (\cos \alpha) \abs{\mu^{\po}_t}(\overline{\Omega})  = |\mu_t| (\overline{\Omega}) = e_\mu(t)$ for a.e.\ $t \in (0, T_*)$.
\par 
\medskip

\subsection{Limit process}

\subsubsection*{Energy dissipation inequality} 
The goal of this subsection to pass to the limit in the approximate energy dissipation inequality \eqref{eq:dissipation_approximate}. To this end, let us first investigate the following claim.
\par 
\medskip \noindent
\textit{Claim:} For all $0 < \kappa < \tau < T_*$, and a not relabeled subsequence of $h \searrow 0$, it holds that 
\begin{equation}\label{eq:liminf_pot_u}
	\begin{aligned}[t]
		\int_s^\tau \norm{\pt \bm{u}}_{H^1_{\chi(t)}}^2 \dt 
		&\leq 
		\underset{h \searrow 0}{\textnormal{liminf}}
		\int_s^\tau \norm{\pt \hat{\bm{u}}^h}_{\bm{H}^1_{\chi^h(t)}}^2 \dt, \\ 
		\int_s^\tau \norm{\pt \bm{u}}_{H^1_{\chi(t)}}^2 \dt 
		= \int_s^\tau \norm{{\mathfrak{U}}}_{H^1_{\chi(t)}}^2 \dt 
		&\leq 
		\underset{h \searrow 0}{\textnormal{liminf}}
		\int_s^\tau \norm{\tfrac{\bar{\bm{u}}^h   -  \bar{\bm{u}}^h ( \lfloor \sfrac{t}{h} \rfloor h)}{ {t - \lfloor \sfrac{t}{h} \rfloor h }}}_{\bm{H}^1_{\chi^h(t)}}^2 \dt . 
	\end{aligned}
\end{equation}

\par 
\medskip \noindent 
\textit{Proof of claim:} This result will follow immediately from \cite[Cor.\ 4.4]{MR2425653} once we have show that for almost all $t \in (0, T_*)$ and all $\bm{v} \in \bm{H}^{1}_{\perp}(\Omega)$ one has
\begin{equation}\numberthis \label{eq:liminf_norms_n}
	\norm{\bm{v}}_{\bm{H}^1_{\chi(t)}}^2 \leq \inf \Big\{\underset{n \to \infty }{\textnormal{liminf}} \norm{\bm{v}_n}_{\bm{H}^1_{\chi^{h_n}(t)}}^2 \colon  \bm{v}_n \rightharpoonup \bm{v} 
	\textnormal{ in } \bm{H}^{1}_{\perp}(\Omega)
	\Big\}, 
\end{equation}
where $(h_n)_{n \in \N}$ denotes a subsequence of $h\searrow 0$ that does not depend on $t \in (0, T_*)$. For the sake of readability we will use $\norm{\cdot }_{\bm{H}^1_{\chi_n(t)}} \coloneqq \norm{\cdot }_{\bm{H}^1_{\chi^{h_n}(t)}}$ in subsequent arguments. Denoting by $(\cdot, \cdot )_{\bm{H}^1_{\chi_n(t)}}$ for the inner product that induces the norm $\norm{\cdot }_{\bm{H}^1_{\chi_n(t)}}$, one finds 
\begin{align*}
	\norm{\bm{v}_n}_{\bm{H}^1_{\chi_n(t)}}^2
	= \langle\bm{v}_n, \bm{v}_n \rangle_{\bm{H}^1_{\chi_n(t)}}
	&= \langle\bm{v}, \bm{v} \rangle_{\bm{H}^1_{\chi_n(t)}}
	+ 2 \langle\bm{v} , \bm{v}_n- \bm{v} \rangle_{\bm{H}^1_{\chi_n(t)}}
	+ \langle \bm{v}_n- \bm{v}  , \bm{v}_n- \bm{v} \rangle_{\bm{H}^1_{\chi_n(t)}} . 
\end{align*}
Due to strong convergence of the order parameter $\chi^h$ in $L^2(0, T_*; L^2(\Omega)) \cong L^2(\Omega_{T_*})$ see \eqref{eq:chi_conv_strong}, we can assume that $\chi^h \to \chi$ pointwise almost everywhere in $\Omega_{T_*}$. In particular, for almost all $t \in (0, T_*)$ we have pointwise convergence almost everywhere in $\Omega$. Invoking Fatou's lemma \cite[Thm.\ 1.20]{ambrosio2000functions} therefore implies 
\begin{align*}
	\underset{n \to \infty}{\textnormal{liminf }} \langle\bm{v}, \bm{v} \rangle_{\bm{H}^1_{\chi_n(t)}} 
	= \underset{n \to \infty}{\textnormal{liminf }} 
	\int_\Omega \C_\nu (\chi^{h_n}(t)) \E(\bm{v}) \colon \E(\bm{v}) \dx 
	\geq \int_\Omega \C_\nu (\chi(t)) \E(\bm{v}) \colon \E(\bm{v}) \dx 
	= \norm{\bm{v}}_{\bm{H}^1_{\chi(t)}}^2 . 
\end{align*}
Thus, superadditivity of the limit inferior and the weak convergence $\bm{v}_n- \bm{v} \rightharpoonup \bm{0}$, along with boudedness of weakly convergent sequences and the almost everywhere convergence of $\chi^{h_n} \to \chi$, yield 
\begin{align*}
	\underset{n \to \infty}{\textnormal{liminf }} \norm{\bm{v}_n}_{\bm{H}^1_{\chi_n(t)}}^2
	&= 	\underset{n \to \infty}{\textnormal{liminf }} \Big( 
	\langle\bm{v}, \bm{v} \rangle_{\bm{H}^1_{\chi_n(t)}}
	+ 2 \langle\bm{v} , \bm{v}_n- \bm{v} \rangle_{\bm{H}^1_{\chi_n(t)}}
	+ \langle \bm{v}_n- \bm{v}  , \bm{v}_n- \bm{v} \rangle_{\bm{H}^1_{\chi_n(t)}}
	\Big)\\ 
	& \geq 
	\norm{\bm{v}}_{\bm{H}^1_{\chi(t)}}^2 
	+ \underset{n \to \infty}{\textnormal{liminf }} \langle \bm{v}_n- \bm{v}  , \bm{v}_n- \bm{v} \rangle_{\bm{H}^1_{\chi_n(t)}},
\end{align*}
where the last term is certainly non-negative. This shows that the inequality in \eqref{eq:liminf_norms_n} holds for any sequence $(\bm{v}_n)_{n \in \N}$ such that $\bm{v}_n \rightharpoonup \bm{v}$ and therefore in particular for the infiumum, as asserted. \\
Convexity and continuity of norms, along with the weak convergences of $\pt \bar{\bm{u}}^h$ and $\tfrac{\bar{\bm{u}}^h   -  \bar{\bm{u}}^h ( \lfloor \sfrac{t}{h} \rfloor h)}{ {t - \lfloor \sfrac{t}{h} \rfloor h }}$ 
in $L^2(0, T_*; \bm{H}^1_{\perp}(\Omega))$ then allow us to apply  \cite[Cor.\ 4.4]{MR2425653} and deduce the assertion. 
\hfill $\diamondsuit$
\par 
\medskip 
Recall that we have already shown, see \eqref{eq:energy_conv_point} and succeeding arguments, that for almost all $t \in (0, T_*)$ the energy converges pointwise, i.e., 
\begin{equation*}
	\int_{\overline{\Omega}} d|\mu^h_t| 
	+ \int_\Omega W(\bar{\chi}^h, \bar{u}^h)(t) \dx 
	+ \int_\Omega \tfrac{1}{2} | \nabla^2 \bar{\bm{u}}^h |^2(t) \dx 
	\to 
	\int_\Omega d|\mu_t| 
	+ \int_\Omega W(\chi, \bm{u})(t) \dx 
	+ \int_\Omega \tfrac{1}{2} | \nabla^2 \bm{u} |^2(t) \dx. 
\end{equation*}
For the sake of brevity, and in a slight abuse of notation, let us write 
\begin{equation*}
	\F(\mu, \chi, \bm{u}) (t) 
	\coloneqq 
	\int_{\overline{\Omega}} d|\mu_t| 
	+ \int_\Omega W(\chi, \bm{u})(t) \dx 
	+ \int_\Omega \tfrac{1}{2} | \nabla^2 \bm{u} |^2(t) \dx. 
\end{equation*}
Along with \eqref{eq:liminf_pot_u}, the weak convergence of $w^h$ and $\pt \hat{\chi}^h$, see \eqref{eq:conv_w} and \eqref{eq:conv_dt_chi} respectively, combined with weak lower semicontinuity of norms, we can first pass to the limit in \eqref{eq:dissipation_approximate_w} for $h \searrow 0$ and than take $\tau \nearrow T$ along with $\kappa \searrow s$ to obtain 
\begin{align*}
	\F(\mu, \chi, \bm{u}) (T)
	&+ \frac{1}{2}   \int_s^T
	\norm{\pt \chi}_{H^{-1}_{(0)}}^2 
	+  \norm{ \pt \bm{u}}^2_{\bm{H}^1_{{\chi}(t) }}  \dt \\ 
	&+ \frac{1}{2}   \int_s^T 
	\norm{ \nabla w  }^2_{L^2}  
	+   \norm{ \pt \bm{u}}^2_{\bm{H}^1_{{\chi}(t) }}  \dt 
	\leq \F(\mu, \chi, \bm{u}) (s) \leq \F(\chi_0, \bm{u}_0)
\end{align*}
for all almost all $0 < s < T < T_*$.\\
We emphasize that the weak limits of $\pt \hat{\bm{u}}^h$ and $\tfrac{\bar{\bm{u}}^h   -  \bar{\bm{u}}^h ( \lfloor \sfrac{t}{h} \rfloor h)  }{ {t - \lfloor \sfrac{t}{h} \rfloor h } }$ coincide, see  \eqref{eq:u_limit_reduced}. This explains why $\pt \bm{u}$ appears twice and we do not obtain two different terms similar to $\pt \chi$ and $\nabla w$. 
\par 

Lastly, the measureability of the total mass measures is a direct consequence of \eqref{eq:F_non-increasing}, since the second part of the energy $\int_\Omega W(\chi, \bm{u})(t) \dx + \int_\Omega \tfrac{1}{2} | \nabla^2 \bm{u} |^2(t) \dx$ is also measurable.

\subsubsection*{Weak formulation}  

Note that we have already established the weak formulation for the time derivative of the phase indicator \eqref{eq:dt_chi} and the elasticity equation \eqref{eq:elastic_thm} in \eqref{eq:conv_dt_chi} and \eqref{eq:elastic_limit_1}, respectively. After testing \eqref{el_w_appoximate} with any $\bm{B} \in C^\infty_c([0, T_*]; \bm{C}^2(\overline{\Omega}))$, $\bm{B}_{|\po}(t) \cdot \bm{n}_{\po} = 0$, and integrating over $(0, T_*)$, the compactness results in \eqref{eq:conv_w}, \eqref{eq:chi_conv_strong}, \eqref{eq:u_pot_conv}, \eqref{eq:u_conv_H2} and \eqref{eq:conv_mu_sum} allow us to pass to the limit in \eqref{el_w_appoximate}, and we obtain \eqref{eq:potential_thm} after substituting $\mathfrak{U}$ for $\pt \bm{u}$ and localizing (see also \cite[Eq.\ (118)]{hensel_stinson}).\\ 
Let $\bm{\eta} \in C^\infty_c(0, T_*; \bm{C}^1(\overline{\Omega}))$, $\bm{\eta}_{|\po}(t) \cdot \bm{n}_{\po} = 0$. By definition of $\mu^{\Omega, h}$, see \eqref{eq:def_mu_Omega}, it obviously holds for all $h > 0$ that 
\begin{align*}
	\int_0^{T_*} \int_\Omega  \bm{\eta} (t, \bm{x}) \cdot d\nabla \bar{\chi}^h(t) \dt 
	&=
	\int_0^{T_*} \int_{\Omega \times \mathbb{S}^{d-1}}  \bm{\eta}(t, \bm{x}) \cdot \bnu \,  d\abs{\nabla \bar{\chi}^h(t)} \dt \\
	&= \int_0^{T_*} \int_{\overline{\Omega} \times \mathbb{S}^{d-1}}  \bm{\eta}(t, \bm{x}) \cdot \bm{s} \,d\mu^{h, \Omega}_t(\bm{x}, \bm{s}) \dt . 
\end{align*}
By passing to the limit using \eqref{eq:chi_conv_strong} and \eqref{eq:conv_mu_measure}, we obtain a time integrated version of \eqref{eq:compatibility_thm}. Finally, we can exploit the separability of $\bm{C}^1(\overline{\Omega})$, apply the fundamental lemma of the calculus of variations and conclude the proof of \eqref{eq:compatibility_thm} by a density argument.\\ 
Moreover, let $\phi \in C^0(\overline{\Omega})$, and fix a vector field $\bm{\xi} \in \bm{C}^1(\overline{\Omega})$ with $\bm{\xi}_{|\po} \cdot \bm{n}_{\po} = \cos \alpha$; finally, let $\eta \in C^\infty_c( (0, T_*))$. By the definitions of $\mu^{\Omega, h}$ and $\mu^{\po, h}$, we find 
\begin{align*}
	&\int_0^{T_*} \eta(t) \int_{\overline{\Omega} \times \S} \phi(\bm{x}) \bm{\xi} (\bm{x}) \cdot \bm{s} \, d \mu_t^{h, \Omega} \dt 
	+ (\cos \alpha) \int_0^{T_*} \int_{\po} \phi(\bm{x}) \, d \abs{\mu_t^{h, \po}} \dt \\ 
	& \quad = 
	\int_0^{T_*} \eta(t) \int_{{\Omega}} \phi(\bm{x}) \bm{\xi} (\bm{x}) \cdot d \nabla\bar{\chi}^h \dt 
	+  \int_0^{T_*} \eta(t)  \int_{\po} \bar{\chi}^h (\bm{x})  \phi(\bm{x})  \bm{\xi} (\bm{x}) \cdot \bm{n}_{\po} \dH \dt \\ 
	&	\quad = - \int_0^{T_*} \eta(t) \int_\Omega \bar{\chi}^h \nabla \cdot (\phi \bm{\xi}) \dx \dt. 
\end{align*}
After passing to the limit, localizing, a continuity argument, and once again applying the formula for integration by parts of $BV$-functions, it follows for almost all $t \in (0, T_*)$, and any $\phi$ and $\bm{\xi}$ as above, that 
\begin{align*}
	& \int_{\overline{\Omega} \times \S} \phi(\bm{x}) \bm{\xi} (\bm{x}) \cdot \bm{s} \, d \mu_t^{ \Omega} 
	+ (\cos \alpha) \int_{\po} \phi(\bm{x}) \, d \abs{\mu_t^{ \po}}\\ 
	& \quad = 
	\int_{{\Omega} } \phi(\bm{x}) \bm{\xi} ( \bm{x}) \cdot d \nabla{\chi}
	+  \int_{\po} {\chi} (\bm{x})  \phi(t, \bm{x})  \bm{\xi} (\bm{x}) \cdot \bm{n}_{\po} \dH, 
\end{align*}
where we also used that $\chi \in L^\infty_{w^*}(0, T; BV(\Omega; \{0, 1\}))$ has a trace on the boundary $\po$. Letting $\phi$ be the identity then implies \eqref{eq:compatibility_thm}. On the other hand by letting $\bm{\xi}$ converge to $(\cos \alpha) \bm{n}_{\po}$ pointwise everywhere and varying $\phi \in C^0(\overline{\Omega})$ yields \eqref{eq:boundary_measure_inequaltiy}.

\subsection{Proof of further properties}\label{sec:further_prop}

\begin{proof}[Proof of Corollary \ref{cor:energy_dissipation}]
	The proof of this assertion follows Hensel and Stinson \cite[Step 9]{hensel_stinson}. 
	First of all, let us observe that a standard localization argument implies that \eqref{eq:elastic_thm} must already hold pointwise for almost all $t \in (0, T_*)$.
	Thus, consider $t \in (0, T_*)$ such that both \eqref{eq:potential_thm} and \eqref{eq:elastic_thm} hold. For $\bm{B} \in \mathcal{S}_{\chi(t)}$, let $w_{\bm{B}} \in H^1_{(0)}(\Omega)$ be the unique solution to the Neumann problem
	\begin{equation}\label{eq:potential_B}
		\left\{ 
		\begin{alignedat}{2}
			\Delta w_{\bm{B}} &= \bm{B} \cdot \nabla \chi (t, \cdot) \quad && \textnormal{in } \Omega, \\ 
			\nabla w_{\bm{B}} \cdot \bm{n}_{\po} &= 0 \quad && \textnormal{on } \po, 
		\end{alignedat}
		\right. 
	\end{equation} 
	where 
	\begin{equation*}
		{}_{H^{-1}_{(0)}}\langle \bm{B} \cdot \nabla {\chi}, \xi \rangle_{H^1_{(0)}} \coloneqq -  \int_\Omega {\chi} \ \div (\xi \bm{B}) \dx 
		\quad \textnormal{for all}\quad \xi \in H^1_{(0)}(\Omega). 
	\end{equation*} 
	Recalling the definition of the $H^{-1}_{(0)}$-norm, we find that $\norm{\nabla w_{\bm{B}}}_{L^2(\Omega)} = \norm{\bm{B} \cdot \nabla \chi(t, \cdot)}_{H^{-1}_{(0)}}$. Moreover, it holds that 
	\begin{equation}\label{eq:energy_dissipation_1}
		\norm{\nabla w(t, \cdot)}_{L^2}^2 \geq \bigg( \sup_{\bm{B} \in \mathcal{S}_\chi(t, \cdot)} \frac{\int_\Omega \nabla w (t, \cdot) \cdot \nabla w_{\bm{B}} \dx }{ \norm{\nabla w_{\bm{B}}}_{L^2} } \bigg)^2 . 
	\end{equation}
	Using \eqref{eq:potential_thm}, one finds the identity 
	\begin{align*}
		\int_\Omega& \nabla w \cdot \nabla w_{\bm{B}} \dx  = \int_\Omega \chi(t, \cdot)\,  \div \big( w\bm{B} \big) \dx \numberthis \label{eq:energy_dissipation_2}\\ 
		&= \int_{\overline{\Omega} \times \S} (\bm{I} - \bm{s} \otimes \bm{s}) \colon \nabla \bm{B} \, d \mu_t (\bm{x}, \bm{s}) 
		- \int_\Omega \C_{\nu}(\chi)  \E(\pt \bm{u})
		\colon \nabla^2 \bm{u} \bm{B}  \dx 
		+\int_\Omega \tfrac{1}{2} \abs{\nabla^2 \bm{u}}^2\  \div \bm{B} \dx \\ 
		&+  \int_\Omega \Big[ W(\chi, \E(\bm{u})) \bm{I} -  ( \nabla \bm{u} )^T \big( \C_{\nu}(\chi)  \E (\pt \bm{u}) +  W_{, \E} (\chi, \E(\bm{u})  \big)\Big] \colon \nabla \bm{B} \dx 
		\\ 
		&- \int_\Omega  (\nabla^2 \bm{u})_{ijk} (\nabla \bm{u})_{ip}\,(\nabla^2 \bm{B})_{pjk}  \dx 
		- \int_\Omega(\nabla^2 \bm{u})_{ijk}  \big( (\nabla^2 \bm{u})_{ipk}\,(\nabla\bm{B})_{jp} 
		+ (\nabla^2 \bm{u})_{ijp}\,(\nabla\bm{B})_{kp} \big) \dx. 
	\end{align*}
	Moreover, let $\bm{v} \in \bm{H}^1_\perp(\Omega) \cap \bm{H}^2(\Omega)$, then the pointwise variant of \eqref{eq:elastic_thm} implies 
	\begin{align*}
		-\int_\Omega 
		\C_{\nu}(\chi) \E ( \pt \bm{u}) 
		&\colon \E(\bm{v}  - \nabla \bm{u} \bm{B})  \dx \\ 
		&= 
		\int_\Omega 
		\C_{\nu}(\chi) \E(\pt \bm{u}) \colon \E (\nabla \bm{u} \bm{B}) + 
		\WE ({\chi}, \E({\bm{u}}))
		\colon \E(\bm{v}) 
		+ \nabla^2 {\bm{u}}\colon  \nabla^2 \bm{v} \dx, \label{eq:energy_dissipation_3} \numberthis 
	\end{align*}
	where we first tested with $\bm{v}$, and additionally added the term $ \int_\Omega \C_{\nu}(\chi) \E ( \pt \bm{u}) \colon \E(\nabla \bm{u} \bm{B})  \dx$ on both sides. 
	As previously noted, the mapping 
	\begin{equation*}
		(\bm{w}, \bm{v}) \mapsto \int_\Omega \C_{\nu}(\chi(t, \cdot)) \E(\bm{w}) \colon \E(\bm{v}) \dx
	\end{equation*}
	is a positive semi-definite bilinear form on $\bm{H}^1(\Omega)$. In particular, the Cauchy--Schwarz inequality yields  
	\begin{equation*}
		\int_\Omega \C_{\nu}(\chi(t, \cdot)) \E(\bm{w}) \colon \E(\bm{v}) \dx \leq \norm{\bm{w}}_{\bm{H}^1_{\chi(t)}}  \norm{\bm{v}}_{\bm{H}^1_{\chi(t)}} , 
	\end{equation*}
	and since $\bm{v} - \nabla \bm{u}\bm{B} \in \bm{H}^1(\Omega)$, we deduce 
	\begin{equation}\label{eq:energy_dissipation_4}
		\norm{\pt \bm{u}}_{\bm{H}^1_{\chi(t)}}^2 \geq  \Bigg( \sup_{\substack{\bm{v} \in \bm{H}^1_\perp(\Omega) \cap \bm{H}^2(\Omega) \\ \bm{B} \in \mathcal{S}_{\chi(t, \cdot)}}} \frac{- \int_\Omega 
			\C_{\nu}(\chi) \E ( \pt \bm{u}) 
			\colon \E(\bm{v}  - \nabla \bm{u} \bm{B})  \dx }{\norm{v - \nabla \bm{u}\bm{B}}_{\bm{H}^1_{\chi(t)}}}  \Bigg)^2 . 
	\end{equation}
	Combining \eqref{eq:energy_dissipation_1}-\eqref{eq:energy_dissipation_4}, along with the elementary relation $\tfrac{1}{2} (a/b)^2 \geq a- \tfrac{1}{2} b^2$, we conclude that   
	\begin{align*}
		\frac{1}{2}&\norm{\nabla w(t, \cdot)}_{L^2}^2 + \frac{1}{2} \norm{\pt \bm{u}(t, \cdot)}_{\bm{H}^1_{\chi(t)}}^2
		\numberthis \label{eq:energy_cor_1}\\ 
		&\geq \frac{1}{2} \Bigg( 
		\frac{\int_\Omega \nabla w (t, \cdot) \cdot \nabla w_{\bm{B}}}{ \norm{\nabla w_{\bm{B}}}_{L^2} } \bigg)^2
		+ \frac{1}{2} \bigg(
		\frac{-\int_\Omega 
			\C_{\nu}(\chi) \E ( \pt \bm{u})  \colon \E(\bm{v}  - \nabla \bm{u} \bm{B})  \dx }{\norm{v - \nabla \bm{u}\bm{B}}_{\bm{H}^1_{\chi(t)}}}  \Bigg)^2 \\ 
		&\geq \begin{aligned}[t]
			& \int_{\overline{\Omega} \times \mathbb{S}^{d-1}} (\bm{I} - \bm{s} \otimes \bm{s}) \colon \nabla \bm{B} \, d\mu (\bm{x}, \bm{s}) \\ 
			&+ \int_\Omega \big[ W(\chi, \E(\bm{u})) \bm{I} -  (\nabla {\bm{u}} )^T W_{, \E} (\chi, \E(\bm{u})  \big)\big] \colon \nabla \bm{B} \dx 
			+ \int_\Omega \tfrac{1}{2} \abs{\nabla^2 \bm{u}}^2\  \div \bm{B} \dx \\ 
			&-\int_\Omega  (\nabla^2 \bm{u})_{ijk} (\nabla \bm{u})_{ip}\,(\nabla^2 \bm{B})_{pjk}  \dx 
			- \int_\Omega(\nabla^2 \bm{u})_{ijk}   (\nabla^2 \bm{u})_{ipk}\,(\nabla\bm{B})_{jp}  \dx \\
			& - \int_\Omega(\nabla^2 \bm{u})_{ijk}  (\nabla^2 \bm{u})_{ijp}\,(\nabla\bm{B})_{kp}  \dx
			+ \int_\Omega \WE (\chi, \E(\bm{u})) \colon \nabla \bm{v} \dx 
			+\int_\Omega \nabla^2 \bm{u} \colon   \nabla^2 \bm{v} \dx\\ 
			& - \frac{1}{2} \norm{\bm{B}\cdot  \nabla \chi}_{H^{-1}_{(0)}}^2 - \frac{1}{2}\norm{\bm{v} - \nabla \bm{u}\bm{B}}_{\bm{H}^1_{\chi(t)}}^2. 
		\end{aligned}
	\end{align*}
	Noting that $\bm{B}$ and $\bm{v}$ were chosen arbitrarily and that this inequality holds for almost all $t \in (0, T_*)$, we can choose any measurable map $\bm{v} \colon (0, T_*) \to  \bm{H}^2(\Omega)$ with $\bm{v}(t) \in \bm{H}^1_\perp (\Omega)$ for all $t \in (0, T_*)$ and $\norm{\bm{v}(t)}_{\bm{H}^2(\Omega)} \in L^2(0, T_*)$, and any measurable function $\bm{B} \colon (0, T_*) \to \bm
	C^2(\overline{\Omega})$ such that $\bm{B}(t) \in \mathcal{S}_{\chi(t)}$ for all $t \in (0, T_*)$ and $\norm{\bm{B}}_{\bm{C}^2} \in L^2(0, T_*)$, and insert it into \eqref{eq:energy_dissipation_thm}, finishing the proof of \eqref{eq:energy_dissipation_cor}.\par 
	To show the second inequality \eqref{eq:energy_dissipation_cor_sup}, we note that since \eqref{eq:energy_cor_1} already holds for all $\bm{B} \in \mathcal{S}_{\chi(t, \cdot)}$ and all $\bm{v} \in \bm{H}^1_\perp(\Omega)\cap \bm{H}^2(\Omega)$, it follows by application of the formula for the first variation of $\F$, see \eqref{eq:first_variation_summary}, that 
	\begin{align*}
		\frac{1}{2}\norm{\nabla w(t, \cdot)}_{L^2}^2 &+ \frac{1}{2} \norm{\pt \bm{u}(t, \cdot)}_{\bm{H}^1_{\chi(t)}}^2  \\
		&\geq 
		\frac{1}{2}
		\sup_{\substack{\bm{B} \in \mathcal{S}_{\chi(t, \cdot)} \\ \bm{v} \in \bm{H}^1_\perp(\Omega) \cap \bm{H}^2(\Omega)}} 
		\Big\{  \delta \F(\mu, \chi, \bm{u} )[\bm{B}, \bm{v}] 
		- \frac{1}{2} \norm{\bm{B}\cdot  \nabla \chi}_{H^{-1}_{(0)}}^2 - \frac{1}{2}\norm{\bm{v} - \nabla \bm{u}\bm{B}}_{\bm{H}^1_{\chi(t)}}^2 \Big\} . 
	\end{align*}
	Thus, it remains to show that the right-hand side constitutes a measurable function in $t \in (0, T_*)$. Following \cite{hensel_stinson}, we first construct an auxiliary function $\bm{\xi}$, see \cite[Proof of Lem.\ 9]{hensel_stinson}, Proof of \eqref{eq:langrange_mult_sg} and \eqref{eq:lagrange_mult_mm}, such that $\bm{\xi} \in L^2(0, T_*; \bm{C}^2(\overline{\Omega}))$. More precisely, we consider the elliptic boundary value problem 
	\begin{equation*}
		\left\{ 
		\begin{alignedat}{2}
			\Delta \phi_\eps &= \chi_\eps - m_\eps \quad && \textnormal{in } \Omega, \\ 
			\nabla \phi_\eps \cdot \bm{n}_{\po } &= 0 \quad && \textnormal{on } \po, \\
			\dashint_\Omega \phi_\eps \dx &= 0, 
		\end{alignedat}
		\right. 
	\end{equation*}
	where $\chi_\eps \coloneqq \chi * \rho_\eps$ for a standard mollifier $\rho_\eps$, and $m_\eps \coloneqq \dashint_\Omega \chi_\eps \dx$; here, $\eps > 0$ is a suitably chosen parameter. Then, there exists a lifting $\phi_\eps \in L^2(0, T; C^3(\overline{\Omega}))$ to the the spatial boundary value problem in the precise sense that 
	\begin{equation*}
		- \int_0^{T_*} \int_\Omega \nabla \phi_\eps \cdot \nabla \eta \dx \dt 
		= \int_0^{T_*} \int_\Omega (\chi_\eps - m_\eps) \eta \dx \dt 
	\end{equation*}
	for all $\eta \in C^0_c(0, T_*; C^1(\overline{\Omega}))$. We conclude the construction by setting $\bm{\xi} \coloneqq \nabla \phi_\eps$ . As in \cite[Eq.\ (67)]{hensel_stinson}, the fact that $\chi \in L^\infty (0, T; BV(\Omega, \{0, 1\}))$, which implies a uniform bound for $|\nabla \chi| (\Omega)$, gives rise to some $\eps > 0$  and a constant $C(m_0, \Omega) > 0$, such that 
	\begin{equation*}
		\substack{\textnormal{\normalsize ess inf}\\ t \in (0, T_*)} \int_\Omega \chi(t, \cdot) \nabla \cdot \bm{\xi}(t, \cdot) \dx  
		\geq C(m_0, \Omega)
		> 0.  
	\end{equation*}
	Similar to previous arguments, cf.\ \eqref{eq:B_C^1_in_S}, we define for all $\bm{B} \in \bm{C}^2(\overline{\Omega})$ with $\bm{B}_{|\po} \cdot \bm{n}_{\po}= 0$ the function 
	\begin{equation*}
		t \mapsto \widetilde{\bm{B}}_{\bm{\xi}}(t, \cdot ) \coloneqq \bm{B} - \frac{ \int_\Omega \chi(t, \cdot) \nabla \cdot \bm{B} \dx }{ \int_\Omega \chi(t, \cdot) \nabla \cdot \bm{\xi} \dx} \bm{\xi} (t, \cdot) \in \mathcal{S}_{\chi(t, \cdot)}. 
	\end{equation*}
	Note that by definition \eqref{eq:def_S_chi}, this map is the identity if $\bm{B} \in \mathcal{S}_{\chi(t, \cdot)}$. In particular, we obtain that 
	\begin{align*}
		&\sup_{\substack{\bm{B} \in \mathcal{S}_{\chi(t, \cdot)} \\ \bm{v} \in \bm{H}^1_\perp(\Omega) \cap \bm{H}^2(\Omega)}} 
		\Big\{  \delta \F(\mu, \chi, \bm{u} )[\bm{B}, \bm{v}] 
		- \frac{1}{2} \norm{\bm{B}\cdot  \nabla \chi}_{H^{-1}_{(0)}}^2 - \frac{1}{2}\norm{\bm{v} - \nabla \bm{u}\bm{B}}_{\bm{H}^1_{\chi(t)}}^2 \Big\}\\ 
		&\quad =
		\sup_{\substack{\bm{B}\in \bm{C}^2(\overline{\Omega}), \bm{B}_{|\po} \cdot \bm{n}_{\po}= 0 \\ \bm{v} \in \bm{H}^1_\perp(\Omega) \cap \bm{H}^2(\Omega)}} 
		\Big\{  \delta \F(\mu, \chi, \bm{u} )[\widetilde{\bm{B}}_{\bm{\xi}}(t, \cdot ), \bm{v}] 
		- \frac{1}{2} \norm{\widetilde{\bm{B}}_{\bm{\xi}}(t, \cdot )\cdot  \nabla \chi}_{H^{-1}_{(0)}}^2 - \frac{1}{2}\norm{\bm{v} - \nabla \bm{u}\widetilde{\bm{B}}_{\bm{\xi}}(t, \cdot )}_{\bm{H}^1_{\chi(t)}}^2 \Big\}
	\end{align*}
	for almost all $t \in (0, T_*)$. By construction, we have $\widetilde{\bm{B}}_{\bm{\xi}}(t, \cdot ) \in L^2(0, T; \bm{C}^2(\overline{\Omega}))$, and deduce with the help of \eqref{eq:first_variation_summary} and \eqref{eq:potential_thm} that 
	\begin{align*}
		t \mapsto  &\delta \F(\mu, \chi, \bm{u} )[\widetilde{\bm{B}}_{\bm{\xi}}(t, \cdot ), \bm{v}] 
		= \int_\Omega \chi\,  \div(w \widetilde{\bm{B}}_{\bm{\xi}}(t, \cdot )) \dx 
		+ \int_\Omega \WE (\chi, \E(\bm{u})) \colon \nabla \bm{v} \dx 
		+\int_\Omega \nabla^2 \bm{u} \colon   \nabla^2 \bm{v} \dx 
	\end{align*}
	is measurable for any fixed $\bm{B} \in \bm{C}^2(\overline{\Omega})$ with $\bm{B}_{|\po} \cdot \bm{n}_{\po}= 0$ and any $\bm{v} \in \bm{H}^1_\perp(\Omega) \cap \bm{H}^2(\Omega)$. Likewise, measurability follows for 
	\begin{equation*}
		t \mapsto \norm{\bm{v} - \nabla \bm{u}\widetilde{\bm{B}}_{\bm{\xi}}(t, \cdot )}_{\bm{H}^1_{\chi(t)}}^2 
		= \int_\Omega \C_{\nu}(\chi) \E(\bm{v} - \nabla \bm{u}\widetilde{\bm{B}}_{\bm{\xi}}(t, \cdot )) \colon \E(\bm{v} - \nabla \bm{u}\widetilde{\bm{B}}_{\bm{\xi}}(t, \cdot )) \dx. 
	\end{equation*}
	Moreover, similar to \eqref{eq:potential_B}, we can find some $w_{\bm{B}} \in L^2(0, T; H^1_{(0)}(\Omega))$, which satisfies 
	\begin{equation*}
		- \int_0^{T_*} \int_\Omega \nabla w_{\bm{B}} \cdot \nabla \eta \dx \dt 
		= \int_0^{T_*} \int_\Omega \chi \, \div (\widetilde{\bm{B}}_{\bm{\xi}}(t, \cdot ) \eta) \dx \dt 
	\end{equation*}
	for all $\eta \in C^0_c(0, T_*; C^1(\overline{\Omega}))$, and conclude that $t \mapsto \norm{\widetilde{\bm{B}}_{\bm{\xi}}(t, \cdot )\cdot  \nabla \chi}_{H^{-1}_{(0)}}^2 = \norm{\nabla w_{\bm{B}} (t, \cdot)}_{L^2}^2$ is measurable. \\ 
	Lastly, we note that for any fixed $t \in (0, T_*)$ the term 
	\begin{equation*}
		\delta \F(\mu, \chi, \bm{u} )[\widetilde{\bm{B}}_{\bm{\xi}}(t, \cdot ), \bm{v}] 
		- \frac{1}{2} \norm{\widetilde{\bm{B}}_{\bm{\xi}}(t, \cdot )\cdot  \nabla \chi}_{H^{-1}_{(0)}}^2 - \frac{1}{2}\norm{\bm{v} - \nabla \bm{u}\widetilde{\bm{B}}_{\bm{\xi}}(t, \cdot )}_{\bm{H}^1_{\chi(t)}}^2
	\end{equation*}
	depends continuously on $\widetilde{\bm{B}}_{\bm{\xi}}$ and $\bm{v}$.\\ 
	Separability of the spaces $\bm{C}^2(\overline{\Omega})$ and $\bm{H}^2(\Omega) \cap \bm{H}^1_\perp(\Omega) \subset \bm{H}^2(\Omega)$ combined with the fact the pointwise supremum of a Carathéodory function over a countable set is measurable, cf.\  \cite[Prop.\ 2.7]{folland}, yields \eqref{eq:energy_dissipation_cor_sup}. 
\end{proof}

\begin{proof}[Proof of Lemma \ref{lem:relation_metric_slope}]
	Fix some $\bm{B} \in \mathcal{S}_\chi$ and $\bm{v}\in \bm{H}^1_{\perp}(\Omega) \cap \bm{H}^2(\Omega)$. Then, we want to consider
	\begin{align*}
		\limsup_{\tau \to 0}
		&\frac{  \big( \F(\mu, \chi, \bm{u}) - \F(\mu^\tau, \chi^\tau, (\bm{u} + \tau \bm{v})^\tau )   \big)_+ }{ \textnormal{d} \big( (\chi, \bm{u}),  (\chi^\tau, (\bm{u} + \tau \bm{v} )^\tau) \big)}
	\end{align*}
	We observe that for the first variation of our energy, it holds that cf.\ \cite[Step 4]{laux_chambolle}, \cite[Eq.\ (145)]{hensel_stinson}, 
	\begin{align*}
		&\lim_{\tau \to 0}   \tfrac{1}{\tau}  \big( \F(\mu,  \chi, \bm{u}) - \F( \mu^\tau, \chi^\tau, (\bm{u} + \tau \bm{v} )^\tau)   \big) 
		= \delta \F (\mu, \chi, \bm{u}) [\bm{B}, \bm{v}]. 
	\end{align*}
	For the explicit formulas and computations, see Section \ref{sec:gradient_flow}, Section~\ref{sec:euler_lagrange} and \eqref{eq:variation_varifold}. 
	Moreover, it was shown in \cite[Lem.\ 10]{hensel_stinson} that 
	\begin{equation*}
		\lim_{\tau \to 0}  \frac{\chi^\tau - \chi}{\tau} \to - \bm{B} \cdot \nabla \chi^h \quad \textnormal{in } H^{-1}_{(0)}(\Omega).
	\end{equation*}
	Lastly, we observe that 
	\begin{equation*}
		\lim_{\tau \to 0} \frac{ ( \bm{u} + \tau \bm{v})^\tau - \bm{u}}{\tau}   
		= \bm{v} - \nabla \bm{u} \bm{B}
		\quad \textnormal{in } \bm{H}^1(\Omega). 
	\end{equation*} 
	Along with the elementary inequality $\tfrac{1}{2} (\tfrac{a}{b})^2 \geq a - \tfrac{1}{2}b^2$, it follows that 
	\begin{align*}
		\limsup_{\tau \to 0}
		&\frac{  \big( \F(\mu, \chi, \bm{u}) - \F(\mu^\tau, \chi^\tau, (\bm{u} + \tau \bm{v})^\tau )   \big)_+ }{ \textnormal{d} \big( (\chi, \bm{u}),  (\chi^\tau, (\bm{u} + \tau \bm{v} )^\tau) \big)}
		= 
		\limsup_{\tau \to 0}
		\frac{\tfrac{1}{\tau}  \big( \F(\mu, \chi, \bm{u}) - \F(\mu^\tau, \chi^\tau, (\bm{u} + \tau \bm{v})^\tau )   \big)_+ }{\tfrac{1}{\tau} \textnormal{d} \big( (\chi, \bm{u}),  (\chi^\tau, (\bm{u} + \tau \bm{v} )^\tau) \big)}\\ 
		&\geq 
		\lim_{\tau \to 0} \frac{ \F(\mu, \chi, \bm{u}) - \F(\mu^\tau, \chi^\tau, (\bm{u} + \tau \bm{v})^\tau )}{\tau}
		-\lim_{\tau \to 0}  \frac{1}{2} \Bigg( \norm{\frac{ \chi^\tau - \chi}{\tau}}_{H^{-1}_{(0)}}^2  + \norm{\frac{  (\bm{u} + \tau \bm{v})^\tau- \bm{u} }{\tau}}^2_{\bm{H}^1_{\chi}}\Bigg) \\ 
		&= \delta \F (\mu, \chi, \bm{u}) [\bm{B}, \bm{v}] 
		- \frac{1}{2} \norm{\bm{B}\cdot  \nabla \chi}_{H^{-1}_{(0)}}^2 - \frac{1}{2}\norm{\bm{v} - \nabla \bm{u}\bm{B}}_{\bm{H}^1_{\chi}}^2. 
	\end{align*}
	Taking the supremum over $\bm{B}$ and $\bm{v}$ therefore yields the assertion. 
\end{proof}

\subsection{BV solutions} \label{sec:BC_solutions}
Instead of working with the weak time derivative $\pt \chi$ and a potential $w$, we now choose to only use a single potential $v$, for which we define the approximation $v^h \in H^1_{(0)}(\Omega)$ as the unique solution to
\begin{equation}\label{eq:potential_bv}
	\left\{ 
	\begin{alignedat}{2}
		- \Delta v^h (t) &= - \tfrac{\chi^h \big( (\lfloor \tfrac{t}{h} \rfloor +1)  h \big) -  \chi^h(\lfloor \tfrac{t}{h} \rfloor   h)  }{h}  \quad &&\textnormal{in } \Omega, \\ 
		\partial_{\bm{n}} v  &= 0 \quad &&\textnormal{on }\partial \Omega. 
	\end{alignedat}	
	\right. 
\end{equation}
In particular, we see that 
$-\pt \chi^h(t) = -\Delta v^h$.
Moreover, for all $\xi \in C^\infty([0, T] \times \overline{\Omega})$ with $\xi(T) = 0$, we obtain by integrating over $\Omega_{T_*}$
\begin{align*}
	\int_0^T \int_\Omega \nabla v^h \cdot \nabla \xi(t) \dx \dt 
	&= \int_0^T \int_\Omega \hat{\chi}^h(t) \partial_t \xi(t) \dx \dt
	+ \int_\Omega \hat{\chi}^h(0) \xi(0) \dx, 
	\numberthis\label{eq:pt_bv_appxo}
\end{align*}
where we used integration by parts on the intervals $[kh, (k+1)h]$ and telescoping for the boundary conditions. \\ 
From the dissipation inequality \eqref{eq:a_piori_trivial} and due to the identity \eqref{eq:potential_bv}, we obtain the estimate 
\begin{align*}
	\norm{ \abs{\nabla \chi^h}(\Omega) }_{L^\infty} 
	+ \norm{ {\bm{u}}^h }_{L^\infty(\bm{H}^2)}
	+ 	\norm{\pt \hat{\chi}^h}_{L^2(H^{-1}_{(0)})}
	+ \norm{ \nabla v }_{L^2(\bm{L}^2)}
	+  \norm{ \pt \hat{\bm{u}}^h}^2_{L^2(\bm{H}^1) } 
	\leq C,
	\numberthis
	\label{est:a_priori_mm_bv}
\end{align*}
where $C < \infty$ is a constant that does not depend on $h > 0$. Recall that we assumed $\cos \alpha = 0$ and therefore see no contribution from the capillary term.\\ 
Obviously, the validity of \eqref{el_elastic_approximate} is unchanged by this definition, but instead of \eqref{el_w_appoximate} we have 
\begin{align*}
	\delta_{\tilde{\chi}} &\F_{per} ({\chi}^h) [\bm{B}] \numberthis \label{eq:potential_approximate_bv} \\ 
	= &\int_\Omega {\chi}^h \  \div \big( v^h \bm{B} \big) \dx  
	+ \int_\Omega \C_{\nu}(\chi^h_{\lfloor \sfrac{t}{h} \rfloor })  \E( \pt \hat{\bm{u}}^h )
	\colon \nabla^2 {\bm{u}}^h \bm{B}  \dx 
	- \int_\Omega \tfrac{1}{2} \abs{\nabla^2{ \bm{u}}^h}^2\  \div \bm{B} \dx \\ 
	&- \int_\Omega \Big[ W({\chi}^h, \E({\bm{u}}^h)) \bm{I} -  ( \nabla {\bm{u} } )^T \big( \C_{\nu}(\chi^h_{\lfloor \sfrac{t}{h} \rfloor })   \E( \pt \hat{\bm{u}}^h ) +  W_{, \E} ({\chi}^h, \E({\bm{u}}^h)  \big)\Big] \colon \nabla \bm{B} \dx 
	\\ 
	&+\int_\Omega
	(\nabla^2 \bm u^h)_{ijk}\,(\nabla \bm u^h)_{ip}\,(\nabla^2 \bm B)_{pjk}
	\dx
	+\int_\Omega
	(\nabla^2 \bm u^h)_{ijk}
	\Big(
	(\nabla^2 \bm u^h)_{ipk}\,(\nabla \bm B)_{jp}
	+
	(\nabla^2 \bm u^h)_{ijp}\,(\nabla \bm B)_{kp}
	\Big)
	\dx
\end{align*}
for all $\bm{B} \in \mathcal{S}_{{\chi}^h}$. 
Here the first variation of the perimeter is given by, \cite{sturzenhecker, roeger, garcke_00}, 
\begin{equation}\label{eq:curvature_bv}
	\delta_{\tilde{\chi}} \F_{per} ({\chi}^h) [\bm{B}] 
	= 
	\int_\Omega \Big( \nabla \cdot \bm{B} - \tfrac{\nabla \chi^h}{\abs{\nabla \chi^h}} \cdot  \nabla \bm{B} \tfrac{\nabla \chi^h}{\abs{\nabla \chi^h}} \Big) \, d\abs{\nabla \chi^h} . 
\end{equation}
Following the proof of \eqref{eq:langrange_mult_sg}, the same identity holds even for all $\bm{B} \in \bm{C}^2(\overline{\Omega})$ with $\bm{B} \cdot \bm{n}_{\partial \Omega}$ if we replace $v^h$ by $\widetilde{v}^h = v^h + \lambda^h$, where 
\begin{align*}
	-\lambda^h &\int_\Omega {\chi}^h \nabla \cdot \bm{\xi} \dx \\
	=& -\int_\Omega  {\chi}^h \nabla \cdot (v^h \bm{\xi}) \dx 
	- \int_\Omega \C_{\nu}(\chi^h_{\lfloor \sfrac{t}{h} \rfloor })  \E( \pt \hat{\bm{u}}^h )
	\colon \nabla^2 {\bm{u}}^h \bm{\xi}  \dx 
	+ \int_\Omega \tfrac{1}{2} \abs{\nabla^2{ \bm{u}}^h}^2\  \div \bm{\xi} \dx \\ 
	&+ \int_\Omega \Big[ W({\chi}^h, \E({\bm{u}}^h)) \bm{I} -  ( \nabla {\bm{u} } )^T \big( \C_{\nu}(\chi^h_{\lfloor \sfrac{t}{h} \rfloor })   \E( \pt \hat{\bm{u}}^h )
	+  W_{, \E} ({\chi}^h, \E({\bm{u}}^h)  \big)\Big] \colon \nabla \bm{\xi} \dx 
	\\ 
	&-\int_\Omega
	(\nabla^2 \bm u^h)_{ijk}\,(\nabla \bm u^h)_{ip}\,(\nabla^2 \bm \xi)_{pjk}
	\dx
	-\int_\Omega
	(\nabla^2 \bm u^h)_{ijk}
	\Big(
	(\nabla^2 \bm u^h)_{ipk}\,(\nabla \bm \xi)_{jp}
	+
	(\nabla^2 \bm u^h)_{ijp}\,(\nabla \bm \xi)_{kp}
	\Big)
	\dx \\ 
	&-\delta_{\tilde{\chi}} \F_{per} ({\chi}^h) [\bm{\xi}] ,
\end{align*}
with $\bm{\xi} \in \bm{C}^{\infty}(\overline{\Omega})$ satisfying $\bm{\xi} \cdot \bm{n}_{\partial \Omega}  = 0$ and  
\begin{equation*}
	\norm{\bm{\xi}}_{C^2} \leq C(\Omega, m_0, \abs{\Gamma}), 
	\quad \textnormal{and} \quad 
	\int_\Omega \chi \nabla \cdot \bm{\xi} \dx  \geq C(m_0, \Omega). 
\end{equation*}
In particular, we obtain the estimate 
\begin{equation}\label{estimate:lambda_implicit_bv}
	\abs{\lambda^h} 
	\leq C\big(\Omega, m_0, \abs{\nabla \chi^h} (\Omega)\big)
	\Big(\norm{\nabla v^h}_{L^1} 
	+ \norm{\pt \hat{\bm{u}}^h}_{\bm{H}^1}   \norm{{\bm{u}}^h}_{\bm{H}^2}+ 
	\norm{{\bm{u}}^h}_{\bm{H}^2}^2 
	+ \abs{\nabla \chi^h} (\Omega)
	\Big), 
\end{equation}
and $\lambda^h$ is uniformly bounded in $L^2(0, T_*; L^2(\Omega))$ due to \eqref{est:a_priori_mm_bv}. \\
Observe that the relevant compactness results, cf.\ \eqref{eq:chi_conv_strong}, \eqref{eq:u_conv_strong}, from prior sections are still applicable, i.e.,
\begin{alignat*}{2}
	\chi^h, \hat{\chi}^h &\to \chi \quad &&\textnormal{in } L^2(0, T_{*}; L^2(\Omega)),
	\\
	\nabla\chi^h
	&\xrightharpoonup{*}\nabla\chi
	\quad &&\textnormal{in } L^\infty(0,T_*; \bm{M}(\Omega)), \\ 
	v^hh + \lambda^h&\to v \coloneqq \tilde{v} + \lambda \qquad && \textnormal{in } L^2(0, T_*, H^1(\Omega)), \\ 
	\bm{u}^h, \hat{\bm{u}}^h &\to \bm{u} \quad 
	&&\textnormal{in } 
	L^2(0, T_*; \bm{H}^1_{\perp}(\Omega)), \\ 
	{\bm{u}}^h &\to \bm{u} \quad &&\textnormal{in } L^2(0, T_*; \bm{H}^2(\Omega)), \\ 
	\pt \hat{\bm{u}}^h &\xrightharpoonup{*}  \pt \bm{u} \quad &&\textnormal{in } L^2(0, T_*; \bm{H}^1_{\perp}(\Omega)), 
\end{alignat*}
where the strong convergence of $\bm{u}$ in $L^2(0, T; \bm{H}^2(\Omega))$ follows analogously to the arguments in \eqref{eq:u_conv_H2}. 
Passing to the limit in \eqref{eq:pt_bv_appxo}, \eqref{el_elastic_approximate} and the right-hand side of \eqref{eq:potential_approximate_bv} is immediately possible without any additional compactness, but \eqref{eq:curvature_bv} requires some extra work.\\ 
Exploiting the strong convergence of $\chi^h(t)$ in $L^2(\Omega)$ a.e.\  in $(0, T_*)$ and the uniform bound on $\abs{\nabla \chi^h(t)} (\Omega)$, we find that 
\begin{equation}\label{conv:nabla_chi_poitwise}
	\nabla \chi^h(t) \xrightharpoonup{*} \nabla \chi (t) \quad  \textnormal{in } \bm{M}(\Omega) \ \textnormal{a.e.\ in } (0, T_*). 
\end{equation}
\par 
\medskip \noindent
Starting from here, we will work under the additional assumption known as \textit{energy convergence hypothesis}:
\begin{equation*}
	\underset{h \searrow 0}{\textnormal{liminf }} \int_0^{T_*} \int_\Omega d\abs{\nabla \chi^h}(t)  \dt 
	\leq 
	\int_0^{T_*}\int_\Omega d\abs{\nabla \chi}(t)  \dt. 
\end{equation*}
Weak lower semicontinuity, Fatou's lemma \cite[Thm.\ 1.20]{ambrosio2000functions}, and the convergence \eqref{conv:nabla_chi_poitwise} then immediately imply 
\begin{equation}\label{eq:perimeter_no_loss}
	\lim_{h \searrow 0} \int_\Omega d\abs{\nabla \chi^h(t)} \to  
	\int_\Omega d\abs{\nabla \chi(t)} 
	\quad \textnormal{a.e.\ in } (0, T_*), 
\end{equation} 
i.e., there is no loss of perimeter in the limit. \\  
\par 
\medskip \noindent
\textit{Claim:} For almost all $t \in (0, T_*)$ it holds that 
\begin{equation}\label{eq:nabla_chi_mass_measure}
	\abs{\nabla \chi^h(t)} \xrightharpoonup{*} \abs{\nabla \chi(t)} \quad \textnormal{in } \bm{M}(\Omega). 
\end{equation}
\noindent 
\textit{Proof of claim:} Following Garcke and Sturzenhecker \cite{sturzenhecker_garcke}, one obtains from \eqref{eq:perimeter_no_loss} and lower semincontinuity of $\abs{\nabla \chi^h}$ for all compact subsets $K \subset \Omega$ that 
\begin{align*}
	\underset{h \searrow 0}{\textnormal{limsup }} \int_{\Omega \cap K} d\abs{\nabla \chi^h(t) } 
	&= \underset{h \searrow 0}{\textnormal{limsup }} \Big( \int_{\Omega} d\abs{\nabla \chi^h(t) }  - \int_{ K} d\abs{\nabla \chi^h(t) }  \Big)\\ 
	&\leq  \int_{\Omega} d\abs{\nabla \chi(t) }  - \int_{ K} d\abs{\nabla \chi(t) }  
	= \int_{\Omega \cap K} d\abs{\nabla \chi(t) } . 
\end{align*}
Since weak convergence of measures in characterized by lower semincontinuity on open sets and upper semincontinuity on compact sets, see \cite[Thm.\ 1.40]{evans2015measure}, this entails that
\begin{equation*}
	\int_\Omega \xi \, d\abs{\nabla \chi^h(t)} \to \int_\Omega \xi \, d\abs{\nabla \chi(t)}
\end{equation*}
for all $\xi \in C^0(\Omega)$ and almost every $t \in (0, T_*)$. 
\hfill $\diamondsuit$
\\ 
\par \medskip \noindent
Using this, we can find approximations of the normal by an argument similar to the proof of Reshetnyak's Lemma \cite{reshetnyak1968weak, Bronsard_Garcke}. 
\par 
\medskip \noindent
\textit{Claim:} For almost all $t \in (0, T_*)$ and all $\eps > 0$ there exists some $\bm{g}_\eps \in \bm{C}^\infty_c(\Omega)$ with $\norm{\bm{g}_\eps} \leq 1$ and $C > 0$ independent of $h >0$ such that 
\begin{equation}\label{eq:g_approx}
	\underset{h \searrow 0}{\textnormal{limsup }}  \int_\Omega \Big| \bm{g}_\eps - \tfrac{\nabla \chi_h(t)}{\abs{\nabla \chi^h(t)}} \Big| \, d\abs{\nabla \chi^h} 
	\leq C \eps^{\sfrac{1}{2}}. 
\end{equation}
\noindent 
\textit{Proof of claim:} For the sake of brevity, let us neglect the dependence of $\chi, \chi^h$ on $t \in (0, T_*)$. By definition of $BV$-functions there exists for all $\eps > 0$ some $\bm{g}_\eps$ with $\norm{\bm{g}_\eps} \leq 1$ such that 
\begin{equation*}
	\int_\Omega \bm{g}_{\eps} \cdot d\nabla \chi = - \int_\Omega  \chi \, \nabla \cdot \bm{g}_\eps \dx 
	\geq \int_\Omega d\abs{\nabla \chi} - \eps, 
\end{equation*}
and using the weak$^*$ convergence of $\nabla \chi^h$ , along with the fact that there is no loss of perimeter in the limit, it follows that 
\begin{equation*}
	\lim_{h \searrow 0} \Big( \int_\Omega d\abs{\nabla \chi^h} - \int_\Omega \bm{g}_\eps \cdot d \nabla \chi^h \Big)
	=  \int_\Omega d\abs{\nabla \chi} - \int_\Omega \bm{g}_\eps \cdot d \nabla \chi 
	\leq \eps . 
\end{equation*}
Hence, 
\begin{equation}\label{eq:estimate_app_normal}
	\lim_{h \searrow 0} \int_\Omega \big( 1 - \bm{g}_\eps \cdot  \tfrac{\nabla \chi^h}{\abs{\nabla \chi^h}}  \big) \, d\abs{\nabla \chi^h}
	= \lim_{h \searrow 0} \Big( \int_\Omega d\abs{\nabla \chi^h} - \int_\Omega \bm{g}_\eps \cdot d \nabla \chi^h \Big) 
	\leq \eps . 
\end{equation}
On the reduced boundary $\partial^*(\{\chi \equiv 1\})$ the fact that $\norm{\bm{g}_\eps} \leq 1$ implies 
\begin{equation*}
	\Big| \bm{g}_\eps - \tfrac{\nabla \chi_h(t)}{\abs{\nabla \chi^h(t)}} \Big|^2
	=
	\bm{g}_\eps^2 - 2 \bm{g}_\eps \cdot \tfrac{\nabla \chi_h(t)}{\abs{\nabla \chi^h(t)}} + 1 
	\leq 2 \Big( 1-  \bm{g}_\eps \cdot  \tfrac{\nabla \chi_h(t)}{\abs{\nabla \chi^h(t)}}  \Big). 
\end{equation*}
Employing Hölder's inequality, it follows 
\begin{align*}
	\int_\Omega \Big| \bm{g}_\eps - \tfrac{\nabla \chi_h(t)}{\abs{\nabla \chi^h(t)}} \Big| \, d\abs{\nabla \chi^h}
	&\leq 
	\Big(  \int_\Omega \Big| \bm{g}_\eps - \tfrac{\nabla \chi_h(t)}{\abs{\nabla \chi^h(t)}} \Big|^2 \, d\abs{\nabla \chi^h} \Big)^{\sfrac{1}{2}} 
	\Big( \int_\Omega  d\abs{\nabla \chi^h} \Big)^{\sfrac{1}{2}}\\ 
	&\leq 
	C \Big(  \int_\Omega \big( 1 - \bm{g}_\eps \cdot \tfrac{\nabla \chi^h}{\abs{\nabla \chi^h}}  \big) \, d\abs{\nabla \chi^h}\Big)^{\sfrac{1}{2}}  . 
\end{align*}
The assertion thus follows from \eqref{eq:estimate_app_normal}. 
\hfill $\diamondsuit$
\par 
\medskip 
Finally, we are in the position to conclude convergence for the $BV$-formulation of the mean curvature, i.e., the first variation of the perimeter \eqref{eq:curvature_bv}. In fact, we have 
{\allowdisplaybreaks
	\begin{align*}
		&\Big| \int_\Omega \tfrac{\nabla \chi^h}{\abs{\nabla \chi^h}} \cdot  \nabla \bm{B} \tfrac{\nabla \chi^h}{\abs{\nabla \chi^h}}  \, d\abs{\nabla \chi^h} 
		-
		\int_\Omega  \tfrac{\nabla \chi}{\abs{\nabla \chi}} \cdot  \nabla \bm{B} \tfrac{\nabla \chi}{\abs{\nabla \chi}}  \, d\abs{\nabla \chi} \Big| \\ 
		&\quad= \begin{aligned}[t]
			\Big|
			&\int_\Omega \big(  \tfrac{\nabla \chi^h}{\abs{\nabla \chi^h}} - \bm{g}_\eps \big) \cdot \nabla \bm{B} \big(  \tfrac{\nabla \chi^h}{\abs{\nabla \chi^h}} + \bm{g}_\eps \big)\, d  \abs{\nabla \chi^h}
			+ \int_\Omega \bm{g}_\eps \cdot \nabla \bm{B} \bm{g}_\eps \, d\big( \abs{\nabla \chi^h} - \abs{\nabla \chi} \big)
			\\ 
			&  
			+ \int_\Omega \bm{g}_\eps \cdot \nabla \bm{B}  \tfrac{\nabla \chi^h}{\abs{\nabla \chi^h}} -  \tfrac{\nabla \chi^h}{\abs{\nabla \chi^h}} \cdot \nabla \bm{B} \bm{g}_\eps \, d \abs{\nabla \chi^h}
			-  \int_\Omega \bm{g}_\eps \cdot \nabla \bm{B}  \tfrac{\nabla \chi}{\abs{\nabla \chi}} -  \tfrac{\nabla \chi}{\abs{\nabla \chi}} \cdot \nabla \bm{B} \bm{g}_\eps \, d \abs{\nabla \chi}\\ 
			&+\int_\Omega \big( \bm{g}_\eps - \tfrac{\nabla \chi}{\abs{\nabla \chi}}\big) \cdot \nabla \bm{B} \big(\tfrac{\nabla \chi}{\abs{\nabla \chi}}  + \bm{g}_\eps \big) \, d\abs{\nabla \chi}
			\Big|
		\end{aligned}\\
		& \quad \leq \begin{aligned}[t]
			C &\int_\Omega \big|  \bm{g}_\eps -  \tfrac{\nabla \chi^h}{\abs{\nabla \chi^h}} \big| \, d  \abs{\nabla \chi^h}
			+\int_\Omega \bm{g}_\eps \cdot \nabla \bm{B} \bm{g}_\eps \, d\big( \abs{\nabla \chi^h} - \abs{\nabla \chi} \big) 
			+ C \int_\Omega \big| \bm{g}_\eps - \tfrac{\nabla \chi}{\abs{\nabla \chi}}\big|  \, d\abs{\nabla \chi} \\
			&+ \Big|
			\int_\Omega \bm{g}_\eps \cdot \nabla \bm{B}  \tfrac{\nabla \chi^h}{\abs{\nabla \chi^h}}  \, d \abs{\nabla \chi^h}
			-  \int_\Omega \bm{g}_\eps \cdot \nabla \bm{B}  \tfrac{\nabla \chi}{\abs{\nabla \chi}} \, d \abs{\nabla \chi}
			\Big| \\
			&+ \Big| 
			\int_\Omega  \tfrac{\nabla \chi}{\abs{\nabla \chi}} \cdot \nabla \bm{B} \bm{g}_\eps \, d \abs{\nabla \chi}
			- \int_\Omega  \tfrac{\nabla \chi^h}{\abs{\nabla \chi^h}} \cdot \nabla  \bm{B} \bm{g}_\eps \, d \abs{\nabla \chi^h}
			\Big|
		\end{aligned}
	\end{align*}
}where we applied the Cauchy--Schwarz inequality in the second step. Due the \eqref{conv:nabla_chi_poitwise} and  \eqref{eq:nabla_chi_mass_measure}, combined with the approximating property of $\bm{g}_\eps$ and \eqref{eq:g_approx}, we see that the right-hand side vanishes for $\eps, h \searrow 0$. Thus, we may also pass to the limit in \eqref{eq:curvature_bv}. 
\par 
\medskip 
Lastly, we remark that the asserted energy dissipation inequality is a simple consequence of \eqref{eq:a_piori_trivial} combined with the arguments in the varifold case.

	\par 
	\medskip 

	\paragraph{\textbf{Acknowledgments}} 
	The authors thank Tim Laux and Kerrek Stinson for helpful discussions. 
	The last author is supported by the Graduiertenkolleg 2339 IntComSin of the Deutsche Forschungsgemeinschaft (DFG, German Research Foundation) – Project-ID 321821685. The support is gratefully acknowledged. 
	\par 
	\medskip 
	\paragraph{\textbf{Conflict of interests and data availability statement}} There is no conflict of interests.
	There is no associated data to the manuscript.

	\bibliographystyle{siam}
	\small
	\bibliography{refs}
	
	\appendix

\end{document}